\documentclass[11pt]{article}

\usepackage{amsmath,amssymb,amsthm}
\usepackage{bbm}
\usepackage{booktabs}

\newcommand{\R}{ \mathbb{R}}          
\newcommand{\Z}{\mathbb{Z}}          
\def\V#1{{\boldsymbol{#1}}}           
\def\M#1{{\bf{#1}}}                   
\newcommand{\Tr}{\mathsf{T}}          

\def\Spc#1{{\mathcal{#1}}}            
\def\Op#1{{\mathrm{#1}}}              
\def\ee{\mathrm{e}}                   
\def\jj{\mathrm{i}}                   
\newcommand{\dint}{{\rm d}}           
\def\Indic{\mathbbm{1}} 

\newcommand{\embedD}{\overset{d}{\hookrightarrow}} 
\newcommand{\embedC}{\overset{}{\hookrightarrow}} 
\newcommand{\eqdef}{:=}
\newcommand{\Identity}{\mathrm{Id}}

\theoremstyle{plain}
\newtheorem{theorem}{Theorem}[section]
\newtheorem{lemma}[theorem]{Lemma}
\newtheorem{corollary}[theorem]{Corollary}
\newtheorem{proposition}[theorem]{Proposition}
\theoremstyle{definition}
\newtheorem{definition}[theorem]{Definition}
\newtheorem{example}[theorem]{Example}
\theoremstyle{remark}
\newtheorem{remark}[theorem]{Remark}

\usepackage[hidelinks]{hyperref}
\usepackage{clchanges}   

\begin{document}

\title{Splines and Operators:\\ Part I---Admissibility and Native Spaces}
\author{Michael Unser\thanks{Biomedical Imaging Group, \'Ecole polytechnique
f\'ed\'erale de Lausanne (EPFL), CH-1015 Lausanne, Switzerland
(michael.unser@epfl.ch).}}
\date{}
\maketitle

\begin{abstract}
The variational theory of splines requires a native space: a Hilbert space over which
the regularization energy $\|\Op L f\|_{L_2}^2$ is well defined and the minimum-energy
interpolation problem admits a spline solution. In this paper, we characterize the operators
$\Op L$ that enable the construction of such spaces. Our framework is based on the completion
of a space $\Spc S$ of test functions in a controlled Gelfand-triple configuration
$\Spc S\embedD\Spc H\embedD\Spc S'$, which we call a Hilbert--Gelfand space. This departs from
the classical reproducing-kernel techniques. Our construction is substantially more general,
as it supports arbitrary linear operators and singular behaviors, and it recovers the
reproducing-kernel setting as a particular case. It is also complete and canonical: the
admissibility conditions on $\Op L$ are necessary and sufficient, and $\Op L$ determines its
native space as the unique Hilbert--Gelfand space on which it extends to a bounded surjection
onto $L_2$. We first show that Hilbert--Gelfand spaces are in
one-to-one correspondence with compatible Riesz pairs $(\Op A,\Op A^{-1})$ of positive-definite
operators $\Spc S\to\Spc S'$. We then extend this correspondence to conditional Riesz pairs,
where $\Op A$ is only conditionally positive-definite, which is the scenario that is relevant
to splines. The theory identifies the \emph{attainable null space} of the energy as the unique
admissible conditioning, realizes the abstract completions as concrete spaces of
distributions, and provides a verifiable admissibility condition in terms of a dual energy
functional. The framework is illustrated with the fractional Laplacian, which covers
polyharmonic and thin-plate splines, and with ordinary differential operators, which cover
exponential splines.
\end{abstract}

\pagebreak
\tableofcontents
\bigskip

\section{Introduction}
\label{Sec:Introduction}

Splines admit an elegant variational characterization
\cite{deBoor1966,Jerome1969,Prenter:1975,Schumaker2007}. Let us consider a finite set of data points $(x_m,f_m)$, $m=1,\dots,M$, and a linear operator
$\Op L$ whose role is to favor certain classes of data fits. The associated
interpolation problem is to find a function $f$ that minimizes the regularization energy
$|\Op L f|_{L_2}^2$
subject to the interpolation constraints $f(x_m)=f_m$. Its solution is a spline associated with the normal operator $\Op B=\Op L^\ast\Op L$.The choice
$\Op L=\Op D$ (the derivative) yields the piecewise-linear splines, while $\Op L=\Op D^m$
produces the classical polynomial splines of degree $2m-1$ \cite{Schoenberg1964}; the multidimensional counterpart with
the thin-plate energy $\|\Delta f\|_{L_2}^2$ leads to the thin-plate splines of Duchon \cite{Duchon1977}. For the
problem to be well posed, the minimization must be carried over a suitable Hilbert space
$\Spc H_{\Op L}$---the \emph{native space} of the energy \cite{Schaback1999,Schaback2000}---whose
specification turns out to be
the crux of the theory.

The goal of this paper is to provide a general and systematic operator-based mechanism for the
specification of the native spaces of splines. We impose two requirements on such a space. The
first
is that the null space of the regularization energy be finite-dimensional, which is
what confers to splines their characteristic structure and what makes the interpolation
problem solvable with finitely many parameters. The second
is that the space be rich
enough to enable the approximation of a broad class of functions; we secure this property by
requiring that the native space contain a nuclear space $\Spc S$ of test functions as a dense
subspace.

The prototypical illustration is the derivative operator $\Op L=\Op D$ on the real line, for
which the minimum-energy interpolation problem
\begin{align}
\label{Eq:Splinelinear}
s_{\mathrm{spline}}=\arg \min_{f \in \Spc H_{\Op D}(\R)} \|\Op D f \|_{L_2}
\ \mbox{ s.t. } f(x_m)=f_m,\ \ m=1, \dots, M
\end{align}
is known to be solved by a piecewise-linear spline; that is, by a curve that is composed of
linear segments that connect the points $(x_m,f_m)$, with constant (minimum-energy)
extrapolation beyond the end points. The generalization of this scheme is to substitute some
other admissible operator $\Op L$ for $\Op D$---for instance, $\Op L=\Op D^n$, which yields 
polynomial spline interpolants. This results in the concept of $\Op L$-splines where a given brand of splines
is attached to some operator $\Op L$ \cite{Schultz1967,Jerome1969}.  A specificity of splines is that the
``interesting'' operators all have a nontrivial finite-dimensional null space, and that the
resolution of \eqref{Eq:Splinelinear} calls for a generalized inverse
$\Op A=(\Op L^\ast\Op L)^{-1}:\Spc S\to\Spc S'$, which is typically obtained by Green's
function techniques. Concretely, this translates into a configuration where $\Op A$ is
$\V p$-conditionally positive-definite, with $\V p$ spanning the effective null space of
$\Op L$.

The classical construction of native spaces relies on reproducing-kernel techniques \cite{Aronszajn1950,deBoor1966,Mosamam2010} and on the
theory of (conditionally) positive-definite functions \cite{Micchelli1986,Madych1990c,Buhmann2003,Wendland2005,Ferreira2013,Iske2018}. The operators of interest, however, all
share a structural difficulty: their null space is nontrivial, so that the energy
$\|\Op Lf\|_{L_2}$ is only a seminorm. The completion of the space $\Spc S$ of test functions
with respect to this seminorm is then an abstract space of equivalence classes, which cannot
be identified with a space of distributions. The simplest instance of this obstruction is the
homogeneous Sobolev space
$\dot H^{1}(\R)=\overline{(\Spc S,\|\Op D\cdot\|_{L_2})}$, whose elements are functions modulo
constants \cite[Ch.~1]{Bahouri2011}.

In this work, we develop an operator-centric alternative that resolves this difficulty with
the help of completion techniques in a controlled Gelfand-triple configuration
$\Spc S\embedD\Spc H\embedD\Spc S'$, which we refer to as a Hilbert--Gelfand space. Our
program has two stages. We first establish a one-to-one correspondence between
Hilbert--Gelfand spaces and compatible Riesz pairs $(\Op A,\Op A^{-1})$ of positive-definite
operators $\Spc S\to\Spc S'$ (Theorem~\ref{Theo:HilbertPositiveDef} and
Corollary~\ref{Cor:AlgebraicInverse}). We then extend the correspondence to
$\V p$-conditional Riesz pairs (Definition~\ref{Def:Intrinsic}), where $\Op A$ is only
conditionally positive-definite; this is the scenario of splines, with $\V p$ a basis of the
null space of the energy. The abstract completions are realized as concrete spaces of
distributions (Proposition~\ref{Prop:Realization}), the realization is made explicit through
biorthogonal systems and oblique projectors (Theorem~\ref{Thm:Concretization}), and a
finite-rank augmentation converts the conditional pair into an unconditioned dual pair of
Hilbert--Gelfand spaces (Theorem~\ref{Thm:Augmentation}), which is the sought-after
native-space structure. Finally, we show that the construction is rigid on both sides: the
conditioning is dictated by the energy (Theorem~\ref{Thm:AttainableNull}) and the resulting space is
dictated by the operator (Corollary~\ref{Cor:PhiIndependence} and
Proposition~\ref{Prop:NativeUniqueness}). The only freedom that remains, the choice of the
biorthogonal family $\V\phi$, affects the complement of the null space but neither the set nor its
topology. For the example of the derivative, this yields a concrete space of functions
$\Spc H_{\Op D}$ whose quotient by the constants is isometrically isomorphic to the abstract factor
space $\dot H^1(\R)$, namely $\Spc H_{\Op D}/\R\cong\dot H^1(\R)$, which is an instance of the general
quotient isometry of Corollary~\ref{Cor:PhiIndependence}.

The proposed setting of a nuclear Gelfand pair $(\Spc S,\Spc S')$ is remarkably general. It
handles functions and objects that are defined over a variety of domains: the real line with
$\Spc S=\Spc S(\R)$, or $\Spc S=\Spc S(\R^d)$ for multidimensional functions; the
(multi-)integers with $\Spc S=\Spc S(\Z^d)$ for discrete signals; the torus with
$\Spc S=\Spc S(\mathbb{T}^d)$ for periodic functions; and the (hyper)sphere with
$\Spc S=\Spc S(\mathbb{S}^{d-1})$. It also accommodates the finite-dimensional setting
$(\Spc S,\Spc S')=(\R^N,\R^N)$, which is the familiar framework of linear algebra.

Our formulation is a substantial extension of the reproducing-kernel methods. Its scope is of
the same nature as the move from ordinary functions to generalized functions, and its power
stems from the underlying nuclear Gelfand pair, which supports arbitrary linear operators and
singular behaviors---for instance, (fractional) derivatives in the sense of distributions. The
reproducing-kernel constructions are recovered as the particular case where the Schwartz
kernel of the inverse Riesz map is a classical kernel; that is, a continuous bivariate
function rather than a bivariate distribution. A very basic example that falls outside the
standard reproducing-kernel formulation, but that is covered by the present one, is
$L_2=\overline{(\Spc S,\|\cdot\|_{L_2})}$, the space of finite-energy functions. Our
framework also connects directly with regularization theory, through the specification of the
spline energy, and it makes the link with the Sobolev spaces explicit,
including with the subfamily of the homogeneous ones.

The conceptual novelty is the notion of \emph{attainable null space} of the regularization
energy (Definition~\ref{Def:AttainableNull}): the space of drifts that the energy annihilates
in the limit, which is generally much smaller than the distributional null space of the operator.
We show that it is the one and only admissible conditioning
(Theorem~\ref{Thm:AttainableNull}) and, dually, that the native space is the one and only
Hilbert--Gelfand space that carries the energy as a bounded surjection onto $L_2$
(Proposition~\ref{Prop:NativeUniqueness}), which is what makes the notation $\Spc H_{\Op L}$
legitimate. The first statement ties our construction to the classical notion of
conditional positive-definiteness that is central to approximation theory and to the theory of
radial basis functions \cite{Madych1990c,Barthelme2023}. The twist is that the conditioning is flexible here: it is not
restricted to polynomials, but is dictated by the operator itself.

The central question then becomes one of \emph{admissibility}: for which operators
$\Op L:\Spc S\to L_2$ does the normal operator $\Op B=\Op L^\ast\Op L$ participate in a
conditional Riesz pair? When $\Op B$ admits a two-sided Green inverse, the answer is
immediate: the pair exists if and only if the conditioning spans the attainable null space
(Corollary~\ref{Cor:GreenNormal}). In the general case, Theorem~\ref{Thm:NormalPair} answers
the question intrinsically, with three conditions (attainable annihilation of the null space,
Riesz representation of the conditioned test functions, and sharpness), and
Theorem~\ref{Thm:PracticalChecklist} converts them into computations on a dual energy
functional $Q$. For Fourier multipliers, all of the checks
reduce to the inspection of an explicit integral
(Proposition~\ref{Prop:MultiplierConstruction}), and the conditional inverse factorizes
through a Green-function inverse of $\Op L^\ast$ (Corollary~\ref{Cor:GreenFactor}), in exact
duality with the factorization of the normal operator.

Beyond the specification of native spaces, the operator-based formulation carries two further
benefits, which are the topics of two companion papers. It yields representer theorems for the
resolution of inverse problems that go beyond the traditional setting of reproducing-kernel
Hilbert spaces \cite{Unser2026c}, and it is in mathematical congruence with the theory of generalized
Gaussian stochastic processes \cite{Unser2026c}, in the spirit of the correspondence that is reviewed
in \cite{Kanagawa2018}.

The paper is organized as follows. In Section~\ref{Sec:HilbertGelfand}, we introduce nuclear
Gelfand pairs, Hilbert--Gelfand spaces, and compatible Riesz pairs, and we state their
equivalence in Theorem \ref{Theo:HilbertPositiveDef}, whose proof is given in the appendix. 
In Section~\ref{Sec:CondRieszPairs}, we
define the attainable null space and the conditional Riesz pairs, prove the realization and
concretization results, and illustrate the theory with the operator $\Op D^\ast\Op D$, which is
linked to the piecewise-linear splines. In Section~\ref{Sec:AdmissibleL}, we characterize the
admissible regularization operators, first through the Green inverse of the normal operator
and then intrinsically, and, in Section~\ref{Sec:OperatorConstructions}, we
develop the practical checklist together with the Fourier-multiplier construction and the
factorization of the conditional inverse, and we show that the native space is determined by
$\Op L$ up to an equivalence of norms (Proposition~\ref{Prop:NativeUniqueness}). In
Section~\ref{Sec:Examples}, we illustrate the proposed formalism with the fractional Laplacian, which covers the polyharmonic and thin-plate
settings, and with the ordinary differential operators $P(\Op D)$, which cover the exponential
splines.

\section{Hilbert--Gelfand Spaces}
\label{Sec:HilbertGelfand}

\begin{definition}[Nuclear Gelfand pair]
\label{Def:NuclearGelfandPair}
The couple $(\Spc S,\Spc S')$ that is formed by a topological vector space $\Spc S$ and
its topological dual $\Spc S'$ is called a \emph{nuclear Gelfand pair} if
\begin{enumerate}
\item $\Spc S$ is nuclear and Montel
\cite{Grothendieck1955,Pietsch1972,Treves2006};
\item there exists a continuous and dense embedding $\iota:\Spc S\to\Spc S'$ that is
canonical, in the sense that
$\langle\iota\varphi,\psi\rangle=\langle\iota\psi,\varphi\rangle$ for all
$\varphi,\psi\in\Spc S$.
\end{enumerate}
Unless stated otherwise, $\Spc S'$ is endowed with its weak-$\ast$ topology. This
choice is harmless because the Montel property ensures that a weak-$\ast$ convergent
sequence in $\Spc S'$ converges strongly as well.
\end{definition}

While the term \emph{nuclear} may be unfamiliar, most readers have already encountered two canonical examples of nuclear Gelfand pairs: $(\R^N,\R^N)$, which is the setting of linear algebra with $\R^N=(\R^N)'$ under the Euclidean inner product and $\iota=\Identity$; and $\big(\Spc S(\R),\Spc S'(\R)\big)$, where $\Spc S(\R)$ is the Schwartz space of smooth, rapidly decaying test functions on $\R$, and $\Spc S'(\R)$ is the space of tempered distributions. This class of spaces was identified by Grothendieck, who, in his groundbreaking PhD thesis under the supervision of L.\ Schwartz, extended the theory of locally convex topological vector spaces with the aim of recovering the formal simplicity of linear algebra \cite{Grothendieck1955}. His central discovery was that Schwartz's kernel theorem is not unique to $\Spc S(\R)$, but rather the hallmark of an entire class of spaces on which continuous linear operators admit a matrix-like representation. He also identified the subclass possessing the Montel property, which is the infinite-dimensional analogue of the Heine--Borel theorem and which we have incorporated into Definition~\ref{Def:NuclearGelfandPair}. The practical consequence is that the analysis of such spaces proceeds as in distribution theory. The principal requirement is to ensure the continuity of the underlying operators and functionals, while the reward is a considerably greater level of generality than that afforded by the classic Lebesgue framework.

The primary pair for spline theory is $\big(\Spc S(\R^d),\Spc S'(\R^d)\big)$, in which a test function
$\varphi\in\Spc S(\R^d)$ is identified with the tempered distribution
$\iota\varphi\in\Spc S'(\R^d)$ that acts on $\Spc S(\R^d)$ as
\begin{align}
\label{Eq:CanonicalEmbedding}
\iota\varphi: \psi \mapsto \langle \iota\varphi,\psi\rangle \eqdef \int_{\R^d} \varphi(\V x)\,\psi(\V x)\dint\V x,
\qquad \psi\in\Spc S(\R^d),
\end{align}
which yields the canonical embedding
$\iota\colon\Spc S(\R^d)\hookrightarrow\Spc S'(\R^d)$. 
For measurements confined to a
domain, this pair may be substituted by $\big(\Spc D(\Omega),\Spc D'(\Omega)\big)$ with
$\Omega\subseteq\R^d$ open. The spaces in both instances are nuclear and Montel
\cite{Treves2006}, and so are the variants that are mentioned in the introduction: the
Fourier-dual couple formed by the rapidly decreasing sequences $\Spc S(\Z^d)$ and the
smooth periodic functions $\Spc S(\mathbb{T}^d)=C^\infty(\mathbb{T}^d)$
\cite{Ruzhansky2010}; the space
$\Spc S(\mathbb{S}^{d-1})=C^\infty(\mathbb{S}^{d-1})$, which is characterized by the
faster-than-polynomial decay of the spherical-harmonic coefficients \cite{Seeley1966};
and, more generally, $C^\infty(M)$ for every compact $C^\infty$ manifold $M$
\cite{Treves2006}.

Given the duality between $\Spc S$ and $\Spc S'$, we use the symbol $\perp$ for the annihilator in
either direction \cite[Ch.~4]{Rudin1991}. For a subspace $\Spc N\subseteq\Spc S'$,
\begin{align*}
\Spc N^\perp\eqdef\{\varphi\in\Spc S:\ \langle\varphi,p\rangle=0\ \text{for all}\ p\in\Spc N\}
\ \subseteq\ \Spc S ,
\end{align*}
and symmetrically $E^\perp\subseteq\Spc S'$ for a subspace $E\subseteq\Spc S$. The convention is
unambiguous because the Montel property implies that $\Spc S$ is reflexive, so that $(\Spc S')'=\Spc S$. The bipolar theorem
\cite[Ch.~IV, \S1]{Schaefer-Wolff1999} (see also \cite[Ch.~3]{Voigt2020}) then reads
$(\Spc N^\perp)^\perp=\overline{\Spc N}^{\,w\ast}$, the weak-$\ast$ closure of $\Spc N$.

\begin{definition}[Hilbert--Gelfand space]
\label{Def:HilbertGelfand}
Let $(\Spc S,\Spc S')$ be a nuclear Gelfand pair. A Hilbert space $\Spc H$ is called an
\emph{$\Spc S$-Hilbert--Gelfand space} if
\begin{align*}
\Spc S \embedD \Spc H \embedD \Spc S',
\end{align*}
where the notation ``$\,\embedD\,$'' denotes a continuous and dense embedding.
\end{definition}

Such a configuration of spaces is commonly referred to as a \emph{rigged Hilbert space} or a \emph{Gelfand triple}; see, for example, \cite{Gelfand-Vilenkin1964,Treves2006}. Throughout this work, we use the term
\emph{Hilbert--Gelfand space} to emphasize the intermediate Hilbert space
$\Spc H$ together with the structural properties induced by its canonical
embeddings into the underlying nuclear Gelfand pair.

A major advantage of working with Hilbert--Gelfand spaces is the compatibility of the duality pairing for $(\Spc H', \Spc H)$ with that of $(\Spc S', \Spc S)$ (see Theorem \ref{Theo:HilbertPositiveDef} below). 
The most basic example of such a space is the finite-energy or baseline space $L_2$ induced by the duality product
for $(\Spc S,\Spc S')$ as
\begin{align}
L_2=\overline{(\Spc S,\|\cdot\|_{L_2})},\quad \|\varphi\|_{L_2}=\sqrt{\langle \varphi,\varphi\rangle},\quad \varphi \in \Spc S,
\end{align}
which is constructed by completion. By construction, $\Spc S$ is continuously and densely embedded in $L_2$, while its dual $L_2=(L_2)'$ is continuously and densely embedded in $\Spc S'$, so that $\Spc S \embedD L_2 \embedD \Spc S'$ (Hilbert--Gelfand property). 

The availability of a common (or, more precisely, mutually compatible) duality product
also facilitates the study of linear operators acting between Hilbert--Gelfand spaces.
Specifically, let $\Spc H_1$ and $\Spc H_2$ be two Hilbert--Gelfand spaces.
Then every bounded linear operator
$\Op T:\Spc H_1\to\Spc H_2$
admits a unique adjoint
$\Op T^\ast:\Spc H_2'\to\Spc H_1'$
characterized by
\begin{align}
\forall(g,f)\in\Spc H_2'\times\Spc H_1:\qquad
\langle \Op T^\ast g,f\rangle
=
\langle g,\Op T f\rangle,
\end{align}
with
$\|\Op T^\ast\|_{\Spc H_2'\to\Spc H_1'}
=
\|\Op T\|_{\Spc H_1\to\Spc H_2}$.
In practice, such operators often arise as the continuous extensions of primary operators
$\Op T:\Spc S\to\Spc S'$, with the corresponding adjoint
$\Op T^\ast:(\Spc S')'=\Spc S\to\Spc S'$.
Since the underlying Hilbert spaces are reflexive, the adjoint operation is involutive,
$(\Op T^\ast)^\ast=\Op T$,
and the generalized impulse response (Schwartz kernel) of $\Op T^\ast$ is simply the transpose of that of $\Op T$, exactly as in finite-dimensional linear algebra.

The key benefit of the framework is that the properties of such operators can be characterized entirely through their action on the common space of test functions; that is, through the restriction
$\Op T|_{\Spc S}:\Spc S\to\Spc S'$.

\begin{definition}[Positive-definite operator]
\label{Def:CondPositiveDef2}
Let $\Op A$ be a continuous operator $\Spc S \to \Spc S'$ and $\Spc N_{\V p}$ a
finite-dimensional subspace of $\Spc S'$ that is spanned by $\V p=(p_1,\dots,p_{N_0})$.
The operator $\Op A$ is said to be
\begin{itemize}
\item symmetric or self-adjoint if, for all $\varphi_1,\varphi_2 \in \Spc S$,
$\langle \Op A \varphi_1,\varphi_2\rangle=\langle \Op A \varphi_2,\varphi_1\rangle$;
\item positive-definite if, for any $\varphi \in \Spc S\backslash\{0\}$,
$\langle \Op A \varphi,\varphi\rangle > 0$;
\item positive-semi-definite if, for any $\varphi \in \Spc S$,
$\langle \Op A \varphi,\varphi\rangle \ge 0$;
\item $\V p$-conditionally positive-definite if
\begin{align}
\label{Eq:PositiveDefinite}
\langle \Op A \varphi,\varphi\rangle > 0
\mbox{ for all }\varphi \in \Spc S_{\V p} \backslash\{0\},
\end{align}
where $\Spc S_{\V p}\eqdef\Spc N_{\V p}^\perp
=\{\varphi \in \Spc S: \V p(\varphi)=\V 0\}$.
\end{itemize}
\end{definition}

An important observation is that any positive (semi-)definite operator $\Op A$ induces a
(semi-)norm on $\Spc S$ defined by
\begin{align}
\label{Eq:SemiInnerA}
|\varphi|_{\Op A}\eqdef\sqrt{\langle \Op A \varphi,\varphi \rangle}\ge 0.
\end{align}
The associated (semi-)inner product is recovered via the classical polarization identity
\begin{align}
\label{Eq:Polar}
\tfrac{1}{2}\left(|\varphi|^2_{\Op A}+|\psi|^2_{\Op A} - |\varphi-\psi|^2_{\Op A}\right)
=
\tfrac{1}{2}\big(\langle \Op A \varphi, \psi \rangle
+\langle \Op A \psi, \varphi \rangle\big),
\end{align}
whose right-hand side reduces to $\langle \Op A \varphi, \psi \rangle$ when $\Op A$ is
symmetric. For a general positive-(semi-)definite operator, \eqref{Eq:Polar} recovers the
symmetric part of the bilinear form
$(\varphi,\psi)\mapsto\langle\Op A\varphi,\psi\rangle$: only this part enters the (semi-)norm
\eqref{Eq:SemiInnerA}, which is why the symmetry of $\Op A$ will not be required as a hypothesis in the
sequel.

Of special interest are the positive-definite operators, as they induce Hilbert spaces and extend canonically to Riesz maps. Since every Riesz map is an isometric isomorphism, it possesses a unique inverse Riesz map. This motivates the introduction of the concept of compatible Riesz pairs.

\begin{definition}[Compatible Riesz pair]
\label{Def:CompatibleRieszPair}
Let $(\Spc S,\Spc S')$ be a nuclear Gelfand pair, and let $\Op A,\Op B:\Spc S\to\Spc S'$ be
continuous positive-definite operators. We form the Hilbert spaces
\begin{align*}
\Spc H=\overline{(\Spc S,|\cdot|_{\Op A})},
\qquad
\Spc H'=\overline{(\Spc S,|\cdot|_{\Op B})}.
\end{align*}
The pair $(\Op A,\Op B)$ is a \emph{compatible Riesz pair} for $(\Spc S,\Spc S')$ if $\Op A$ and
$\Op B$ extend to mutually inverse isometric isomorphisms
\begin{align*}
\Op A:\Spc H\to\Spc H',
\qquad
\Op B:\Spc H'\to\Spc H.
\end{align*}
When this holds, $\Op B$ acts as the inverse of $\Op A$, which justifies the notation
$\Op B=\Op A^{-1}$ that we use in the sequel.
\end{definition}

The above definition is motivated by its exact correspondence with the notion of a
Hilbert--Gelfand space. Indeed, Hilbert--Gelfand spaces naturally occur in dual pairs
$(\Spc H,\Spc H')$, with the associated Riesz maps providing mutually inverse
isometries. The following result, whose proof is given in the appendix, shows that these two
viewpoints are equivalent.

\begin{theorem}[Hilbert--Gelfand spaces and compatible Riesz pairs]
\label{Theo:HilbertPositiveDef}
Let $(\Spc S,\Spc S')$ be a nuclear Gelfand pair and let $\Spc H$ be a Hilbert space that contains
$\Spc S$ as a dense subspace. Then, the following statements are equivalent.
\begin{enumerate}
\item $\Spc H$ is an $\Spc S$-Hilbert--Gelfand space; i.e.,
$\Spc S \embedD \Spc H \embedD \Spc S'$.

\item The continuous dual $\Spc H'$ is an $\Spc S$-Hilbert--Gelfand space;
i.e., $\Spc S \embedD \Spc H' \embedD \Spc S'$.

\item The Hilbert space $\Spc H$ is continuously embedded in $\Spc S'$:
$\Spc H \embedC \Spc S'$.

\item There exists a compatible Riesz pair $(\Op A,\Op A^{-1})$ for
$(\Spc S,\Spc S')$ (see Definition~\ref{Def:CompatibleRieszPair})
such that
\begin{align*}
\|\varphi\|_{\Spc H}
=
|\varphi|_{\Op A}
=
\sqrt{\langle \Op A\varphi,\varphi\rangle},
\qquad
\varphi\in\Spc S.
\end{align*}
\end{enumerate}
Whenever these conditions hold, the pair in Item 4 is unique and is
identified as
\begin{align*}
\Op A
=
\Op J_{\Spc H}|_{\Spc S},
\qquad
\Op A^{-1}
=
\Op J_{\Spc H}^{-1}|_{\Spc S},
\end{align*}
where $\Op J_{\Spc H}:\Spc H\to\Spc H'$ is the underlying Riesz map;
moreover,
\begin{align*}
\|\varphi\|_{\Spc H'}
=
|\varphi|_{\Op A^{-1}},
\qquad
\varphi\in\Spc S.
\end{align*}
Conversely, every compatible Riesz pair $(\Op A,\Op A^{-1})$
generates a Hilbert--Gelfand dual pair
\begin{align*}
(\Spc H,\Spc H')
=
\Big(
\overline{(\Spc S,|\cdot|_{\Op A})},\;
\overline{(\Spc S,|\cdot|_{\Op A^{-1}})}
\Big),
\end{align*}
where the Riesz map $\Op J_{\Spc H}:\Spc H\to\Spc H'$ is the unique
continuous extension of $\Op A$.
Moreover, the canonical duality pairing for $(\Spc S,\Spc S')$ is
compatible with the pairing for $(\Spc H,\Spc H')$. Specifically, for
every $f\in\Spc H'\subset\Spc S'$ and every
$\varphi\in\Spc S\subset\Spc H$,
\begin{align}
\label{Eq:DualityProdH}
\langle f,\varphi\rangle
=
\langle f,\varphi\rangle_{\Spc H'\times\Spc H}.
\end{align}
By continuity, \eqref{Eq:DualityProdH} extends to every
$(f,g)\in\Spc H'\times\Spc H$, with
\begin{align}
\label{Eq:HDualityBound}
|\langle f,g\rangle|
\le
\|f\|_{\Spc H'}\,\|g\|_{\Spc H}.
\end{align}
\end{theorem}

Item 3 in Theorem~\ref{Theo:HilbertPositiveDef} is noteworthy, as it provides a simple criterion
for identifying a Hilbert--Gelfand space: a single continuous embedding suffices, while the
continuity of $\Spc S\embedD\Spc H$ and the density of $\Spc H$ in $\Spc S'$ come for free. We also
note that the positive-definiteness of $\Op A^{-1}$ in Item 4 can be replaced by the embedding
condition $\Spc H \embedC \Spc S'$ or by the denseness of $\Spc S$ in $\Spc H'$. The converse
equivalence is fundamental, for it yields a constructive mechanism for specifying
Hilbert--Gelfand spaces, starting from a compatible Riesz pair $(\Op A, \Op A^{-1})$.

While the positive-definiteness of $\Op A^{-1}$ in Theorem~\ref{Theo:HilbertPositiveDef}
may appear to be redundant, it is an essential requirement in the infinite-dimensional setting.
The operator $\Op B=\Op D^\ast\Op D$ on $\Spc S(\R)$ is positive-definite, but its inverse is only
conditionally positive-definite, so that the completion
$\overline{(\Spc S,|\cdot|_{\Op D^\ast\Op D})}=\dot H^1(\R)$ is a space of functions modulo constants
that cannot be identified with a subspace of $\Spc S'$; on the realization of such spaces, see
\cite{Monguzzi2020}. This obstruction, which is the one that the theory of the next sections
removes, is worked out in Example~\ref{Ex:LinearSpline}.

We conclude this section by highlighting a useful implication of
Theorem~\ref{Theo:HilbertPositiveDef}, which supports the algebraic notion of operator
invertibility on $\Spc S$. For a continuous positive-definite operator $\Op A:\Spc S\to\Spc S'$,
the \emph{Hilbertian extension} $\widetilde{\Op A}=\Op J_{\Spc H_\Op A}:\Spc H_{\Op A}\to\Spc H_{\Op A}'$
is the Riesz map that extends $\Op A$ to $\Spc H_{\Op A}=\overline{(\Spc S,|\cdot|_{\Op A})}\subset\Spc S'$.
Given a second such operator $\Op B$, the identity $\Op A\Op B=\Identity$ on $\Spc S$ is to be read
through these extensions: it means that $\Op B(\Spc S)\subseteq\Spc H_{\Op A}$ and
$\widetilde{\Op A}\,\Op B\varphi=\varphi$ for every $\varphi\in\Spc S$; that is,
$\widetilde{\Op A}|_{\Op B(\Spc S)}=(\Op B|_{\Spc S})^{-1}$. When $\Spc H_{\Op A}$ is a
Hilbert--Gelfand space, this relation between the extended operators restricts unambiguously to
the stated identity on $\Spc S$.

\begin{corollary}[Algebraic characterization of compatibility]
\label{Cor:AlgebraicInverse}
Let $\Op A,\Op B:\Spc S\to\Spc S'$ be continuous positive-definite operators with Hilbertian extensions $\widetilde{\Op A}$ and
$\widetilde{\Op B}$ on the complete spaces $\Spc H_{\Op A}$ and $\Spc H_{\Op B}$. Then, the following statements are equivalent.
\begin{enumerate}
\item[(i)] $(\Op A,\Op B)$ is a compatible Riesz pair for $(\Spc S,\Spc S')$
(Definition~\ref{Def:CompatibleRieszPair}).
\item[(ii)] $\Op A\Op B=\Op B\Op A=\Identity$ on $\Spc S$ \emph{through Hilbertian
extensions}; equivalently,
the Hilbertian extensions $\widetilde{\Op A}$ and $\widetilde{\Op B}$ are mutually inverse Riesz
maps.
\item[(iii)] The two norms are dual to one another on $\Spc S$:
\begin{align*}
|\psi|_{\Op B}=\sup_{\varphi\in\Spc S\backslash\{0\}}
\frac{|\langle\varphi,\psi\rangle|}{|\varphi|_{\Op A}}
\qquad\text{and}\qquad
|\varphi|_{\Op A}=\sup_{\psi\in\Spc S\backslash\{0\}}
\frac{|\langle\varphi,\psi\rangle|}{|\psi|_{\Op B}}.
\end{align*}
\end{enumerate}
When these hold, $\Spc H_{\Op A}$ and $\Spc H_{\Op B}=\Spc H_{\Op A}'$ are Hilbert--Gelfand spaces,
and $\Op B$ is the unique operator that inverts $\Op A$, which justifies the notation
$\Op B=\Op A^{-1}$.
\end{corollary}

In light of these equivalences, we shall hereafter drop the symbol $\ \widetilde{\phantom{ }}\ $ and represent the operator $\Op A: \Spc S \to \Spc S'$ and its Hilbertian extension $\widetilde{\Op A}: \Spc H_{\Op A} \to \Spc H_{\Op A}'$ by the same symbol $\Op A$.

\begin{proof}
We establish (ii)$\Leftrightarrow$(iii) and (i)$\Leftrightarrow$(iii).

\medskip
\emph{(ii)$\Rightarrow$(iii)}:
Let us assume that $\widetilde{\Op A}$ and $\widetilde{\Op B}$ are mutually inverse. Then
$\widetilde{\Op B}=\widetilde{\Op A}^{-1}$ identifies $\Spc H_{\Op B}$ isometrically with the dual
$\Spc H_{\Op A}'$, so that $\|\Op A\varphi\|_{\Spc H_{\Op B}}=|\varphi|_{\Op A}$ for every
$\varphi\in\Spc S$. Since $\widetilde{\Op B}$ extends $\Op B$, the Riesz identity reads
$\langle\varphi,\psi\rangle=\langle\Op A\varphi,\psi\rangle_{\Spc H_{\Op B}}$ for every
$\psi\in\Spc S$; as $\Spc S$ is dense in $\Spc H_{\Op B}$, the supremum over $\psi$ gives
$\sup_{\psi\neq0}|\langle\varphi,\psi\rangle|/|\psi|_{\Op B}=\|\Op A\varphi\|_{\Spc H_{\Op B}}
=|\varphi|_{\Op A}$, which is the second identity of (iii). The first follows from the symmetric
computation with $\widetilde{\Op A}\widetilde{\Op B}=\Identity$.

\medskip
\emph{(iii)$\Rightarrow$(ii)}:
The identities of (iii) are the relations \eqref{Eq:CompatBound}--\eqref{Eq:CompatSharp} in the proof of Theorem \ref{Theo:HilbertPositiveDef} with
$\Op A^{-1}=\Op B$. By Step 2 of the proof of Theorem~\ref{Theo:HilbertPositiveDef}, they realize
$\widetilde{\Op A}$ and $\widetilde{\Op B}$ as mutually inverse Riesz maps.

\medskip
\emph{(i)$\Leftrightarrow$(iii)}:
If $(\Op A,\Op B)$ is a compatible Riesz pair, then Step 1 of the proof of
Theorem~\ref{Theo:HilbertPositiveDef}, applied to $\Spc H=\Spc H_{\Op A}$, produces
\eqref{Eq:CompatBound}--\eqref{Eq:CompatSharp} in the appendix with $\Op A^{-1}=\Op B$, which is (iii).
Conversely, (iii) supplies these relations, which are the sole ingredients of Step 2 of that proof:
they realize $\Spc H_{\Op A}$ and $\Spc H_{\Op B}=\Spc H_{\Op A}'$ as mutually dual
Hilbert--Gelfand spaces and identify $\widetilde{\Op A},\widetilde{\Op B}$ as the unique continuous
extensions of $\Op A,\Op B$; hence $(\Op A,\Op B)$ is a compatible Riesz pair.

\medskip
\emph{Uniqueness}:
Let us suppose that $(\Op A,\Op B_1)$ also satisfies (iii). The first supremum gives
$|\cdot|_{\Op B_1}=|\cdot|_{\Op B}$ on $\Spc S$, so that $\Op B_1$ and $\Op B$ induce the same
Hilbert--Gelfand space $\Spc H_{\Op A}'$; the uniqueness part of
Theorem~\ref{Theo:HilbertPositiveDef} (Step 6) gives $\Op B_1=\Op B$.
\end{proof}

\begin{remark}[Automatic symmetry and uniqueness]
\label{Rem:AutomaticSymmetry}
The symmetry of $\Op A$ is not required as a hypothesis anywhere in this work, the reason being
that it is a consequence of the Riesz-pair property. Let $(\Op A,\Op B)$ be a compatible Riesz
pair, let $\Op J_{\Spc H}:\Spc H\to\Spc H'$ be the Riesz map of $\Spc H=\Spc H_{\Op A}$, and let
$\Op T:\Spc H\to\Spc H'$ be the isometric extension of $\Op A$ that
Definition~\ref{Def:CompatibleRieszPair} provides. Then,
$\Op U\eqdef\Op J_{\Spc H}^{-1}\Op T:\Spc H\to\Spc H$ is an isometric isomorphism, and the
compatibility \eqref{Eq:DualityProdH} of the pairings gives
$\langle\Op U u,u\rangle_{\Spc H}=\langle\Op A u,u\rangle=\|u\|_{\Spc H}^2$ for $u\in\Spc S$,
hence for every $u\in\Spc H$ by density. Consequently,
\begin{align*}
\|\Op U u-u\|_{\Spc H}^2
=\|\Op U u\|_{\Spc H}^2-2\langle\Op U u,u\rangle_{\Spc H}+\|u\|_{\Spc H}^2
=0,
\end{align*}
so that $\Op U=\Identity$ and $\Op T=\Op J_{\Spc H}$. This yields
$\langle\Op A\varphi,\psi\rangle=\langle\varphi,\psi\rangle_{\Spc H}$ for all
$\varphi,\psi\in\Spc S$, which is a symmetric expression. The operator $\Op A$ is therefore
symmetric, and it is the unique operator that induces $\Spc H$ (Step 6 of the proof of
Theorem~\ref{Theo:HilbertPositiveDef}). The same argument applies to the $\V p$-conditional Riesz
pairs of Section~\ref{Sec:CondRieszPairs}, where it yields the symmetry of $\Op A$ on
$\Spc S_{\V p}\times\Spc S_{\V p}$ and the uniqueness of the conditioned operator
$\Op A_{\V \phi}$ of Theorem~\ref{Thm:Concretization}, while $\Op A$ remains unconstrained off
$\Spc S_{\V p}$.
\end{remark}

\section{Conditional Riesz Pairs}
\label{Sec:CondRieszPairs}
The remainder of the paper focuses on the more challenging scenarios where one of the operators in the Riesz pair is only $\V p$-conditionally positive-definite.
\subsection{Abstract Formulation}
\label{Sec:AbstractFormulation}
For Theorem~\ref{Theo:HilbertPositiveDef} to be applicable, both $\Op A$ and its inverse must be positive-definite. We have already pointed out that this requirement fails for the operators arising in spline theory, which have a nontrivial finite-dimensional null space. As we shall see, these are precisely the operators whose inverse is only \emph{conditionally} positive-definite, as illustrated by $\Op B=\Op D^\ast\Op D$. In that case, the energy completion $\dot{\Spc H}_{\Op B}=\overline{(\Spc S,|\cdot|_{\Op B})}$ consists of equivalence classes of distributions. The obstruction that prevents $\dot{\Spc H}_{\Op B}$ from being an ordinary space of (generalized) functions---and that, once neutralized, will yield the desired Hilbert--Gelfand embedding---is the part of the null
space that the energy \emph{attains} in $\Spc S'$.

\begin{definition}[Attainable null space]
\label{Def:AttainableNull}
Let $q$ be a continuous seminorm on $\Spc S$. Its \emph{attainable null space} is
\begin{align*}
\Spc N_q\eqdef\big\{p\in\Spc S':\ p=\textstyle\lim_{k}\chi_k\ \text{in}\ \Spc S'
\ \text{for some}\ (\chi_k)\subset\Spc S\ \text{with}\ q(\chi_k)\to0\big\}.
\end{align*}
For the seminorm induced by a positive-semi-definite operator $\Op B:\Spc S\to\Spc S'$, we write
$\Spc N_{\Op B}\eqdef\Spc N_{|\cdot|_{\Op B}}$.
\end{definition}

The set $\Spc N_{\Op B}$ is a linear subspace of $\Spc S'$ that captures the obstruction exactly: the
completion $\overline{(\Spc S,|\cdot|_{\Op B})}$ embeds not into $\Spc S'$ but into the quotient
$\Spc S'/\Spc N_{\Op B}$, and the embedding is genuine precisely when $\Spc N_{\Op B}=\{0\}$. In effect,
the attainable null space $\Spc N_{\Op B}$
collects the \emph{drifts} (or trends) that the spline energy annihilates in the limit.
Let us stress that Definition~\ref{Def:AttainableNull} imposes no finiteness on $\Spc N_{\Op B}$,
which is specified as a set of weak-$\ast$ limits and could a priori be infinite-dimensional. 

Our proposed strategy is to neutralize these drifts by conditioning against them. Throughout,
$(\Spc S,\Spc S')$ is a nuclear Gelfand pair and $\V p=(p_1,\dots,p_{N_0})$ is a linearly independent
family in $\Spc S'$ spanning $\Spc N_{\V p}$, with

\begin{align*}
\Spc S_{\V p}\eqdef\Spc N_{\V p}^\perp=\{\varphi\in\Spc S:\V p(\varphi)=\V 0\}.
\end{align*}
The members of $\Spc S_{\V p}$ are the \emph{detrended} test functions---those annihilated by the moment
conditions $\V p(\varphi)=\V 0$. Conditioning thus has two dual faces: the drifts are \emph{quotiented
out} on the energy side, where they become classes of $\Spc S'/\Spc N_{\Op B}$, and \emph{detrended}
on the test-function side.
We take $\Op A,\Op B:\Spc S\to\Spc S'$ continuous, with $\Op B$ symmetric positive-definite and
$\Op A$ its conditional inverse, which is fixed solely by the requirement that it invert $\Op B$.
By completion, we form the two
Hilbert spaces
\begin{align}
\label{Eq:HA} 
\Spc H_{\Op A}\eqdef\overline{(\Spc S_{\V p},|\cdot|_{\Op A})}\\%
\label{Eq:HB} 
\dot{\Spc H}_{\Op B}\eqdef\overline{(\Spc S,|\cdot|_{\Op B})},
\end{align}
taken in the abstract sense (equivalence classes of Cauchy sequences), the second of which we
call the \emph{homogeneous space} induced by $\Op B$. 
Their realizations are asymmetric:
$\Spc H_{\Op A}$ is a genuine subspace of $\Spc S'$, whereas $\dot{\Spc H}_{\Op B}$ is a subspace of
the quotient $\Spc S'/\Spc N_{\Op B}$, with $\pi:\Spc S'\to\Spc S'/\Spc N_{\Op B}$ the corresponding quotient map (canonical
projection).

It remains to pair the two spaces. The correct configuration matches the conditioning to the
drifts, $\Spc N_{\V p}=\Spc N_{\Op B}$, and requires $\Op A$ and $\Op B$ to act as mutually inverse
isometries between $\Spc H_{\Op A}$ and $\dot{\Spc H}_{\Op B}$.

\begin{definition}[$\V p$-conditional Riesz pair]
\label{Def:Intrinsic}
Let $\Op A,\Op B:\Spc S\to\Spc S'$ be continuous, with $\Op B$ symmetric positive-definite and
$\Op A$ $\V p$-conditionally positive-definite, so that the Hilbert spaces $\Spc H_{\Op A}$ and
$\dot{\Spc H}_{\Op B}$ of \eqref{Eq:HA}  and \eqref{Eq:HB}  are well defined. 
Then, the pair $(\Op A,\Op B)$ is a \emph{$\V p$-conditional Riesz pair} for $(\Spc S,\Spc S')$ if
\begin{enumerate}
\item[(i)] $\Spc N_{\Op B}=\Spc N_{\V p}$; and
\item[(ii)] the map $\pi\,\Op A$ and the operator $\Op B$ extend to mutually inverse isometric
isomorphisms
\begin{align*}
\pi\,\Op A:\ \Spc H_{\Op A}\ \longrightarrow\ \dot{\Spc H}_{\Op B},
\qquad
\Op B:\ \dot{\Spc H}_{\Op B}\ \longrightarrow\ \Spc H_{\Op A}.
\end{align*}
\end{enumerate}
\end{definition}

The conditional positive-definiteness of $\Op A$ is required here only to ensure the well-posedness of $\Spc H_{\Op A}$.
It is not an independent assumption: Theorem~\ref{Thm:AttainableNull} below derives it from
Clause~(i) and the underlying invertibility requirements.

Let us spell out Clause~(ii) at the level of test functions. The Cauchy--Schwarz inequality
$|\langle\Op B\psi_k,\psi\rangle|\le|\psi_k|_{\Op B}|\psi|_{\Op B}$ shows that $|\psi_k|_{\Op B}\to0$
forces $\Op B\psi_k\to0$ in $\Spc S'$, so that the extension of $\Op B$ annihilates the null
classes of $\dot{\Spc H}_{\Op B}$ and is well defined on the equivalence classes
$[f]=f+\Spc N_{\V p}$. For $\varphi\in\Spc S_{\V p}$, the distribution $\Op A\varphi$ lies in
$\Spc S'$ and $\pi(\Op A\varphi)$ is a well-defined element of $\Spc S'/\Spc N_{\Op B}$. Clause~(ii)
then requires, for every $\varphi\in\Spc S_{\V p}$, a $|\cdot|_{\Op B}$-Cauchy sequence
$(\psi_k)\subset\Spc S$ that represents the abstract class $\pi\Op A\varphi$, in the sense that
$\pi\psi_k\to\pi(\Op A\varphi)$ in $\Spc S'/\Spc N_{\Op B}$, and whose $\Op B$-image returns
$\varphi$, in the sense that $\Op B\psi_k\to\varphi$ in $\Spc S'$; the isometry requirement reads
$\|\pi\Op A\varphi\|_{\dot{\Spc H}_{\Op B}}=|\varphi|_{\Op A}$, and the reverse composition
$\pi\Op A\circ\Op B=\Identity$ on $\dot{\Spc H}_{\Op B}$ is interpreted in the same sequential
fashion. The identification $\dot{\Spc H}_{\Op B}=(\Spc H_{\Op A})'$ of
Proposition~\ref{Prop:Realization} below then realizes the duality as the continuous extension of
the $\Spc S$--$\Spc S'$ pairing, which is unambiguous because every element of $\Spc H_{\Op A}$
is annihilated by $\Spc N_{\V p}=\Spc N_{\Op B}$.

\begin{proposition}[Realization]
\label{Prop:Realization}
Let $(\Op A,\Op B)$ be a $\V p$-conditional Riesz pair in the sense of
Definition~\ref{Def:Intrinsic}. Then, the following properties hold.
\begin{enumerate}
\item \emph{Duality bound}: For every $\varphi\in\Spc S_{\V p}$ and $\psi\in\Spc S$,
\begin{align*}
|\langle\varphi,\psi\rangle|\le|\varphi|_{\Op A}\,|\psi|_{\Op B},
\qquad\text{and}\qquad
\langle\varphi,\psi\rangle=\langle\pi\Op A\varphi,\pi\psi\rangle_{\dot{\Spc H}_{\Op B}} .
\end{align*}
\item \emph{Realization of $\Spc H_{\Op A}$}: Let $h\in\Spc H_{\Op A}$ be represented by the
$|\cdot|_{\Op A}$-Cauchy sequence $(\varphi_j)\subset\Spc S_{\V p}$. Then the limit
\begin{align*}
\langle h,\psi\rangle\eqdef\lim_{j\to\infty}\langle\varphi_j,\psi\rangle
\end{align*}
exists for every $\psi\in\Spc S$, depends only on $h$, and defines an element of $\Spc S'$. The
resulting map $\Spc H_{\Op A}\to\Spc S'$ is linear, injective, and continuous for the weak-$\ast$
topology; it realizes $\Spc H_{\Op A}$ as a subspace of $\Spc S'$.
\item \emph{Realization of $\dot{\Spc H}_{\Op B}$}: In the same fashion, every $g\in\dot{\Spc H}_{\Op B}$
that is represented by the $|\cdot|_{\Op B}$-Cauchy sequence $(\psi_k)\subset\Spc S$ defines an
element of $(\Spc S_{\V p})'\cong\Spc S'/\Spc N_{\V p}$ through
$\langle g,\varphi\rangle\eqdef\lim_k\langle\psi_k,\varphi\rangle$ for $\varphi\in\Spc S_{\V p}$, and
the resulting map $\dot{\Spc H}_{\Op B}\to\Spc S'/\Spc N_{\V p}$ is linear, injective, and
continuous.
\end{enumerate}
\end{proposition}
\begin{proof}
\emph{Duality bound}:
Let $\varphi\in\Spc S_{\V p}$ and $\psi\in\Spc S$. Clause (ii) provides a
$|\cdot|_{\Op B}$-Cauchy sequence $(\psi_k)\subset\Spc S$ that represents $\pi\Op A\varphi$ and
satisfies $\Op B\psi_k\to\varphi$ in $\Spc S'$. Therefore
\begin{align*}
\langle\varphi,\psi\rangle=\lim_{k\to\infty}\langle\Op B\psi_k,\psi\rangle,
\qquad
|\langle\Op B\psi_k,\psi\rangle|\le|\psi_k|_{\Op B}\,|\psi|_{\Op B},
\end{align*}
where the inequality is the Cauchy--Schwarz inequality for the $\Op B$-form. In the limit,
$|\psi_k|_{\Op B}\to\|\pi\Op A\varphi\|_{\dot{\Spc H}_{\Op B}}=|\varphi|_{\Op A}$, where the last
equality is the isometry property of Clause (ii). This proves the bound. The same computation
identifies the pairing: the $\Op B$-form extends continuously to $\dot{\Spc H}_{\Op B}$, which gives
$\lim_k\langle\Op B\psi_k,\psi\rangle=\langle\pi\Op A\varphi,\pi\psi\rangle_{\dot{\Spc H}_{\Op B}}$.

\medskip
\emph{Realization of $\Spc H_{\Op A}$}:
The bound of Item 1 gives
$|\langle\varphi_j-\varphi_i,\psi\rangle|\le|\varphi_j-\varphi_i|_{\Op A}\,|\psi|_{\Op B}$, so that
the scalar sequence $(\langle\varphi_j,\psi\rangle)_j$ is Cauchy and its limit vanishes for every
null sequence: the limit exists and depends only on $h$. In the limit,
$|\langle h,\psi\rangle|\le\|h\|_{\Spc H_{\Op A}}\,|\psi|_{\Op B}$ for every $\psi\in\Spc S$. Since
$\Op B$ is continuous, $|\cdot|_{\Op B}$ is a continuous seminorm on $\Spc S$, so that
$\langle h,\cdot\rangle\in\Spc S'$; the same bound makes the realization map weak-$\ast$ continuous.
For the injectivity, let us suppose that $\langle h,\psi\rangle=0$ for every $\psi\in\Spc S$, and
let $\Op J$ denote the isometric extension (Riesz map) of $\pi\Op A$ to $\Spc H_{\Op A}$. The
identity of Item 1 passes to the limit and gives
$\langle h,\psi\rangle=\langle\Op J h,\pi\psi\rangle_{\dot{\Spc H}_{\Op B}}$, so that the class
$\Op J h$ is orthogonal to $\pi(\Spc S)$, which is dense in $\dot{\Spc H}_{\Op B}$. Hence
$\Op J h=0$ and, by the isometry, $h=0$.

\medskip
\emph{Realization of $\dot{\Spc H}_{\Op B}$}:
The preceding argument applies with the roles of $(\Spc S_{\V p},|\cdot|_{\Op A})$ and
$(\Spc S,|\cdot|_{\Op B})$ exchanged: the bound of Item 1 makes the limit $\langle g,\varphi\rangle$
exist and depend only on $g$, with
$|\langle g,\varphi\rangle|\le\|g\|_{\dot{\Spc H}_{\Op B}}\,|\varphi|_{\Op A}$, and the continuity of
$\Op A$ gives $\langle g,\cdot\rangle\in(\Spc S_{\V p})'$. Two points are specific to this side.
First, the dual $(\Spc S_{\V p})'$ identifies with $\Spc S'/\Spc N_{\V p}$, because
$\Spc S_{\V p}^\perp=\Spc N_{\V p}$ is finite-dimensional, hence weak-$\ast$ closed. Second, the
injectivity follows from
$\langle g,\varphi\rangle=\langle g,\pi\Op A\varphi\rangle_{\dot{\Spc H}_{\Op B}}$ and from the
density of $\pi\Op A(\Spc S_{\V p})$ in $\dot{\Spc H}_{\Op B}$, which holds because $\pi\Op A$ is an
isometric isomorphism and $\Spc S_{\V p}$ is dense in $\Spc H_{\Op A}$.
\end{proof}

The two clauses of Definition~\ref{Def:Intrinsic} may look independent, but they are not: each is
equivalent to the single requirement that $\V p$ span the attained null space, and to the minimality
of the conditioning.

\begin{theorem}[Attainable null space, conditional invertibility, and minimality]
\label{Thm:AttainableNull}
Let $\V p\subset\Spc S'$ be a linearly independent family, let $\Op B:\Spc S\to\Spc S'$ be
symmetric positive-definite with attainable null space $\Spc N_{\Op B}$
(Definition~\ref{Def:AttainableNull}), and let $\Op A:\Spc S\to\Spc S'$ be a continuous
operator that inverts $\Op B$ through the Hilbertian extensions, in the sense that
$\Op A\Op B=\Identity$ on $\Spc S$ and $\Op B\Op A=\Identity$ on $\Spc S_{\V p}$. Then, the
following statements are equivalent.
\begin{enumerate}
\item[(i)] \emph{(Attainability)} $\Spc N_{\Op B}=\Spc N_{\V p}$.
\item[(ii)] \emph{(Invertibility through completion)} $(\Op A,\Op B)$ is a $\V p$-conditional Riesz
pair (Definition~\ref{Def:Intrinsic}); equivalently, the Hilbertian extensions of $\pi\Op A$ and
$\Op B$ are mutually inverse isometries between $\Spc H_{\Op A}$ and $\dot{\Spc H}_{\Op B}$.
\item[(iii)] \emph{(Minimality)} The operator $\Op A$ is $\V p$-conditionally positive-definite and
its finiteness domain is the detrended space,
\begin{align*}
\Spc S_{\V p}=\Spc N^\perp_\Op B=\{\varphi\in\Spc S:\ \pi\Op A\varphi\in\dot{\Spc H}_{\Op B}\},
\end{align*}
which is the largest subspace of $\Spc S$ on which the image $\Op A\varphi$ has finite energy. No
smaller conditioning is admissible: whenever $\Spc N_{\V q}\not\supseteq\Spc N_{\Op B}$, the space
$\Spc S_{\V q}$ contains test functions whose image leaves $\dot{\Spc H}_{\Op B}$.
\end{enumerate}
When they hold, $\Spc N_{\Op B}=\Spc N_{\V p}$ is the trend space of the theory, and
$\langle\Op A\varphi,\varphi\rangle=\|\pi\Op A\varphi\|_{\dot{\Spc H}_{\Op B}}^2=|\varphi|_{\Op A}^2$
for every $\varphi\in\Spc S_{\V p}$.
\end{theorem}
\begin{proof}
\emph{(i)$\Leftrightarrow$(ii)}:
Let $\varphi\in\Spc S_{\V p}$. By the Bridge below, $\Spc N_{\Op B}^\perp=\Spc S_{\V p}$, so that
$\pi\Op A\varphi\in\dot{\Spc H}_{\Op B}$ is represented by some $|\cdot|_{\Op B}$-Cauchy sequence
$(\psi_k)\subset\Spc S$ with $\pi\psi_k\to\pi(\Op A\varphi)$. As $\Op A\varphi\in\Spc S'$, the relation
$\Op B\Op A=\Identity$ on $\Spc S_{\V p}$ then translates into $\Op B\psi_k\to\varphi$ in $\Spc S'$, which completes the
sequential reading of Clause~(ii). For the isometry, the Cauchy--Schwarz bound
$|\langle\Op B\psi_k-\Op B\psi_j,\psi_k\rangle|\le|\psi_k-\psi_j|_{\Op B}|\psi_k|_{\Op B}$ gives
$\langle\Op B\psi_k,\psi_k\rangle-\langle\varphi,\psi_k\rangle\to0$.
Since $\varphi\in\Spc S_{\V p}$, the weak-$\ast$ continuous functional $f\mapsto\langle \varphi,f\rangle$
vanishes on $\Spc N_{\V p}=\Spc N_{\Op B}$, so that it descends to a continuous functional on
$\Spc S'/\Spc N_{\Op B}$. The convergence $\pi\psi_k\to\pi(\Op A\varphi)$ therefore gives
$\langle\varphi,\psi_k\rangle\to\langle\Op A\varphi,\varphi\rangle$.
Hence
\begin{align*}
\|\pi\Op A\varphi\|_{\dot{\Spc H}_{\Op B}}^2=\lim_{k\to\infty}|\psi_k|_{\Op B}^2
=\langle\Op A\varphi,\varphi\rangle=|\varphi|_{\Op A}^2 .
\end{align*}
The reverse composition follows from $\Op A\Op B=\Identity$ on $\Spc S$, which extends to
$\dot{\Spc H}_{\Op B}$ by density of $\pi(\Spc S)$. The converse is immediate, as Clause~(i) is part
of Definition~\ref{Def:Intrinsic}.

\medskip
\emph{Bridge}:
For $\varphi\in\Spc S$, write
$\|\varphi\|_{\Op A}^\ast\eqdef\sup_{\psi\in\Spc S\setminus\{0\}}|\langle\varphi,\psi\rangle|/|\psi|_{\Op B}\in[0,+\infty]$
for the dual norm induced by $|\cdot|_{\Op B}$. On $\Spc S_{\V p}$, it coincides with the 
seminorm $|\varphi|_{\Op A}=\sqrt{\langle\Op A\varphi,\varphi\rangle}$ of
Definition~\ref{Def:CondPositiveDef2} (Proposition~\ref{Prop:Realization}), whereas on the remainder
of $\Spc S$ the two need not agree and $\|\varphi\|_{\Op A}^\ast$ may be infinite. The image
$\Op A\varphi$ (or rather $\pi\Op A \varphi$) lies in $\dot{\Spc H}_{\Op B}$ if and only if the functional
$\psi\mapsto\langle\varphi,\psi\rangle$ is $|\cdot|_{\Op B}$-bounded; that is,
$\|\varphi\|_{\Op A}^\ast<\infty$; and this holds if and only if $\varphi\in\Spc N_{\Op B}^\perp$.
Indeed, if $\langle\varphi,p\rangle\neq0$ for some $p\in\Spc N_{\Op B}$, then
$p=\lim_k\chi_k$ with $|\chi_k|_{\Op B}\to0$, so that the quotient
$|\langle\varphi,\chi_k\rangle|/|\chi_k|_{\Op B}$ diverges; conversely, the inverse relation
$\Op A\Op B=\Identity$ on $\Spc S$ identifies $\|\varphi\|_{\Op A}^\ast$ with
$\|\pi\Op A\varphi\|_{\dot{\Spc H}_{\Op B}}$, which is finite exactly on $\Spc N_{\Op B}^\perp$.
Therefore
\begin{align*}
\{\varphi\in\Spc S:\ \pi\Op A\varphi\in\dot{\Spc H}_{\Op B}\}=\Spc N_{\Op B}^\perp .
\end{align*}

\medskip
\emph{(i)$\Rightarrow$(iii)}:
Let us assume that $\Spc N_{\Op B}=\Spc N_{\V p}$, so that
$\Spc S_{\V p}=\Spc N_{\Op B}^\perp$ is the finiteness domain identified in the Bridge.
For $\varphi\in\Spc S_{\V p}$, the relation $\Op B\Op A=\Identity$ gives
\begin{align*}
\langle\Op A\varphi,\varphi\rangle
=\langle\Op A\varphi,\Op B\Op A\varphi\rangle
=\|\pi\Op A\varphi\|_{\dot{\Spc H}_{\Op B}}^2\ \ge\ 0 .
\end{align*}
This quantity vanishes only for $\varphi=0$: the identity
$\|\pi\Op A\varphi\|_{\dot{\Spc H}_{\Op B}}=0$ means that $\Op A\varphi\in\Spc N_{\Op B}$, which the
extension of $\Op B$ annihilates (see the discussion that follows
Definition~\ref{Def:Intrinsic}), from which we get $\varphi=\Op B\Op A\varphi=0$. The operator
$\Op A$ is therefore $\V p$-conditionally positive-definite.

\medskip
\emph{(iii)$\Rightarrow$(i)}:
The Bridge identifies the finiteness domain of~(iii) as $\Spc N_{\Op B}^\perp$, so that the identity
of~(iii) reads $\Spc N_{\V p}^\perp=\Spc N_{\Op B}^\perp$. Let us apply $\perp$ once more: the
bipolar theorem gives $\overline{\Spc N_{\V p}}^{\,w\ast}=\overline{\Spc N_{\Op B}}^{\,w\ast}$. The
space $\Spc N_{\V p}$ is finite-dimensional, hence weak-$\ast$ closed, which yields
$\Spc N_{\V p}=\overline{\Spc N_{\Op B}}^{\,w\ast}\supseteq\Spc N_{\Op B}$. The attainable null
space $\Spc N_{\Op B}$ is then contained in a finite-dimensional space, so that it is itself
finite-dimensional and weak-$\ast$ closed. We conclude that $\Spc N_{\Op B}=\Spc N_{\V p}$, which
is~(i). Let us stress that the finite-dimensionality of $\Spc N_{\Op B}$ is not an assumption here:
it is a consequence of the finiteness-domain identity of~(iii).

\medskip
\emph{Minimality}:
Let us assume that $\Spc N_{\V q}\not\supseteq\Spc N_{\Op B}$ and pick some
$p\in\Spc N_{\Op B}\setminus\Spc N_{\V q}$. Since $\Spc S$ separates the points of $\Spc S'$ and
$\Spc N_{\V q}$ is finite-dimensional, there is a test function
$\varphi\in\Spc N_{\V q}^\perp=\Spc S_{\V q}$ with $\langle\varphi,p\rangle\neq0$. Such a
$\varphi$ lies outside $\Spc N_{\Op B}^\perp$, which is precisely the region identified in the
Bridge: there, $\Op A\varphi$ leaves $\dot{\Spc H}_{\Op B}$ and the dual norm
$\|\varphi\|_{\Op A}^\ast$ diverges. No conditioning smaller than $\Spc N_{\V p}=\Spc N_{\Op B}$ is
therefore admissible.
\end{proof}

The important outcome of this section is that \emph{the realization determines the conditioning}.
Proposition~\ref{Prop:Realization} realizes $\Spc H_{\Op A}$ as a genuine subspace of $\Spc S'$ but
the energy space $\dot{\Spc H}_{\Op B}$ only in the quotient $\Spc S'/\Spc N_{\Op B}$, and
Theorem~\ref{Thm:AttainableNull} turns this asymmetry into a rigidity statement: the embedding
$\Spc H_{\Op A}\hookrightarrow\Spc S'$ carries the structure of Definition~\ref{Def:Intrinsic} if
and only if $\Spc N_{\V p}=\Spc N_{\Op B}$. A smaller $\V p$ violates the positivity of
$|\cdot|_{\Op A}$, a larger one violates the attainability, and in both cases $(\Op A,\Op B)$ is no
longer a conditional Riesz pair. Once the conditioning is fixed, the concrete Hilbert--Gelfand
spaces that we construct in the next section are themselves determined by $(\Op A,\Op B)$, up to an
equivalence of norms (Corollary~\ref{Cor:PhiIndependence}).

\subsection{Concrete Realization of Hilbert--Gelfand Spaces}

Our Definition~\ref{Def:Intrinsic} of conditional Riesz pairs involves only $\V p\in (\Spc S')^{N_0}$, $\Op A$ and $\Op B$. It is abstract because the induced space $\dot{\Spc H}_{\Op B}$ is formed of equivalence classes of distributions.
We now
show that through the selection of a biorthogonal family $\V \phi \in\Spc S^{N_0}$, we can define an oblique projector that isometrically maps the abstract space $\dot{\Spc H}_{\Op B}$ into  the primary component of a concrete Hilbert--Gelfand space. We start with the existence of such a family.

\begin{lemma}[Biorthogonal systems and projectors]
\label{Lem:Biorthogonal}
Let $\V p=(p_1,\dots,p_{N_0})$ be a linearly independent family in $\Spc S'$.
\begin{enumerate}
\item There is a family $\V \phi=(\phi_1,\dots,\phi_{N_0})$ in $\Spc S$ that is biorthogonal to $\V p$, in
the sense that $\langle p_i,\phi_j\rangle=\delta_{ij}$ for all $i,j$.
\item For any such $\V \phi$, we form the rank-$N_0$ operators
\begin{align*}
\Op R_{\V p}\eqdef\sum_{n=1}^{N_0} p_n\,\langle p_n,\cdot\rangle:\Spc S\to\Spc S',
\qquad
\Op R_{\V \phi}\eqdef\sum_{n=1}^{N_0} \phi_n\,\langle \cdot,\phi_n\rangle:\Spc S'\to\Spc S,
\end{align*}
which satisfy $\Op R_{\V p}\varphi=0$ for every $\varphi\in\Spc S_{\V p}$ and
$\Op R_{\V p}\Op R_{\V \phi}\Op R_{\V p}=\Op R_{\V p}$. The operator
$\Op P_{\V \phi}\eqdef\Identity-\Op R_{\V \phi}\Op R_{\V p}$ is a continuous projector of $\Spc S$ onto $\Spc S_{\V p}$
obliquely along $\Spc N_{\V \phi}\eqdef\mathrm{span}(\V \phi)$, so that $\Spc S=\Spc S_{\V p}\oplus\Spc N_{\V \phi}$.
Its adjoint $\Op P_{\V \phi}^\ast=\Identity-\Op R_{\V p}\Op R_{\V \phi}$ is a projector of $\Spc S'$ onto
$\Spc S'_{\V \phi}\eqdef\{f\in\Spc S':\langle f,\phi_k\rangle=0,\ k=1,\dots,N_0\}$ along $\Spc N_{\V p}$.
\end{enumerate}
\end{lemma}

\begin{proof}
\emph{Existence}:
Let us fix $j\in\{1,\dots,N_0\}$ and set $M_j\eqdef\bigcap_{i\neq j}\ker p_i\subset\Spc S$. A continuous
linear functional on $\Spc S$ belongs to the span of $\{p_i\}_{i\neq j}$ if and only if it vanishes on
$M_j$; this is the standard consequence of the Hahn--Banach theorem which identifies the annihilator of a
finite family of functionals \cite[Ch.~3]{Rudin1991}. Since $\V p$ is linearly independent, $p_j$ is not in the span of
$(p_i)_{i\neq j}$, and therefore $p_j$ does not vanish on $M_j$. Consequently, there exists some $\psi_j\in M_j$ with
$\langle p_j,\psi_j\rangle\neq0$, and we set $\phi_j\eqdef\psi_j/\langle p_j,\psi_j\rangle$. This gives
$\langle p_i,\phi_j\rangle=\delta_{ij}$ for all $i,j$, which is the announced biorthogonality.

\medskip
\emph{Projector relations}:
For $\varphi\in\Spc S$, a direct computation gives
$\Op R_{\V \phi}\Op R_{\V p}\varphi=\sum_{k}\phi_k\langle p_k,\varphi\rangle$, because
$\langle p_n,\phi_k\rangle=\delta_{nk}$. If $\varphi\in\Spc S_{\V p}$, then $\langle p_n,\varphi\rangle=0$
for all $n$, and hence $\Op R_{\V p}\varphi=0$. Applying $\Op R_{\V p}$ to
$\Op R_{\V \phi}\Op R_{\V p}\varphi=\sum_k\phi_k\langle p_k,\varphi\rangle$ and using biorthogonality once
more yields $\Op R_{\V p}\Op R_{\V \phi}\Op R_{\V p}\varphi=\sum_n p_n\langle p_n,\varphi\rangle=\Op R_{\V p}\varphi$,
which is the second stated relation.

\medskip
\emph{Projector}:
For every $\varphi\in\Spc S$ and every $n$, we have
$\langle p_n,\Op P_{\V \phi}\varphi\rangle=\langle p_n,\varphi\rangle-\sum_k\langle p_n,\phi_k\rangle\langle p_k,\varphi\rangle
=\langle p_n,\varphi\rangle-\langle p_n,\varphi\rangle=0$, and hence $\Op P_{\V \phi}\varphi\in\Spc S_{\V p}$. If
$\varphi\in\Spc S_{\V p}$, then $\Op R_{\V p}\varphi=0$, and therefore $\Op P_{\V \phi}\varphi=\varphi$. It
follows that $\Op P_{\V \phi}^2=\Op P_{\V \phi}$ and $\mathrm{ran}\,\Op P_{\V \phi}=\Spc S_{\V p}$. Moreover $\Op P_{\V \phi}\varphi=0$ holds if
and only if $\varphi=\Op R_{\V \phi}\Op R_{\V p}\varphi=\sum_k\phi_k\langle p_k,\varphi\rangle$, that is
$\varphi\in\Spc N_{\V \phi}$; thus $\ker\Op P_{\V \phi}=\Spc N_{\V \phi}$ and
$\Spc S=\Spc S_{\V p}\oplus\Spc N_{\V \phi}$.

\medskip
\emph{Adjoint}:
A direct computation gives $\langle\Op P_{\V \phi}^\ast f,\varphi\rangle=\langle f,\Op P_{\V \phi}\varphi\rangle$ for all
$f\in\Spc S'$ and $\varphi\in\Spc S$, which identifies $\Op P_{\V \phi}^\ast$ as the adjoint of $\Op P_{\V \phi}$. The
relations $\ker\Op P_{\V \phi}^\ast=\Spc N_{\V p}$, $\mathrm{ran}\,\Op P_{\V \phi}^\ast=\Spc S'_{\V \phi}$, and $\Spc S'=\Spc S'_{\V \phi}\oplus\Spc N_{\V p}$ follow by
transposition of those for $\Op P_{\V \phi}$.
\end{proof}

To facilitate the reading of the results that follow, we recapitulate in
Table~\ref{Tab:Glossary} the notation for the objects that underlie our subsequent
constructions.
\begin{table}[htbp]
\centering
\small
\begin{tabular}{@{}lp{0.66\textwidth}@{}}
\toprule
Symbol & Meaning\\
\midrule
$(\Spc S,\Spc S')$ & nuclear Gelfand pair, with canonical embedding $\iota:\Spc S\to\Spc S'$ (Definition~\ref{Def:NuclearGelfandPair})\\
$\V p=(p_1,\dots,p_{N_0})$ & conditioning family in $\Spc S'$, with $\Spc N_{\V p}=\mathrm{span}(\V p)$ and $N_0=\dim\Spc N_{\V p}$\\
$\Spc S_{\V p}=\Spc N_{\V p}^\perp$ & detrended test functions, those with $\V p(\varphi)=\V 0$\\
$\Spc N_{\Op B}$ & attainable null space of the energy $|\cdot|_{\Op B}$ (Definition~\ref{Def:AttainableNull})\\
$\V \phi=(\phi_1,\dots,\phi_{N_0})$ & family in $\Spc S$ biorthogonal to $\V p$, with $\Spc N_{\V \phi}=\mathrm{span}(\V \phi)$ (Lemma~\ref{Lem:Biorthogonal})\\
$\Op R_{\V p},\ \Op R_{\V \phi}$ & rank-$N_0$ operators $\sum_n p_n\langle p_n,\cdot\rangle$ and $\sum_n\phi_n\langle\cdot,\phi_n\rangle$\\
$\Op P_{\V \phi},\ \Op P_{\V \phi}^\ast$ & oblique projectors of $\Spc S$ onto $\Spc S_{\V p}$ and of $\Spc S'$ onto $\Spc S'_{\V \phi}$\\
$\pi$ & quotient map $\Spc S'\to\Spc S'/\Spc N_{\Op B}$\\
$\Op B=\Op L^\ast\Op L$ & normal (energy) operator, with $|\varphi|_{\Op B}=\|\Op L\varphi\|_{L_2}$\\
$\Op A,\ \Op A_{\V \phi}$ & conditional inverse of $\Op B$ and its conditioned form $\Op P_{\V \phi}^\ast\Op A\Op P_{\V \phi}$\\
$\Op A_+,\ \Op B_+$ & augmented pair $\Op A_{\V \phi}+\Op R_{\V p}$ and $\Op B+\Op R_{\V \phi}$ (Theorem~\ref{Thm:Augmentation})\\
$\Spc H_{\Op A}$ & $\overline{(\Spc S_{\V p},|\cdot|_{\Op A})}$, realized in $\Spc S'$ \eqref{Eq:HA}\\
$\dot{\Spc H}_{\Op B}$ & $\overline{(\Spc S,|\cdot|_{\Op B})}$, realized in the quotient $\Spc S'/\Spc N_{\Op B}$ \eqref{Eq:HB}\\
$\Spc H'_{\Op A}$ & $\Op P_{\V \phi}^\ast(\dot{\Spc H}_{\Op B})\subset\Spc S'_{\V \phi}$, the concrete realization of $\dot{\Spc H}_{\Op B}$\\
$\Spc H,\ \Spc H'$ & $\Spc H_{\Op A}\oplus\Spc N_{\V \phi}$ and $\Spc H'_{\Op A}\oplus\Spc N_{\V p}$, the augmented dual pair; independent of $\V \phi$ as sets (Corollary~\ref{Cor:PhiIndependence})\\
$\Spc H_{\Op L}$, $\Op L:\Spc H_{\Op L}\to L_2$ & native space $\Spc H'$ of an admissible $\Op L$ and the extension $\widetilde{\Op L}\pi$ of $\Op L$ to it (Proposition~\ref{Prop:NativeUniqueness})\\
$Q,\ e_\varphi$ & dual energy and representer of $\varphi$ (Theorem~\ref{Thm:PracticalChecklist})\\
\bottomrule
\end{tabular}
\caption{The objects of the construction and where they are introduced.}
\label{Tab:Glossary}
\end{table}

In accordance
with the convention of Section~\ref{Sec:HilbertGelfand}, we denote the operators and their
Hilbertian extensions by the same symbols: $\Op B$ extends to the isometric isomorphism
$\dot{\Spc H}_{\Op B}\to\Spc H_{\Op A}$ of Definition~\ref{Def:Intrinsic} and annihilates
$\Spc N_{\V p}$, while $\pi\Op A$ extends to the inverse isometry
$\Spc H_{\Op A}\to\dot{\Spc H}_{\Op B}$. At the level of test functions, the sequential reading
of Definition~\ref{Def:Intrinsic}(ii) translates these mutually inverse extensions into the two
inversion identities
\begin{align}
\label{Eq:QuasiInverse}
\Op A\,\Op B=\Identity\ \text{on}\ \Spc S,
\qquad
\Op B\,\Op A=\Identity\ \text{on}\ \Spc S_{\V p},
\end{align}
where the second identity is restricted to $\Spc S_{\V p}$ because $\Op A\varphi$ lies in
$\dot{\Spc H}_{\Op B}$---the domain of the extended $\Op B$---only when
$\varphi\in\Spc S_{\V p}$ (Theorem~\ref{Thm:AttainableNull}, Item (iii)). In the concrete
constructions of the later sections, where the operators are realized as convolutions or as
Fourier multipliers, the identities \eqref{Eq:QuasiInverse} can also be verified
distributionally.

\begin{theorem}[Concretization]
\label{Thm:Concretization}
Let $(\Op A,\Op B)$ be a $\V p$-conditional Riesz pair in the sense of
Definition~\ref{Def:Intrinsic}, so that the inversion identities \eqref{Eq:QuasiInverse} hold.
Let us fix a biorthogonal family $\V \phi\subset\Spc S$ with
$\langle p_i,\phi_j\rangle=\delta_{ij}$, whose existence together with the projectors $\Op P_{\V \phi}$
and $\Op P_{\V \phi}^\ast$ is granted by Lemma~\ref{Lem:Biorthogonal}, and let us set
$\Op A_{\V \phi}\eqdef\Op P_{\V \phi}^\ast\,\Op A\,\Op P_{\V \phi}$. Then, the
following hold.
\begin{enumerate}
\item $\Op P_{\V \phi}^\ast$ restricts to an isometric identification
\begin{align*}
\Op P_{\V \phi}^\ast:\ \dot{\Spc H}_{\Op B}\ \xrightarrow{\ \sim\ }\
\Op P_{\V \phi}^\ast\!\big(\dot{\Spc H}_{\Op B}\big)\subset\Spc S'_{\V \phi},
\end{align*}
which selects in each equivalence class the unique representative that annihilates $\V \phi$. 
This enables us to identify a concrete dual of $\Spc H_\Op A$ as $\Spc H'_{\Op A}=\Op P_{\V \phi}^\ast(\dot{\Spc H}_{\Op B})\cong \dot{\Spc H}_{\Op B}$.
\item The conditioned operator $\Op A_{\V \phi}$ satisfies $\ker\Op A_{\V \phi}=\Spc N_{\V \phi}$ and
\begin{align*}
\Op B\,\Op A_{\V \phi}=\Op P_{\V \phi}\ \text{on}\ \Spc S,
\qquad
\Op A_{\V \phi}\,\Op B=\Op P_{\V \phi}^\ast\ \text{on}\ \Spc S.
\end{align*}
\item The objects $\Spc H_{\Op A}$, $\dot{\Spc H}_{\Op B}$ and $\pi\Op A$ are independent of
$\V \phi$. 
\end{enumerate}
\end{theorem}

\begin{proof}
\emph{Projectors and the identification}:
By Lemma~\ref{Lem:Biorthogonal}, $\Op P_{\V \phi}$ projects $\Spc S$ onto $\Spc S_{\V p}$ along
$\Spc N_{\V \phi}$, and $\Op P_{\V \phi}^\ast$ projects $\Spc S'$ onto $\Spc S'_{\V \phi}$ along $\Spc N_{\V p}$.
The hypothesis $\Spc N_{\Op B}=\Spc N_{\V p}$ shows that $\ker\Op P_{\V \phi}^\ast=\Spc N_{\Op B}$, which means
that $\Op P_{\V \phi}^\ast$ annihilates exactly the null classes. It therefore descends to the stated isometry
from $\dot{\Spc H}_{\Op B}$ onto its representatives in $\Spc S'_{\V \phi}$.

\medskip
\emph{Inversion relations}:
As noted after Definition~\ref{Def:Intrinsic}, the extension of $\Op B$ annihilates
$\Spc N_{\Op B}=\mathrm{span}(\V p)$, from which we deduce that $\Op B\Op P_{\V \phi}^\ast=\Op B$. The symmetry of $\Op B$ gives
$\langle p_k,\Op B\,\cdot\rangle=\langle\Op B p_k,\cdot\rangle=0$, which yields $\Op P_{\V \phi}\Op B=\Op B$.
Therefore
\begin{align*}
\Op B\Op A_{\V \phi}=\Op B\Op P_{\V \phi}^\ast\Op A\Op P_{\V \phi}=\Op B\Op A\Op P_{\V \phi}=\Op P_{\V \phi},
\qquad
\Op A_{\V \phi}\,\Op B=\Op P_{\V \phi}^\ast\Op A\Op P_{\V \phi}\Op B=\Op P_{\V \phi}^\ast\Op A\Op B=\Op P_{\V \phi}^\ast.
\end{align*}
The last equality in the first chain applies the second identity of
\eqref{Eq:QuasiInverse} to $\Op P_{\V \phi}\varphi\in\Spc S_{\V p}$, while the last equality in the
second chain applies the first identity of \eqref{Eq:QuasiInverse}. 

\medskip
\emph{Null space}:
Let us suppose that $\Op A_{\V \phi}\varphi=0$. Then $\Op A\Op P_{\V \phi}\varphi\in\ker\Op P_{\V \phi}^\ast=\Spc N_{\V p}=\ker\pi$,
which gives $\pi\Op A(\Op P_{\V \phi}\varphi)=0$. The map $\pi\Op A$ is injective on $\Spc H_{\Op A}$, and
hence $\Op P_{\V \phi}\varphi=0$, that is $\varphi\in\Spc N_{\V \phi}$. The reverse inclusion follows from
$\Op P_{\V \phi}\V \phi=\V 0$, which gives $\Op A_{\V \phi}\V \phi=\V 0$. Therefore
$\ker\Op A_{\V \phi}=\Spc N_{\V \phi}$. Finally, for $\varphi\in\Spc S_{\V p}$ we have
$\Op A_{\V \phi}\varphi=\Op P_{\V \phi}^\ast\Op A\varphi$, which differs from $\Op A\varphi$ by an element of
$\Spc N_{\V p}=\ker\pi$; consequently, $\pi\Op A_{\V \phi}\varphi=\pi\Op A\varphi$: the
conditioned operator represents the same isometry as $\Op A$.

\medskip
\emph{Independence}:
The spaces $\Spc H_{\Op A}$ and $\dot{\Spc H}_{\Op B}$ and the map $\pi\Op A$ make no reference to
$\V \phi$.
\end{proof}

The concretization of Theorem~\ref{Thm:Concretization} still carries the projector $\Op P_{\V \phi}$, which depends on the choice of $\V \phi$. One can
go one step further and augment the two operators by
complementary rank-$N_0$ terms, which reinstall the missing directions as explicit
finite-dimensional blocks. The outcome is an unconditioned dual pair of Hilbert--Gelfand spaces with
a direct-sum structure.

\begin{theorem}[Augmented pair and direct-sum structure]
\label{Thm:Augmentation}
Under the assumptions and notations of Theorem~\ref{Thm:Concretization}, and with the rank-$N_0$
operators $\Op R_{\V p}$ and $\Op R_{\V \phi}$ of Lemma~\ref{Lem:Biorthogonal}, we define the augmented
operators
\begin{align*}
\Op A_+\eqdef\Op A_{\V \phi}+\Op R_{\V p},
\qquad
\Op B_+\eqdef\Op B+\Op R_{\V \phi}.
\end{align*}
The notation $\V \phi(f)\eqdef\big(\langle f,\phi_1\rangle,\dots,\langle f,\phi_{N_0}\rangle\big)$ collects the
$\V \phi$-coordinates of $f \in \Spc S'$. Then, the following properties hold.
\begin{enumerate}
\item {\em Positivity:} The operators $\Op A_+$ and $\Op B_+$ are positive-definite on $\Spc S$, with
\begin{align*}
|\varphi|_{\Op A_+}^2=|\Op P_{\V \phi}\varphi|_{\Op A}^2+|\V p(\varphi)|_2^2,
\qquad
|\varphi|_{\Op B_+}^2=|\varphi|_{\Op B}^2+|\V \phi(\varphi)|_2^2,
\end{align*}
where $|\cdot|_2$ denotes the Euclidean norm on $\R^{N_0}$.
\item {\em Inversion:} $\Op B_+\Op A_+=\Identity$ and $\Op A_+\Op B_+=\Identity$ on $\Spc S$. In particular, no
conditioning and no projector remain.
\item {\em Direct sums:} The completions
\begin{align*}
\Spc H\eqdef\overline{(\Spc S,|\cdot|_{\Op A_+})}
=\Spc H_{\Op A}\oplus\Spc N_{\V \phi},
\qquad
\Spc H'\eqdef\overline{(\Spc S,|\cdot|_{\Op B_+})}
=\Spc H'_{\Op A}\oplus\Spc N_{\V p}
\end{align*}
are orthogonal direct sums, where $\Spc H'_{\Op A}\eqdef\Op P_{\V \phi}^\ast(\dot{\Spc H}_{\Op B})\subset\Spc S'_{\V \phi}$
is the concrete realization of $\dot{\Spc H}_{\Op B}$, and where the finite-dimensional blocks carry
the Euclidean metrics $|\V p(\cdot)|_2$ and $|\V \phi(\cdot)|_2$, respectively. Both spaces satisfy
$\Spc S\embedD\Spc H\embedD\Spc S'$ and $\Spc S\embedD\Spc H'\embedD\Spc S'$ with the canonical
inclusions; they are Hilbert--Gelfand spaces.
\item {\em Duality:} The extension of the $\Spc S$--$\Spc S'$ pairing identifies $\Spc H'$ with the dual
of $\Spc H$, and the pairing is block-diagonal with respect to the two direct sums. The induced
extensions
\begin{align*}
\Op A_+:\Spc H\to\Spc H',
\qquad
\Op B_+:\Spc H'\to\Spc H
\end{align*}
are isometric, mutually inverse, and exchange the finite-dimensional blocks:
$\Op A_+\phi_n=p_n$ and $\Op B_+p_n=\phi_n$ for $n=1,\dots,N_0$. Consequently, $(\Op A_+,\Op B_+)$
is a compatible Riesz pair (the case $N_0=0$) for $(\Spc S,\Spc S')$.
\end{enumerate}
\end{theorem}

\begin{proof}
We first establish two key properties. The hypothesis $\Spc N_{\Op B}=\Spc N_{\V p}$ provides, for each $n$, a
sequence $(\chi^{(n)}_k)\subset\Spc S$ such that $|\chi^{(n)}_k|_{\Op B}\to0$ and
$\chi^{(n)}_k\to p_n$ in $\Spc S'$. The Cauchy--Schwarz inequality gives
$|\langle\Op B\varphi,\chi^{(n)}_k\rangle|\le|\varphi|_{\Op B}\,|\chi^{(n)}_k|_{\Op B}\to0$ for every
$\varphi\in\Spc S$. This controlled limit behavior has two consequences, which are central to the argument below:
\begin{enumerate}
\item[(i)]  the extension
of $\Op B$ to $\Spc N_{\V p}$ vanishes, so that $\Op B p_n=0$; and
\item[(ii)]  the pairing
$\langle p_n,\Op B\varphi\rangle\eqdef\lim_k\langle\chi^{(n)}_k,\Op B\varphi\rangle=0$ is well defined
and independent of the chosen sequence.
\end{enumerate}

\medskip
\emph{Positivity}:
Every $\varphi\in\Spc S$ decomposes as $\varphi=\Op P_{\V \phi}\varphi+\V \phi^\Tr\V p(\varphi)$. Since
$\Op A_{\V \phi}=\Op P_{\V \phi}^\ast\Op A\Op P_{\V \phi}$, we obtain
$\langle\Op A_{\V \phi}\varphi,\varphi\rangle=\langle\Op A\Op P_{\V \phi}\varphi,\Op P_{\V \phi}\varphi\rangle
=|\Op P_{\V \phi}\varphi|_{\Op A}^2$ and
$\langle\Op R_{\V p}\varphi,\varphi\rangle=|\V p(\varphi)|_2^2$, which proves the first norm identity.
Let us now suppose that $|\varphi|_{\Op A_+}=0$. Then, $\Op P_{\V \phi}\varphi=0$ (because $\Op P_{\V \phi}\varphi\in\Spc S_{\V p}$ and $\Op A$ is
$\V p$-conditionally positive-definite) and
$\V p(\varphi)=\V 0$, which gives $\varphi=0$. The second identity follows from
$\langle\Op R_{\V \phi}\varphi,\varphi\rangle=|\V \phi(\varphi)|_2^2$, and the positivity of $\Op B_+$ is
immediate because $|\cdot|_{\Op B}$ is already a norm.

\medskip
\emph{Inversion}:
For $\varphi\in\Spc S$, we expand the product into four terms:
\begin{align*}
\Op B_+\Op A_+
=\Op B\,\Op A_{\V \phi}+\Op B\,\Op R_{\V p}+\Op R_{\V \phi}\,\Op A_{\V \phi}+\Op R_{\V \phi}\,\Op R_{\V p}.
\end{align*}
Theorem~\ref{Thm:Concretization} gives $\Op B\Op A_{\V \phi}=\Op P_{\V \phi}$. Next,
$\Op B\Op R_{\V p}\varphi=\sum_n\langle p_n,\varphi\rangle\,\Op Bp_n=0$ by property~(i).
Likewise, $\Op R_{\V \phi}\Op A_{\V \phi}\varphi=\sum_n\phi_n\langle\Op A_{\V \phi}\varphi,\phi_n\rangle=0$,
because $\Op A_{\V \phi}\varphi\in\Spc S'_{\V \phi}$. Finally, the biorthogonality of $(\V p, \V \phi)$ yields
\begin{align*}
\Op R_{\V \phi}\Op R_{\V p}\varphi
=\sum_{m,n}\phi_m\,\langle p_n,\varphi\rangle\,\langle p_n,\phi_m\rangle
=\sum_{n}\phi_n\,\langle p_n,\varphi\rangle
=(\Identity-\Op P_{\V \phi})\varphi .
\end{align*}
The sum of the four terms is $\Op P_{\V \phi}+(\Identity-\Op P_{\V \phi})=\Identity$. For the reverse order, we expand
\begin{align*}
\Op A_+\Op B_+
=\Op A_{\V \phi}\,\Op B+\Op A_{\V \phi}\,\Op R_{\V \phi}+\Op R_{\V p}\,\Op B+\Op R_{\V p}\,\Op R_{\V \phi}.
\end{align*}
Theorem~\ref{Thm:Concretization} gives $\Op A_{\V \phi}\Op B=\Op P_{\V \phi}^\ast$ on $\Spc S$. Next,
$\Op A_{\V \phi}\Op R_{\V \phi}\varphi=0$, because
$\operatorname{ran}(\Op R_{\V \phi})=\Spc N_{\V \phi}=\ker\Op A_{\V \phi}$. Property (ii) above gives
$\Op R_{\V p}\Op B\varphi=\sum_n p_n\langle p_n,\Op B\varphi\rangle=0$. Finally,
$\Op R_{\V p}\Op R_{\V \phi}\varphi=\sum_n p_n\langle \varphi,\phi_n\rangle=(\Identity-\Op P_{\V \phi}^\ast)\varphi$, and the sum is
$\Op P_{\V \phi}^\ast+(\Identity-\Op P_{\V \phi}^\ast)=\Identity$. 

\medskip
\emph{Compatibility}:
We now verify Condition (ii) of Corollary~\ref{Cor:AlgebraicInverse}; i.e., that the
identities of Item 2 hold through the Hilbertian extensions of $\Op A_+$ and $\Op B_+$. This is equivalent to having the inclusions $\Op B_+(\Spc S)\subset\Spc H$ and
$\Op A_+(\Spc S)\subset\Spc H'$, where $\Spc H=\overline{(\Spc S,|\cdot|_{\Op A_+})}$ and
$\Spc H'=\overline{(\Spc S,|\cdot|_{\Op B_+})}$. To prove the first inclusion, we take
$\Op B_+\varphi=\Op B\varphi+\Op R_{\V \phi}\varphi$, where
$\Op R_{\V \phi}\varphi\in\Spc S$ and where $\Op B\varphi\in\Spc H_{\Op A}$ is the limit of an
$|\cdot|_{\Op A}$-Cauchy sequence in $\Spc S_{\V p}$ (Definition~\ref{Def:Intrinsic}); by the
norm identity of Item 1, this sequence is also $|\cdot|_{\Op A_+}$-Cauchy, so that
$\Op B_+\varphi\in\Spc H$. For the second one, we write
$\Op A_+\varphi=\Op A_{\V \phi}\varphi+\Op R_\V p \varphi$ and pick a
$|\cdot|_{\Op B}$-Cauchy sequence $(\psi_k)\subset\Spc S$ that represents the class
$\pi\Op A(\Op P_{\V \phi}\varphi)$, so that $\Op P_{\V \phi}^\ast\psi_k\to\Op A_{\V \phi}\varphi$ in $\Spc S'$, the difficulty here being that $\Op P_{\V \phi}^\ast\psi_k \notin \Spc S$. Hence our strategy is to
correct the $\V \phi$-coordinates of the $\psi_k$ with the sequences of the preamble: for each $k$, we set
$\tilde\psi_k\eqdef\psi_k+\sum_n\big(\V p(\varphi)_n-\V \phi(\psi_k)_n\big)\chi^{(n)}_{j(k)}
\in\Spc S$, where the index $j(k)$ is so large that the correction perturbs the
$|\cdot|_{\Op B}$-seminorm, the $\V \phi$-coordinates, and the convergence in $\Spc S'$ by at
most $1/k$. By the norm identity of Item 1, the sequence $(\tilde\psi_k)$ is
$|\cdot|_{\Op B_+}$-Cauchy with $\tilde\psi_k\to\Op A_+\varphi$ in $\Spc S'$; hence
$\Op A_+\varphi\in\Spc H'$. Corollary~\ref{Cor:AlgebraicInverse} then shows that
$(\Op A_+,\Op B_+)$ is a compatible Riesz pair, and Theorem~\ref{Theo:HilbertPositiveDef}
realizes $\Spc H$ and $\Spc H'$ as mutually dual Hilbert--Gelfand spaces whose pairing extends
the canonical one and whose Riesz maps are the mutually inverse isometric extensions of
$\Op A_+$ and $\Op B_+$. This proves Items 3 and 4, up to the identification of the blocks.

\medskip
\emph{Block structure}:
Let $u\in\Spc S_{\V p}$ and $\M c\in\R^{N_0}$. Then
$\langle\Op A_+u,\V \phi^\Tr\M c\rangle
=\langle\Op A_{\V \phi}u,\V \phi^\Tr\M c\rangle+\V p(u)^\Tr\M c=0$, because
$\Op A_{\V \phi}u\in\Spc S'_{\V \phi}$ and $\V p(u)=\V 0$; likewise,
$\langle\Op A_+(\V \phi^\Tr\M c),u\rangle=\langle\V p^\Tr\M c,u\rangle=0$. The norm identity of
Item 1 is therefore an orthogonal decomposition of $\Spc S=\Spc S_{\V p}\oplus\Spc N_{\V \phi}$.
Since $\Spc H$ is realized in $\Spc S'$, the two blocks of any $|\cdot|_{\Op A_+}$-Cauchy
sequence converge separately in $\Spc S'$---the first in $\Spc H_{\Op A}$, and the second in the
finite-dimensional, hence closed, space $\Spc N_{\V \phi}$---which gives
$\Spc H=\Spc H_{\Op A}\oplus\Spc N_{\V \phi}$.

On the dual side, the preamble sequences give $\Spc N_{\V p}\subset\Spc H'$: the sequence
$(\chi^{(n)}_k)$ is $|\cdot|_{\Op B_+}$-Cauchy, because $|\chi^{(n)}_k|_{\Op B}\to0$ and
$\V \phi(\chi^{(n)}_k)\to\V e_n$, and it converges to $p_n$ in $\Spc S'$, with
$\|p_n\|_{\Spc H'}=1$. Likewise, for $g\in\Spc S$, the corrected sequence
$g-\sum_n\V \phi(g)_n\,\chi^{(n)}_k$ converges to $\Op P_{\V \phi}^\ast g$ with
$\|\Op P_{\V \phi}^\ast g\|_{\Spc H'}=|g|_{\Op B}$; since $\Op P_{\V \phi}^\ast(\Spc S)$ is dense in
$\Spc H'_{\Op A}=\Op P_{\V \phi}^\ast(\dot{\Spc H}_{\Op B})$ and $\Spc H'$ is complete, we obtain
$\Spc H'_{\Op A}\subset\Spc H'$. Conversely, the seminorm $|f|_{\Op B}$ depends only on the
class $\pi f$, so that the decomposition $f=\Op P_{\V \phi}^\ast f+\V p^\Tr\V \phi(f)$ of
Lemma~\ref{Lem:Biorthogonal} gives
$\|f\|_{\Spc H'}^2=\|\Op P_{\V \phi}^\ast f\|_{\dot{\Spc H}_{\Op B}}^2+|\V \phi(f)|_2^2$ for every
$f\in\Spc S$, which extends by density and identifies $\Spc H'$ with the orthogonal direct sum
$\Spc H'_{\Op A}\oplus\Spc N_{\V p}$.

\medskip
\emph{Block-diagonal duality and exchange of the blocks}:
Let $h=u+\V \phi^\Tr\M c\in\Spc H$ and $f=v+\V p^\Tr\M d\in\Spc H'$ with
$(u,v)\in\Spc H_{\Op A}\times\Spc H'_{\Op A}$ and $\M c,\M d\in\R^{N_0}$. The extended pairing
splits into four terms: $\langle v,u\rangle$ is the
$\dot{\Spc H}_{\Op B}$--$\Spc H_{\Op A}$ duality of the conditional pair;
$\langle v,\V \phi^\Tr\M c\rangle=0$, because $v$ annihilates $\V \phi$;
$\langle\V p^\Tr\M d,u\rangle=0$, because $u$ annihilates $\V p$; and
$\langle\V p^\Tr\M d,\V \phi^\Tr\M c\rangle=\M d^\Tr\M c$, by biorthogonality. The pairing is
therefore block-diagonal. Finally, since $\Op P_{\V \phi}\phi_n=0$ and by biorthogonality,
$\Op A_+\phi_n=\Op A_{\V \phi}\phi_n+\Op R_{\V p}\phi_n=p_n$, while property (i) of the
preamble gives $\Op B_+p_n=\Op Bp_n+\Op R_{\V \phi}p_n=\phi_n$; on the primary block,
$\Op A_+u=\Op A_{\V \phi}u$ is the intrinsic isometry $\pi\Op A$ in its concrete realization
(Theorem~\ref{Thm:Concretization}). This completes the identification of the Riesz maps and
the proof.
\end{proof}

The augmented pair of Theorem~\ref{Thm:Augmentation} still refers to the biorthogonal family
$\V \phi$, through the complements $\Spc N_{\V \phi}$ and $\Spc H'_{\Op A}=\Op P_{\V \phi}^\ast(\dot{\Spc H}_{\Op B})$. The next result shows that this dependence is confined to the choice of the complements: the two Hilbert--Gelfand spaces themselves are determined by the conditional Riesz pair $(\Op A,\Op B)$ alone, up to an equivalence of norms.

\begin{corollary}[Independence from the biorthogonal family]
\label{Cor:PhiIndependence}
Under the assumptions of Theorem~\ref{Thm:Augmentation}, let $\V \phi,\tilde{\V \phi}\subset\Spc S$ be two families that are biorthogonal to $\V p$, and let $(\Spc H_{\V \phi},\Spc H'_{\V \phi})$ and $(\Spc H_{\tilde{\V \phi}},\Spc H'_{\tilde{\V \phi}})$ denote the corresponding augmented pairs, with norms $|\cdot|_{\Op A_+},|\cdot|_{\Op B_+}$ and $|\cdot|_{\tilde{\Op A}_+},|\cdot|_{\tilde{\Op B}_+}$, respectively. Then, the following properties hold.
\begin{enumerate}
\item \emph{Canonical description}: As subsets of $\Spc S'$,
\begin{align}
\label{Eq:CanonicalHprime}
\Spc H'_{\V \phi}=\pi^{-1}\big(\dot{\Spc H}_{\Op B}\big)=\{f\in\Spc S':\ \pi f\in\dot{\Spc H}_{\Op B}\},
\qquad
\Spc H_{\V \phi}=\Spc H_{\Op A}+\Spc S,
\end{align}
where $\dot{\Spc H}_{\Op B}$ is realized in $\Spc S'/\Spc N_{\V p}$ (Proposition~\ref{Prop:Realization}, Item 3). In particular, $\Spc H'_{\V \phi}=\Spc H'_{\tilde{\V \phi}}$ and $\Spc H_{\V \phi}=\Spc H_{\tilde{\V \phi}}$ as sets.
\item \emph{Equivalence of norms}: The norms $|\cdot|_{\Op B_+}$ and $|\cdot|_{\tilde{\Op B}_+}$ are equivalent on $\Spc H'_{\V \phi}$, and the norms $|\cdot|_{\Op A_+}$ and $|\cdot|_{\tilde{\Op A}_+}$ are equivalent on $\Spc H_{\V \phi}$.
\item \emph{Quotient structure}: For every $\V \phi$, the quotient map $\pi$ induces an isometric isomorphism from $\Spc H'_{\V \phi}/\Spc N_{\V p}$, equipped with the quotient norm $\inf_{q\in\Spc N_{\V p}}|f+q|_{\Op B_+}$, onto $\dot{\Spc H}_{\Op B}$. The orthogonal complement of $\Spc N_{\V p}$ in $\Spc H'_{\V \phi}$ is $\Spc H'_{\Op A}=\Spc S'_{\V \phi}\cap\Spc H'_{\V \phi}$; it is the only object that depends on $\V \phi$.
\end{enumerate}
\end{corollary}

\begin{proof}
\emph{Canonical description}:
Let $f=\Op P_{\V \phi}^\ast g+q\in\Spc H'_{\Op A}\oplus\Spc N_{\V p}$ with $g\in\dot{\Spc H}_{\Op B}$. Since $\Op P_{\V \phi}^\ast$ selects a representative of the class $g$ and since $\pi q=0$, we have that $\pi f=g\in\dot{\Spc H}_{\Op B}$. Conversely, let $f\in\Spc S'$ be such that $\pi f\in\dot{\Spc H}_{\Op B}$. Lemma~\ref{Lem:Biorthogonal} gives $f=\Op P_{\V \phi}^\ast f+\V p^\Tr\V \phi(f)$, where $\Op P_{\V \phi}^\ast f$ is the representative of the class $\pi f$ that annihilates $\V \phi$, so that $\Op P_{\V \phi}^\ast f\in\Op P_{\V \phi}^\ast(\dot{\Spc H}_{\Op B})=\Spc H'_{\Op A}$. Hence $f\in\Spc H'_{\V \phi}$, which proves the first identity. For the second one, the splitting $\Spc S=\Spc S_{\V p}\oplus\Spc N_{\V \phi}$ of Lemma~\ref{Lem:Biorthogonal} and the inclusion $\Spc S_{\V p}\subset\Spc H_{\Op A}$ give $\Spc H_{\Op A}+\Spc S=\Spc H_{\Op A}+\Spc N_{\V \phi}=\Spc H_{\V \phi}$. Neither description refers to $\V \phi$.

\medskip
\emph{Equivalence of norms}:
Item 3 of Theorem~\ref{Thm:Augmentation} gives $|f|_{\tilde{\Op B}_+}^2=\|\pi f\|_{\dot{\Spc H}_{\Op B}}^2+|\tilde{\V \phi}(f)|_2^2$ for every $f\in\Spc H'_{\V \phi}$. Since $\tilde\phi_n\in\Spc S$, the functional $f\mapsto\langle f,\tilde\phi_n\rangle$ is continuous on $\Spc S'$ and, therefore, on $\Spc H'_{\V \phi}$, which is continuously embedded in $\Spc S'$. Consequently, there is a constant $C$ such that $|\tilde{\V \phi}(f)|_2\le C|f|_{\Op B_+}$, which yields $|f|_{\tilde{\Op B}_+}\le\sqrt{1+C^2}\,|f|_{\Op B_+}$. The reverse inequality follows by exchanging the roles of $\V \phi$ and $\tilde{\V \phi}$. On the dual side, let $h=u+\V \phi^\Tr\M c=\tilde u+\tilde{\V \phi}^\Tr\tilde{\M c}$ be the two decompositions of $h\in\Spc H_{\V \phi}$, with $u,\tilde u\in\Spc H_{\Op A}$ and $\M c,\tilde{\M c}\in\R^{N_0}$, so that $|h|_{\Op A_+}^2=\|u\|_{\Spc H_{\Op A}}^2+|\M c|_2^2$ and $|h|_{\tilde{\Op A}_+}^2=\|\tilde u\|_{\Spc H_{\Op A}}^2+|\tilde{\M c}|_2^2$. The test function $\V \phi^\Tr\M c-\tilde{\V \phi}^\Tr\tilde{\M c}=\tilde u-u$ lies in $\Spc S\cap\Spc H_{\Op A}=\Spc S_{\V p}$ (a test function $\varphi\in\Spc H_{\Op A}$ satisfies $|\langle\varphi,\psi\rangle|\le\|\varphi\|_{\Spc H_{\Op A}}|\psi|_{\Op B}$ by Proposition~\ref{Prop:Realization}, so that $\varphi\in\Spc N_{\Op B}^\perp=\Spc S_{\V p}$ by the Bridge in the proof of Theorem~\ref{Thm:AttainableNull}), so that the application of $\V p$ gives $\tilde{\M c}=\M c$. Hence $\tilde u=u+\sum_n c_n(\phi_n-\tilde\phi_n)$ with $\phi_n-\tilde\phi_n\in\Spc S_{\V p}\subset\Spc H_{\Op A}$, and the triangle inequality yields $|h|_{\tilde{\Op A}_+}^2\le\big(2+2\sum_n\|\phi_n-\tilde\phi_n\|_{\Spc H_{\Op A}}^2\big)|h|_{\Op A_+}^2$. The reverse inequality follows by symmetry.

\medskip
\emph{Quotient structure}:
Since the direct sum $\Spc H'_{\V \phi}=\Spc H'_{\Op A}\oplus\Spc N_{\V p}$ is orthogonal, the distance from $f$ to $\Spc N_{\V p}$ is the norm of its component in $\Spc H'_{\Op A}$, namely $\|\Op P_{\V \phi}^\ast f\|_{\Spc H'_{\Op A}}=\|\pi f\|_{\dot{\Spc H}_{\Op B}}$ by Item 1 of Theorem~\ref{Thm:Concretization}. The map $\pi$ is onto $\dot{\Spc H}_{\Op B}$ by \eqref{Eq:CanonicalHprime} and has null space $\Spc N_{\V p}$, which proves the isometry. The identification of the orthogonal complement is Item 3 of Theorem~\ref{Thm:Augmentation} together with $\Spc H'_{\Op A}\subset\Spc S'_{\V \phi}$.
\end{proof}

\begin{remark}[What the conditional Riesz pair determines]
\label{Rem:RigidityAugmented}
Corollary~\ref{Cor:PhiIndependence} complements the rigidity statement at the end of Section~\ref{Sec:AbstractFormulation}: the conditioning $\V p$ is dictated by $\Op B$, and the resulting Hilbert--Gelfand spaces $\Spc H=\Spc H_{\Op A}+\Spc S$ and $\Spc H'=\pi^{-1}(\dot{\Spc H}_{\Op B})$ are dictated by $(\Op A,\Op B)$ up to an equivalence of norms. The norm itself is not determined by $(\Op A,\Op B)$, nor is it determined by the sets: every continuous symmetric positive-semi-definite operator $\Op C:\Spc S\to\Spc S'$ whose quadratic form is bounded on $\Spc H'$ and positive-definite on $\Spc N_{\V p}$ yields, through $\Op B+\Op C$, the same set $\Spc H'$ with an equivalent norm but with another Riesz pair. What singles out the augmentations $\Op B_+=\Op B+\Op R_{\V \phi}$ among these is Item 3 of the corollary: the quotient map $\pi$ is a partial isometry, in the sense that $|f|_{\Op B_+}=\|\pi f\|_{\dot{\Spc H}_{\Op B}}$ on the orthogonal complement of $\Spc N_{\V p}$. This property is investigated further in Proposition~\ref{Prop:NativeUniqueness}, in the operator-based setting of Section~\ref{Sec:OperatorConstructions}. Let us also note that Corollary~\ref{Cor:PhiIndependence} extends verbatim to biorthogonal families $\tilde{\V\phi}$ taken in the dual $\Spc H$ rather than in $\Spc S$, since its proof only uses the continuity of $f\mapsto\langle f,\tilde\phi_n\rangle$ on $\Spc H'$.
\end{remark}

\subsection{Example Linked to Linear Splines}
\label{Sec:OpD}
\begin{example}[$\Op B=\Op D^\ast\Op D$ on $\Spc S(\R)$]
\label{Ex:LinearSpline}
Let $(\Spc S,\Spc S')=(\Spc S(\R),\Spc S'(\R))$ and $\Op D=\tfrac{\dint}{\dint x}$, so that
$\Op D^\ast=-\Op D$ and
\begin{align*}
\Op B=\Op D^\ast\Op D=-\Op D^2,\qquad \widehat{\Op B\varphi}(\omega)=\omega^2\,\hat\varphi(\omega),
\end{align*}
where $\hat\varphi(\omega)=\int_\R\varphi(x)\,\ee^{-\jj\omega x}\,\dint x$. The operator $\Op B$ is
symmetric and positive-definite on $\Spc S(\R)$, because
$\langle\Op B\varphi,\varphi\rangle=\|\Op D\varphi\|_{L_2}^2>0$ for $\varphi\neq0$. Its
distributional null space $\ker(\Op B:\Spc S'\to\Spc S')=\{a+bx:a,b\in\R\}$ is the space of affine
functions, but only the constants are attainable in the energy: the relevant null space is
$\Spc N_{\Op B}=\Spc N_{\V p}=\mathrm{span}(p_1)$, with
$p_1\equiv1$ and $N_0=1$. Since
$\langle p_1,\varphi\rangle=\int_\R\varphi(x)\dint x=\hat\varphi(0)$,
\begin{align*}
\Spc S_{\V p}=\Big\{\varphi\in\Spc S(\R):\ \int_\R\varphi(x)\dint x=0\Big\}.
\end{align*}

\medskip
\emph{Biorthogonal function and projector}:
Among the infinitely many admissible choices of $\phi_1\in\Spc S$, we select the normalized Gaussian
\begin{align*}
\phi_1(x)=\frac{1}{\sqrt{2\pi}}\,\ee^{-x^2/2}\in\Spc S(\R),
\qquad \langle p_1,\phi_1\rangle=\int_\R\phi_1(x)\,\dint x=1.
\end{align*}
The corresponding
projector of $\Spc S(\R)$ onto $\Spc S_{\V p}$ along
$\Spc N_{\V \phi}=\mathrm{span}(\phi_1)$ is
\begin{align*}
\Op P_{\V \phi}=\Identity-\phi_1\,\langle p_1,\cdot\rangle:\ \varphi\mapsto\varphi-\phi_1\int_\R\varphi(x)\,\dint x,
\end{align*}
with adjoint $\Op P_{\V \phi}^\ast=\Identity-p_1\,\langle\cdot,\phi_1\rangle$. The construction below does not
depend on this particular choice.

\medskip
\emph{Conditional inverse}:
Let $\Op A:\varphi\mapsto g\ast\varphi$ with $g(x)=-\tfrac12|x|=\Spc F^{-1}\{\frac{1}{\omega^2}\}(x)$ be the symmetric convolution inverse of
$\Op B$ that solves $\Op Bg=\delta$; it is such that $\Op A\Op B=\Op B\Op A=\Identity$
on $\Spc S(\R)$. The fundamental solution $g$ is determined only up to an affine function,
which is why we replace $\Op A$ by the conditioned operator
$\Op A_{\V \phi}=\Op P_{\V \phi}^\ast\,\Op A\,\Op P_{\V \phi}$. The latter is insensitive to
that ambiguity and satisfies
\begin{align*}
\Op B\,\Op A_{\V \phi}=\Op P_{\V \phi},\qquad \Op A_{\V \phi}\,\Op B=\Op P_{\V \phi}^\ast,\qquad
\ker\Op A_{\V \phi}=\Spc N_{\V \phi}=\mathrm{span}(\phi_1).
\end{align*}

\medskip
\emph{Conditional positivity}:
Since  $\Op P_{\V \phi}\varphi=\varphi$ for $\varphi\in\Spc S_{\V p}$, the induced quadratic form is that of $\Op A$:
\begin{align*}
\langle\Op A_{\V \phi}\varphi,\varphi\rangle
=\langle\Op A\varphi,\varphi\rangle
=-\tfrac12\iint_{\R^2}|x-y|\,\varphi(x)\varphi(y)\,\dint x\,\dint y
=\frac{1}{2\pi}\int_\R\frac{|\hat\varphi(\omega)|^2}{\omega^2}\,\dint\omega,
\end{align*}
where the last (Fourier) expression is finite and strictly positive on $\Spc S_{\V p}$. The spatial
form applies to every $\varphi\in\Spc S(\R)$ and is strictly negative whenever $\varphi\ge0$ carries
a nonzero mass. Hence $\Op A_{\V \phi}$ and $\Op A$ are $\V p$-conditionally positive-definite but
not positive-definite, whereas $\Op B$ is positive-definite.

\medskip
\emph{Hilbert--Gelfand spaces and Riesz isometry}:
With $|\varphi|_{\Op A}^2=\tfrac{1}{2\pi}\int_\R\frac{|\hat\varphi(\omega)|^2}{\omega^2}\,\dint\omega$ and
$|f|_{\Op B}^2=\|f'\|_{L_2}^2$,
\begin{align*}
\Spc H_{\Op A}=\overline{(\Spc S_{\V p},|\cdot|_{\Op A})}=\dot H^{-1}(\R)\subset\Spc S'(\R),
\qquad
\dot{\Spc H}_{\Op B}=\overline{(\Spc S(\R),|\cdot|_{\Op B})}=\dot H^{1}(\R).
\end{align*}
The space $\Spc H_{\Op A}$ can be identified as the space of distributions with mean zero and Sobolev regularity $-1$. The space $\dot{\Spc H}_{\Op B}=\dot H^1(\R)$ is a space of
functions modulo constants: the energy $\|f'\|_{L_2}$ vanishes on $\Spc N_{\V p}\cong\R$, so that
$\dot{\Spc H}_{\Op B}$ is realized in the quotient $\Spc S'(\R)/\R$
rather
than as a subspace of $\Spc S'(\R)$. The two spaces are dual under the $\Spc S$--$\Spc S'$ pairing,
which is well defined because the mean-zero constraint on $\Spc H_{\Op A}$ makes it insensitive to
the added constant. The Riesz maps $\Op A:\dot H^{-1}(\R)\to\dot H^{1}(\R)$ and
$\Op B:\dot H^{1}(\R)\to\dot H^{-1}(\R)$ are isometric and mutually inverse, since, for
$\varphi\in\Spc S_{\V p}$,
\begin{align*}
|\Op A\varphi|_{\Op B}^2=\langle\Op B\Op A\varphi,\Op A\varphi\rangle=\langle\Op P_{\V \phi}\varphi,\Op A\varphi\rangle=\langle\varphi,\Op A\varphi\rangle=|\varphi|_{\Op A}^2 .
\end{align*}
Thus $g(x-y)=-\tfrac12|x-y|$ is the conditionally positive-definite kernel of the first-order spline
energy $\|f'\|_{L_2}^2$, and $(\Op A,\Op B)$ is a $\V p$-conditional Riesz pair in the sense of
Definition~\ref{Def:Intrinsic}.

\medskip
\emph{Augmentation}:
Theorem~\ref{Thm:Augmentation} produces the unconditioned pair
\begin{align*}
\Op A_+=\Op P_{\V \phi}^\ast\Op A\Op P_{\V \phi}+p_1\langle p_1,\cdot\rangle,
\qquad
\Op B_+=-\Op D^2+\phi_1\langle\cdot,\phi_1\rangle,
\end{align*}
with norms $|\varphi|_{\Op A_+}^2=|\Op P_{\V \phi}\varphi|_{\Op A}^2+\big(\int_\R\varphi(x)\,\dint x\big)^2$ and
$|f|_{\Op B_+}^2=\|f'\|_{L_2}^2+\langle f,\phi_1\rangle^2$. The induced dual pair of
Hilbert--Gelfand spaces is
\begin{align*}
\Spc H&=\dot H^{-1}(\R)\oplus\mathrm{span}(\phi_1),\\
\Spc H'&=\Spc H_\Op D=\{f\in\Spc S'(\R):f'\in L_2(\R)\}=\Spc H'_{\Op A}\oplus\mathrm{span}(1),
\end{align*}
and the exchange of the finite-dimensional blocks reads $\Op A_+\phi_1=p_1=1$ and $\Op B_+1=\phi_1$.
The resulting ``native'' space $\Spc H_\Op D=\overline{(\Spc S(\R),|\cdot|_{\Op D^\ast\Op D +\Op R_\V \phi})}$ is strictly larger than the classical Sobolev space $H^1(\R)$: it contains the constants and, more generally, every $f$ with $f'\in L_2(\R)$, whatever its growth (for instance, $\ln(1+|x|)$). The classical norm $\|f'\|_{L_2}^2+|f(0)|^2$ of this space is of the augmented form with $\tilde\phi_1=\delta$, which is not a test function but a continuous functional on $\Spc H_{\Op D}$; Proposition~\ref{Prop:NativeUniqueness} below explains in which sense the two norms define the same native space.
\end{example}

\section{Conditional Riesz Pairs Generated by an Operator}
\label{Sec:AdmissibleL}
Taking inspiration from the previous example, we now consider a general continuous and
injective operator $\Op L:\Spc S\to L_2$ and form the normal operator
$\Op B=\Op L^\ast\Op L$, with the conditional inverse $\Op A$ determined from the Green's
function $(\Op L^\ast\Op L)^{-1}$ of the normal operator. There are two routes. When a
Green inverse of $\Op B$ is available, Theorem~\ref{Thm:AttainableNull} applies
directly and reduces the admissibility of $\Op L$ to the correct delineation of the attainable
null space. In the general case, Theorem~\ref{Thm:NormalPair} below characterizes the
admissible operators $\Op L$ intrinsically, without any a priori knowledge of $\Op A$.

\begin{corollary}[Green inverse of the normal operator]
\label{Cor:GreenNormal}
Let $\V p\subset\Spc S'$ be a linearly independent family, let $\Op L:\Spc S\to L_2$ be a
continuous and injective linear operator, and let us assume
that the normal operator $\Op B=\Op L^\ast\Op L:\Spc S\to\Spc S'$ admits a
Green
inverse $\Op A=(\Op L^\ast\Op L)^{-1}:\Spc S\to\Spc S'$ through the Hilbertian extensions, in
the sense that $\Op A\Op B=\Identity$ on $\Spc S$ and $\Op B\Op A=\Identity$ on
$\Spc S_{\V p}$. Then, the following statements are equivalent.
\begin{enumerate}
\item[(i)] $(\Op A,\Op B)$ is a $\V p$-conditional Riesz pair for $(\Spc S,\Spc S')$.
\item[(ii)] The family $\V p$ spans the attainable null space of the spline energy:
$\Spc N_{\V p}=\Spc N_{\Op B}$.
\item[(iii)] The detrended space is the finiteness domain of the dual energy:
\begin{align*}
\Spc S_{\V p}
=\Big\{\varphi\in\Spc S:\ \sup_{\psi\in\Spc S\backslash\{0\}}
\frac{|\langle\varphi,\psi\rangle|}{\|\Op L\psi\|_{L_2}}<\infty\Big\}.
\end{align*}
\end{enumerate}
\end{corollary}

\begin{proof}
The operator $\Op B$ is symmetric and positive-definite with
$|\varphi|_{\Op B}=\|\Op L\varphi\|_{L_2}$, which is a genuine norm because $\Op L$ is
injective. The equivalence of (i) and (ii) is then Theorem~\ref{Thm:AttainableNull}. For the
equivalence with (iii), the Bridge in the proof of that theorem identifies the finiteness
domain of the dual energy with $\Spc N_{\Op B}^\perp$, so that (iii) reads
$\Spc N_{\V p}^\perp=\Spc N_{\Op B}^\perp$; the bipolar argument of Step
(iii)$\Rightarrow$(i) of the same proof then converts this identity into (ii).
\end{proof}

When no Green inverse is available a priori, the existence of an admissible conditional
inverse can still be characterized intrinsically, in terms of $\Op L$ alone, as shown in Theorem \ref{Thm:NormalPair}. The latter makes use of the following result.

\begin{lemma}[Riesz representation criterion]
\label{Lem:UnconditionedCriterion}
Let $(\Spc S,\Spc S')$ be a nuclear Gelfand pair, $\Spc K$ be a Hilbert space, and 
$\Op T:\Spc S\to\Spc K$ be a continuous and injective linear operator, with
$\Spc R_{\Op T}\eqdef\overline{\Op T(\Spc S)}^{\,\Spc K}$. Then, there exists a continuous
operator $\Op A:\Spc S\to\Spc S'$ such that $(\Op A,\Op T^\ast\Op T)$ is a compatible Riesz
pair for $(\Spc S,\Spc S')$ if and only if the two following conditions hold.
\begin{enumerate}
\item[(a)] \emph{(Representation)} For every $\varphi\in\Spc S$, there exists
$e_\varphi\in\Spc R_{\Op T}$ such that $\Op T^\ast e_\varphi=\iota\varphi$.
\item[(b)] \emph{(Sharpness)} The subspace $\Spc E\eqdef\{e_\varphi:\varphi\in\Spc S\}$ is
dense in $\Spc R_{\Op T}$.
\end{enumerate}
When they hold, the form of the (unique) admissible operator is
$\langle\Op A\varphi,\psi\rangle=\langle e_\varphi,e_\psi\rangle_{\Spc K}$, so that
$|\varphi|_{\Op A}=\|e_\varphi\|_{\Spc K}$, and the extension of $\Op T$ is an isometric
isomorphism $\widetilde{\Op T}:\Spc H_{\Op T^\ast\Op T}\to\Spc R_{\Op T}$.
\end{lemma}

\begin{proof}
We write $\Op B=\Op T^\ast\Op T$, so that $|\psi|_{\Op B}=\|\Op T\psi\|_{\Spc K}$ is a norm on
$\Spc S$ by the injectivity of $\Op T$, and we introduce the dual energy
\begin{align*}
Q(\varphi)\eqdef\sup_{\psi\in\Spc S\backslash\{0\}}
\frac{|\langle\varphi,\psi\rangle|^2}{|\psi|_{\Op B}^2}\in[0,+\infty].
\end{align*}

\medskip
\emph{Representers}:
The identity $\langle\Op T^\ast w,\psi\rangle=\langle w,\Op T\psi\rangle_{\Spc K}$ gives
$\ker(\Op T^\ast)=\Spc R_{\Op T}^\perp$, so that $\Op T^\ast$ is injective on
$\Spc R_{\Op T}$. For $\varphi\in\Spc S$ fixed, we then consider the linear functional 
\begin{align*}
\ell_\varphi:\Op T(\Spc S)\to\R,
\qquad
\ell_\varphi(\Op T\psi)\eqdef\langle\varphi,\psi\rangle ,
\end{align*}
which is well defined because $\Op T$ is injective. Its norm on
$\Op T(\Spc S)\subseteq\Spc R_{\Op T}$ is
\begin{align*}
\|\ell_\varphi\|
=\sup_{\psi\in\Spc S\backslash\{0\}}
\frac{|\langle\varphi,\psi\rangle|}{\|\Op T\psi\|_{\Spc K}}
=\sqrt{Q(\varphi)} ,
\end{align*}
so that $\ell_\varphi$ is continuous exactly when $Q(\varphi)<\infty$. In that case,
$\ell_\varphi$ extends uniquely  with the same norm to $\Spc R_{\Op T}$. Since $\Spc R_\Op T$ is a closed subspace
of $\Spc K$, it is itself a Hilbert space. The Riesz representation theorem on
$\Spc R_{\Op T}$ then provides a unique $e_\varphi\in\Spc R_{\Op T}$ such that
\begin{align*}
\ell_\varphi(v)=\langle e_\varphi,v\rangle_{\Spc K}
\quad\text{for all }v\in\Spc R_{\Op T},
\qquad
\|e_\varphi\|_{\Spc K}=\|\ell_\varphi\|=\sqrt{Q(\varphi)} .
\end{align*}
The choice $v=\Op T\psi$ gives
$\langle e_\varphi,\Op T\psi\rangle_{\Spc K}=\langle\varphi,\psi\rangle$ for every
$\psi\in\Spc S$, which is the identity $\Op T^\ast e_\varphi=\iota\varphi$. Conversely, the
Cauchy--Schwarz inequality applied to that identity gives
$Q(\varphi)\le\|e_\varphi\|^2_{\Spc K}<\infty$. We have therefore established the pointwise
criterion
\begin{align}
\label{Eq:RepresenterCriterion}
Q(\varphi)<\infty
\qquad\Longleftrightarrow\qquad
\exists\,e_\varphi\in\Spc R_{\Op T}:\ \Op T^\ast e_\varphi=\iota\varphi ,
\end{align}
in which case $e_\varphi$ is unique, $Q(\varphi)=\|e_\varphi\|^2_{\Spc K}$, and
\begin{align}
\label{Eq:Reproducing}
\langle\varphi,\psi\rangle=\langle e_\varphi,\Op T\psi\rangle_{\Spc K},
\qquad\psi\in\Spc S,
\end{align}
which we refer to as the reproducing identity. Equivalently,
\begin{align*}
e_\varphi=\big(\Op T^\ast|_{\Spc R_{\Op T}}\big)^{-1}\iota\varphi ,
\end{align*}
where the restriction to $\Spc R_{\Op T}$ is what makes $\Op T^\ast$ invertible. The criterion
\eqref{Eq:RepresenterCriterion} trades the computation of a supremum for the resolution of the
equation $\Op T^\ast e=\iota\varphi$ in $\Spc R_{\Op T}$, the divergent case included: no
solution means $Q(\varphi)=\infty$. Condition (a) requires
\eqref{Eq:RepresenterCriterion} to hold for every $\varphi\in\Spc S$.
The target $\Spc K$ enters only through $\Spc R_{\Op T}$, which is why no density assumption
on $\Op T(\Spc S)$ is needed.

\medskip
\emph{Continuity}:
Let us assume (a). The family $\{|\psi|_{\Op B}^{-1}\iota\psi\}_{\psi\neq0}$ is pointwise
bounded on $\Spc S$, which is barrelled, so that the Banach--Steinhaus theorem for barrelled
spaces \cite[Ch.~III]{Schaefer-Wolff1999} makes $\varphi\mapsto\sqrt{Q(\varphi)}$ a continuous
seminorm. The Gram form
$(\varphi,\psi)\mapsto\langle e_\varphi,e_\psi\rangle_{\Spc K}$ therefore specifies a
continuous operator $\Op A:\Spc S\to\Spc S'$, which is positive-definite because
$\Op T^\ast e_\varphi=\iota\varphi\neq0$ for $\varphi\neq0$.

\medskip
\emph{The two dual norms}:
The identity $|\varphi|_{\Op A}=\|e_\varphi\|_{\Spc K}=\sqrt{Q(\varphi)}$ is the second
identity of Corollary~\ref{Cor:AlgebraicInverse}(iii). For the first one, the reproducing
identity converts the supremum into the norm of an orthogonal projection,
\begin{align*}
\sup_{\varphi\in\Spc S\backslash\{0\}}
\frac{|\langle\varphi,\psi\rangle|}{|\varphi|_{\Op A}}
=\sup_{\varphi\neq0}
\frac{|\langle e_\varphi,\Op T\psi\rangle_{\Spc K}|}{\|e_\varphi\|_{\Spc K}}
=\big\|\Op Q_{\Spc E}\,\Op T\psi\big\|_{\Spc K},
\end{align*}
where $\Op Q_{\Spc E}$ denotes the orthogonal projector onto $\overline{\Spc E}$. This quantity
equals $|\psi|_{\Op B}=\|\Op T\psi\|_{\Spc K}$ for every $\psi\in\Spc S$ if and only if
$\Op T(\Spc S)\subseteq\overline{\Spc E}$; that is, if and only if (b) holds.

\medskip
\emph{Conclusion}:
Under (a) and (b), Corollary~\ref{Cor:AlgebraicInverse}(iii) is met, so that $(\Op A,\Op B)$
is a compatible Riesz pair, with $\Op A$ unique and symmetric by
Remark~\ref{Rem:AutomaticSymmetry}. Conversely, the same corollary applied to such a pair
supplies the two identities: the second gives $Q<\infty$ on $\Spc S$, hence (a), and the first
gives $\Op T(\Spc S)\subseteq\overline{\Spc E}$, hence (b). Finally, $\widetilde{\Op T}$ is the isometric extension of $\psi\mapsto\Op T\psi$ from the
dense subspace $\Spc S$ of $\Spc H_{\Op B}$, whose range is $\Spc R_{\Op T}$.
\end{proof}

\begin{remark}[Relation to the theory of Hilbertian subspaces]
\label{Rem:SchwartzHilbertian}
Lemma~\ref{Lem:UnconditionedCriterion} is a reformulation, in terms of the energy rather than
of the kernel, of the classical correspondence between the Hilbertian subspaces of a locally
convex space and the positive kernels that they induce \cite{Schwartz1964}. The same comment
applies to Theorem~\ref{Theo:HilbertPositiveDef}, which is the version of that correspondence
that is adapted to a nuclear Gelfand pair. The realization
$\Spc H_{\Op A}=\Op T^\ast(\Spc R_{\Op T})$ with the transported norm is the operator-range
construction of \cite{Fillmore-Williams1971}, while the reproducing-kernel counterpart is due
to Aronszajn \cite{Aronszajn1950}. What is specific to our formulation is that the two
conditions are stated on $\|\Op T\cdot\|_{\Spc K}$, which is the regularization energy of the
variational problem, and that they survive the conditioning. The classical correspondence, by
contrast, fails precisely when the energy is only a seminorm, which is the situation of the
Beppo--Levi spaces \cite{Deny-Lions1954}. The augmentation $\Op B_+=\Psi^\ast\Psi$ in the
proof of Theorem~\ref{Thm:NormalPair} is what restores its applicability.
\end{remark}

\begin{theorem}[$\V p$-conditional Riesz pairs generated by $\Op L$]
\label{Thm:NormalPair}
Let $\Op L:\Spc S\to L_2$ be a continuous and injective linear operator with adjoint
$\Op L^\ast:L_2\to\Spc S'$, and set
\begin{align*}
\Op B=\Op L^\ast\Op L:\Spc S\to\Spc S',
\qquad
\Spc R=\overline{\Op L(\Spc S)}^{\,L_2},
\end{align*}
so that $\Op B$ is symmetric and positive-definite with
$|\varphi|_{\Op B}=\|\Op L\varphi\|_{L_2}$, and
$\ker(\Op L^\ast)=\Spc R^{\perp}$. Then, there exists a continuous operator
$\Op A:\Spc S\to\Spc S'$ such that $(\Op A,\Op B)$ is a $\V p$-conditional Riesz pair for
$(\Spc S,\Spc S')$ (Definition~\ref{Def:Intrinsic}) if and only if the following three
conditions hold.
\begin{enumerate}
\item \emph{Attainable annihilation of $\Spc N_{\V p}$}:
 for every $n\in\{1,\dots,N_0\}$,
there exists a sequence $(\chi^{(n)}_k)\subset\Spc S$ such that
$\|\Op L\chi^{(n)}_k\|_{L_2}\to0$ and
$\lim_{k\to \infty}\chi^{(n)}_k=p_n$ in $\Spc S'$.
\item \emph{Riesz representation of the conditioned test functions:} for every
$\varphi\in\Spc S_{\V p}$, there exists a unique $e_\varphi\in\Spc R$ such that
$\Op L^\ast e_\varphi=\iota\varphi$; equivalently, there exists $C_\varphi<\infty$ such that
\begin{align*}
|\langle\psi,\varphi\rangle|\le C_\varphi\,\|\Op L\psi\|_{L_2},
\qquad\psi\in\Spc S.
\end{align*}
\item \emph{Sharpness:}
\begin{align*}
\|\Op L\psi\|_{L_2}
\le\sup_{\varphi\in\Spc S_{\V p}\backslash\{0\}}
\frac{|\langle\psi,\varphi\rangle|}{\|e_\varphi\|_{L_2}},
\qquad\psi\in\Spc S;
\end{align*}
equivalently, the subspace
$\Spc E\eqdef\{e_\varphi:\varphi\in\Spc S_{\V p}\}$ is dense in $\Spc R$.
\end{enumerate}
In this case, $\Spc N_{\Op B}=\Spc N_{\V p}$; the optimal constant in Item 2 is
$C_\varphi=\|e_\varphi\|_{L_2}=|\varphi|_{\Op A}$, where $|\cdot|_{\Op A}$ is the norm that is dual to $\|\Op L\cdot\|_{L_2}$; the form of every
admissible $\Op A$ is
$\langle\Op A\varphi,\psi\rangle=\langle e_\varphi,e_\psi\rangle_{L_2}$ on
$\Spc S_{\V p}\times\Spc S_{\V p}$; and the extension of $\psi\mapsto\Op L\psi$ is an isometric
isomorphism $\widetilde{\Op L}:\dot{\Spc H}_{\Op B}\to\Spc R$, so that
$\widetilde{\Op L}(\dot{\Spc H}_{\Op B})=\Spc R$ and $\Spc H_{\Op A}\cong\Op L^\ast(\Spc R)$. Under these
identifications, the representer map $\varphi\mapsto e_\varphi$ from $\Spc S_{\V p}$ to $\Spc R$
coincides with $\widetilde{\Op L}\circ\pi\Op A$; equivalently, the map
$\pi\Op A:\Spc S_{\V p}\to\Spc S'/\Spc N_{\V p}$, which extends to the isometry
$\Spc H_{\Op A}\to\dot{\Spc H}_{\Op B}$, is given by $\pi\Op A\varphi=\widetilde{\Op L}^{-1}e_\varphi$
(the representer $e_\varphi$ itself lives in the intermediate space $\Spc R$).
\end{theorem}

\begin{proof}
Let us fix a biorthogonal family $\V \phi\subset\Spc S$ with projectors $\Op P_{\V \phi}$ and
$\Op P_{\V \phi}^\ast$ (Lemma~\ref{Lem:Biorthogonal}). The strategy is to remove the
conditioning by augmentation, so as to reduce the statement to
Lemma~\ref{Lem:UnconditionedCriterion}.

\medskip
\emph{Augmented factorization}:
We set $\Spc K_+\eqdef L_2\oplus\R^{N_0}$ and
\begin{align*}
\Psi:\Spc S\to\Spc K_+,
\qquad
\Psi\varphi\eqdef\big(\Op L\varphi,\V \phi(\varphi)\big),
\end{align*}
which is continuous, and injective because $\Op L$ is. Its adjoint is
$\Psi^\ast(w,\M a)=\Op L^\ast w+\iota\big(\V \phi^\Tr\M a\big)$, so that
\begin{align*}
\Psi^\ast\Psi=\Op L^\ast\Op L+\Op R_{\V \phi}=\Op B_+,
\qquad
|\varphi|_{\Op B_+}=\|\Psi\varphi\|_{\Spc K_+} ,
\end{align*}
in the notation of Theorem~\ref{Thm:Augmentation}. The augmented normal operator is therefore
itself a normal operator, for an operator that carries no conditioning. We write
$\Spc R_\Psi=\overline{\Psi(\Spc S)}^{\,\Spc K_+}$.

\medskip
\emph{Conditions 1--3 are the hypotheses of Lemma~\ref{Lem:UnconditionedCriterion} for $\Psi$}:
Let us first observe that $\Psi^\ast(\V 0,\V e_n)=\iota\phi_n$, which represents $\phi_n$ at no
cost, and that $\Psi^\ast(e_\varphi,\V 0)=\Op L^\ast e_\varphi=\iota\varphi$ for
$\varphi\in\Spc S_{\V p}$ under Condition 2. Since
$\Spc S=\Spc S_{\V p}\oplus\Spc N_{\V \phi}$, Condition 2 is therefore equivalent to
Condition (a) for $\Psi$, with
\begin{align*}
e^+_\varphi=\big(e_{\Op P_{\V \phi}\varphi},\ \V p(\varphi)\big),
\qquad
\Spc E_+=\Spc E\oplus\R^{N_0} .
\end{align*}
Next, the biorthogonality gives $\V \phi(p_n)=\V e_n$, so that the sequences of Condition 1
satisfy $\Psi\chi^{(n)}_k\to(\V 0,\V e_n)$, which places $\{\V 0\}\oplus\R^{N_0}$ in
$\Spc R_\Psi$ and yields $\Spc R_\Psi=\Spc R\oplus\R^{N_0}$. Conversely, let us assume
Condition 2 together with $(\V 0,\V e_n)\in\Spc R_\Psi$, and let $(\chi_k)\subset\Spc S$ be such
that $\|\Op L\chi_k\|_{L_2}\to0$ and $\V \phi(\chi_k)\to\V e_n$. The decomposition
$\psi=\Op P_{\V \phi}\psi+\V \phi^\Tr\V p(\psi)$ and the reproducing identity
\eqref{Eq:Reproducing} give
\begin{align*}
\langle\chi_k,\psi\rangle
=\langle e_{\Op P_{\V \phi}\psi},\Op L\chi_k\rangle_{L_2}
+\V p(\psi)^\Tr\V \phi(\chi_k)
\longrightarrow
\langle p_n,\psi\rangle,
\qquad\psi\in\Spc S,
\end{align*}
so that $\chi_k\to p_n$ in $\Spc S'$. Condition 1 is thus equivalent to the splitting
$\Spc R_\Psi=\Spc R\oplus\R^{N_0}$: it states that the finite-dimensional block is genuinely
present in the concrete dual. Given this splitting, the density of $\Spc E_+$ in $\Spc R_\Psi$
reduces to the density of $\Spc E$ in $\Spc R$, so that Condition 3 is equivalent to
Condition (b) for $\Psi$.

\medskip
\emph{Augmented pair}:
Lemma~\ref{Lem:UnconditionedCriterion}, applied to $\Op T=\Psi$, provides a continuous
positive-definite operator $\Op A_+:\Spc S\to\Spc S'$ such that $(\Op A_+,\Op B_+)$ is a
compatible Riesz pair, with
\begin{align*}
\langle\Op A_+\varphi,\psi\rangle
=\langle e^+_\varphi,e^+_\psi\rangle_{\Spc K_+}
=\langle e_{\Op P_{\V \phi}\varphi},e_{\Op P_{\V \phi}\psi}\rangle_{L_2}
+\V p(\varphi)^\Tr\V p(\psi).
\end{align*}
This identifies $\Op A_+=\Op A_{\V \phi}+\Op R_{\V p}$ in the notation of
Theorem~\ref{Thm:Augmentation}, where the first term is the conditioned Gram operator
$\Op A_{\V \phi}$ of Theorem~\ref{Thm:Concretization}. By
Theorem~\ref{Theo:HilbertPositiveDef}, the completions
$\Spc H=\overline{(\Spc S,|\cdot|_{\Op A_+})}$ and
$\Spc H'=\overline{(\Spc S,|\cdot|_{\Op B_+})}$ are mutually dual Hilbert--Gelfand spaces whose
Riesz maps are the extensions of $\Op A_+$ and $\Op B_+$. All the completion arguments are
inherited from that theorem, and none of them needs to be repeated here.

\medskip
\emph{Descent to $(\Op A_{\V \phi},\Op B)$}:
The block form of the representers gives $\Op A_+\phi_n=p_n$, so that the isometry
$\Op A_+:\Spc H\to\Spc H'$ maps $\Spc N_{\V \phi}$ onto $\Spc N_{\V p}$. It therefore maps the
orthogonal complements onto one another,
\begin{align*}
\Op A_+:\ \Spc H_{\Op A}=\Spc H\ominus\Spc N_{\V \phi}
\ \longrightarrow\
\Spc H'_{\Op A}=\Spc H'\ominus\Spc N_{\V p},
\end{align*}
isometrically and onto. The elements of $\Spc H'_{\Op A}$ annihilate $\V \phi$, so that the norm
of $\Spc H'$ reduces there to $|\cdot|_{\Op B}$, which identifies $\Spc H'_{\Op A}$ with
$\dot{\Spc H}_{\Op B}$ through $\Op P_{\V \phi}^\ast$ (Theorem~\ref{Thm:Concretization},
Item 1). The restriction of $\Op A_+$ to the primary blocks is thus the isometry
$\pi\Op A:\Spc H_{\Op A}\to\dot{\Spc H}_{\Op B}$ of Definition~\ref{Def:Intrinsic}(ii), with
inverse the extension of $\Op B$, while Clause~(i) of that definition is the equivalence
established in the second step. Hence $(\Op A_{\V \phi},\Op B)$ is a $\V p$-conditional Riesz
pair, and Theorem~\ref{Thm:AttainableNull} adds that the conditioning is minimal. The
identifications of the statement are those of Lemma~\ref{Lem:UnconditionedCriterion}, read
through this splitting: $C_\varphi=\|e_\varphi\|_{L_2}=|\varphi|_{\Op A}$, the map
$\widetilde{\Op L}$ is the isometry $\dot{\Spc H}_{\Op B}\to\Spc R$, and
$\pi\Op A\varphi=\widetilde{\Op L}^{-1}e_\varphi$.
\end{proof}

\begin{remark}
\label{Rem:NaiveAnnihilation}
Condition 1 is the attainable form of the requirement $\Op L(\Spc N_{\V p})=\{0\}$. The naive
form of this requirement---that the distributional (Fourier-multiplier) extension of $\Op L$
to $\Spc S'$ annihilates $\Spc N_{\V p}$---is necessary but not sufficient (see
Example~\ref{Ex:FracLaplacian}). For $N_0=0$, the theorem reduces to the characterization of
compatible Riesz pairs generated through $L_2$.
\end{remark}

\begin{remark}
\label{Rem:SharpnessAsDuality}
Since $\ker(\Op L^\ast)=\Spc R^{\perp}$, Condition 2 is the range inclusion
$\iota(\Spc S_{\V p})\subseteq\Op L^\ast(\Spc R)$, with the representer recovered as
$e_\varphi=\big(\Op L^\ast|_{\Spc R}\big)^{-1}\iota\varphi$; this is the stable inverse
${\Op L^\ast}^{-1}$ that appears in Theorem~\ref{Thm:PracticalChecklist}. Condition 3 is the
sharpness of the resulting bound, and it is equivalent to the density of $\Spc E$ in
$\Spc R$. In practice, only the stated inequality of Condition 3 needs to be verified, because
the reproducing identity \eqref{Eq:Reproducing} yields the reverse one by the
Cauchy--Schwarz inequality. The duality that underlies the two conditions is discussed after
Theorem~\ref{Thm:PracticalChecklist}.
\end{remark}

\begin{corollary}[$L_2$-realization of the conditional pair]
\label{Cor:L2RealizationConditional}
Let the conditions of Theorem~\ref{Thm:NormalPair} hold. Then, the extension
$\Op T:\Spc H_{\Op A}\to\Spc R$ of $\varphi\mapsto e_\varphi$ and the extension
$\widetilde{\Op L}:\dot{\Spc H}_{\Op B}\to\Spc R$ of $\psi\mapsto\Op L\psi$ are isometric isomorphisms
onto $\Spc R=\overline{\Op L(\Spc S)}^{\,L_2}$. In particular, the following statements are
equivalent:
\begin{enumerate}
\item $\ker(\Op L^\ast)=\{0\}$;
\item $\Op L(\Spc S)$ is dense in $L_2$;
\item $\Op T:\Spc H_{\Op A}\to L_2$ is an isometric isomorphism;
\item $\widetilde{\Op L}:\dot{\Spc H}_{\Op B}\to L_2$ is an isometric isomorphism.
\end{enumerate}
In this case, the inverse of $\Op T$ is the adjoint $\Op L^\ast:L_2\to\Spc S'$, which maps
$L_2$ isometrically onto the realized space $\Spc H_{\Op A}=\Op L^\ast(L_2)$, with
$\|\Op L^\ast w\|_{\Spc H_{\Op A}}=\|w\|_{L_2}$ for every $w\in L_2$. The same criterion governs
the concrete realization
$\Spc H'_{\Op A}=\Op P_{\V \phi}^\ast\big(\dot{\Spc H}_{\Op B}\big)$ of
Theorem~\ref{Thm:Concretization}: since $\Op P_{\V \phi}^\ast$ is an isometric identification of
$\dot{\Spc H}_{\Op B}$ (Theorem~\ref{Thm:Concretization}, Item 1), the composition
$\Op P_{\V \phi}^\ast\widetilde{\Op L}^{-1}:L_2\to\Spc H'_{\Op A}$ is an isometric isomorphism if and only if
Items 1--4 hold. Correspondingly, we also have $L_2=\widetilde{\Op L}(\dot{\Spc H}_{\Op B})$
with $\|\Op P_{\V \phi}^\ast f\|_{\Spc H'_{\Op A}}=\|f\|_{\dot{\Spc H}_{\Op B}}=\|\widetilde{\Op L} f\|_{L_2}$ for every $f \in \dot{\Spc H}_{\Op B}$.
\end{corollary}

\begin{proof}
By the last step of the proof of Theorem~\ref{Thm:NormalPair}, $\Op T$ is the restriction of
$\Op A_+$ to the primary blocks, which is an isometry onto $\Spc R$, so that
$\Op T(\Spc H_{\Op A})=\Spc R$; likewise, $\widetilde{\Op L}$ is the isometric extension of
$\psi\mapsto\Op L\psi$ from the dense subspace $\Spc S$ of $\dot{\Spc H}_{\Op B}$, with range
$\Spc R$. Since $\ker(\Op L^\ast)=\Spc R^{\perp}$, each of Items 1--4 is
equivalent to $\Spc R=L_2$. The assertion on $\Op T^{-1}$ restates the realization
$\Spc H_{\Op A}\cong\Op L^\ast(\Spc R)$ of Theorem~\ref{Thm:NormalPair} with $\Spc R=L_2$,
where the injectivity of $\Op L^\ast$ guarantees that the transported norm is well-defined.
The statement on $\Spc H'_{\Op A}$ follows by composing the two isometries.
\end{proof}

We note that the conditioning is invisible at the intermediate level of $\Spc R=L_2$: the family $\V p$ determines
which distributions realize $\Spc H_{\Op A}$ (an annihilator) and $\dot{\Spc H}_{\Op B}$ (a
quotient, or its concrete representative $\Spc H'_{\Op A}$) inside $\Spc S'$, but not their
common $L_2$-model $\Spc R$. In the fractional-Laplacian setting of
Example~\ref{Ex:FracLaplacian}, the symbol $\|\V\omega\|^{\gamma}$ vanishes only on a null
set, so that $\ker(\Op L^\ast)=\{0\}$ and all of the above spaces are canonically
isometrically isomorphic to $L_2$.
\section{Concrete Operator-Based Constructions}
\label{Sec:OperatorConstructions}
\subsection{A Practical Checklist}
\begin{theorem}[Characterization of $L_2$-realizable conditional pairs]
\label{Thm:PracticalChecklist}
Let $\Op L:\Spc S\to L_2$ be a continuous linear operator, and let
\begin{align*}
Q(\varphi)
\eqdef
\sup_{\psi\in\Spc S\backslash\{0\}}
\frac{|\langle\varphi,\psi\rangle|^2}{\|\Op L\psi\|_{L_2}^2}
\in[0,+\infty],
\qquad
\varphi\in\Spc S,
\end{align*}
denote the dual energy. The following four conditions are each necessary and, together,
sufficient for the existence of a continuous operator $\Op A:\Spc S\to\Spc S'$ such that
$(\Op A,\Op L^\ast\Op L)$ is a $\V p$-conditional Riesz pair for $(\Spc S,\Spc S')$ whose
canonical maps $\Op T:\Spc H_{\Op A}\to L_2$ and $\widetilde{\Op L}:\dot{\Spc H}_{\Op B}\to L_2$ are
isometric isomorphisms.
\begin{enumerate}
\item \emph{(Injectivity)} $\Op L$ is injective on $\Spc S$.
\item \emph{(Dense range)} $\Op L(\Spc S)$ is dense in $L_2$; equivalently,
$\ker(\Op L^\ast)=\{0\}$.
\item \emph{(Strict $\V p$-conditional positivity)}
\begin{align*}
Q(\varphi)<\infty
\quad\Longleftrightarrow\quad
\varphi\in\Spc S_{\V p}.
\end{align*}
\item \emph{(Sharpness)}
\begin{align*}
\|\Op L\psi\|_{L_2}
\le\sup_{\varphi\in\Spc S_{\V p}\backslash\{0\}}
\frac{|\langle\psi,\varphi\rangle|}{\sqrt{Q(\varphi)}},
\qquad
\psi\in\Spc S;
\end{align*}
equivalently, the representers
$e_\varphi\eqdef{\Op L^\ast}^{-1}\iota\varphi$, which are well-defined for
$\varphi\in\Spc S_{\V p}$ by Conditions 2--3 and satisfy
$\|e_\varphi\|_{L_2}^2=Q(\varphi)$, span a dense subspace of $L_2$.
\end{enumerate}
When Conditions 1--4 hold, the following identifications are valid:
$\Spc N_{\Op B}=\Spc N_{\V p}$; the quadratic form of every admissible $\Op A$ is $Q$ on
$\Spc S_{\V p}$, and the conditioned operator that is built from a Green-function inverse of
$\Op L^\ast$ is admissible (Corollary~\ref{Cor:GreenFactor}); the stable inverse
${\Op L^\ast}^{-1}$ maps the realized space $\Spc H_{\Op A}=\Op L^\ast(L_2)$ isometrically
onto $L_2$ and retrieves the representers, $e_\varphi={\Op L^\ast}^{-1}\iota\varphi$; and the
concrete realization $\Spc H'_{\Op A}=\Op P_{\V \phi}^\ast(\dot{\Spc H}_{\Op B})$ is isometrically
isomorphic to $L_2$ through $\Op P_{\V \phi}^\ast\widetilde{\Op L}^{-1}$.
\end{theorem}

\begin{proof}
Throughout the proof, we set $\Spc R=\overline{\Op L(\Spc S)}^{\,L_2}$ and
$\Spc D\eqdef\{\varphi\in\Spc S:\ Q(\varphi)<\infty\}$. When $\Op L$ is injective, the
\emph{Representers} step of the proof of Lemma~\ref{Lem:UnconditionedCriterion}, applied to
$\Op T=\Op L$, states that $Q(\varphi)<\infty$ if and only if there exists a unique $e_\varphi\in\Spc R$ with
$\Op L^\ast e_\varphi=\iota\varphi$, in which case $Q(\varphi)=\|e_\varphi\|_{L_2}^2$. Hence $\Spc D$ is the subspace of test functions that
are representable through $\Op L^\ast$, and Condition 3 splits into the representation
property $\Spc S_{\V p}\subseteq\Spc D$ and the minimality property
$\Spc D\subseteq\Spc S_{\V p}$.

\medskip
\emph{Step 1 (Minimality is equivalent to attainability)}:
We claim that, if $\Op L$ is injective and $\Spc S_{\V p}\subseteq\Spc D$, then
\begin{align*}
\Spc N_{\V p}\subseteq\Spc N_{\Op B}
\qquad\Longleftrightarrow\qquad
\Spc D=\Spc S_{\V p}.
\end{align*}
Let us first assume that $\Spc D=\Spc S_{\V p}$, and let $\Psi$ be the map of the proof of
Theorem~\ref{Thm:NormalPair}. A pair $(w,\M a)\in L_2\oplus\R^{N_0}$ annihilates the range
of $\Psi$ if and only if
$\langle w,\Op L\varphi\rangle_{L_2}+\M a^\Tr\V \phi(\varphi)=0$ for every
$\varphi\in\Spc S$; that is, $\Op L^\ast w=-\iota\big(\V \phi^\Tr\M a\big)$. This
forces $\V \phi^\Tr\M a\in\Spc D=\Spc S_{\V p}$, and the decomposition
$\Spc S=\Spc S_{\V p}\oplus\Spc N_{\V \phi}$ then gives $\M a=\V 0$. The annihilator of
$\Psi(\Spc S)$ is therefore contained in $L_2\oplus\{\V 0\}$, so that
$\Spc R_\Psi\supseteq\{\V 0\}\oplus\R^{N_0}$. By the second step of the proof of
Theorem~\ref{Thm:NormalPair}, that splitting is equivalent to Condition 1 of that theorem,
which gives $p_n\in\Spc N_{\Op B}$ for every $n$.

Conversely, let us assume that $\Spc N_{\V p}\subseteq\Spc N_{\Op B}$ and let
$\varphi\in\Spc D$. For each $n$, we pick $(\chi_k)\subset\Spc S$ with
$\|\Op L\chi_k\|_{L_2}\to0$ and $\chi_k\to p_n$ in $\Spc S'$; then
\begin{align*}
\langle p_n,\varphi\rangle
=\lim_k\langle\chi_k,\varphi\rangle
=\lim_k\langle\Op L\chi_k,e_\varphi\rangle_{L_2}
=0,
\end{align*}
so that $\varphi\in\Spc S_{\V p}$. Hence $\Spc D\subseteq\Spc S_{\V p}$, and equality follows
from the standing inclusion. This proves the claim.

\medskip
\emph{Step 2 (Sufficiency)}:
Let us assume Conditions 1--4. Condition 1 makes
$|\cdot|_{\Op B}=\|\Op L\cdot\|_{L_2}$ a norm, and Condition 2 gives $\Spc R=L_2$. Condition 3
provides the representation property on $\Spc S_{\V p}$ and, through Step 1, the attainable
annihilation of $\Spc N_{\V p}$; Condition 4 is the sharpness requirement. Hence the three
conditions of Theorem~\ref{Thm:NormalPair} are met, which yields an admissible $\Op A$, and
Corollary~\ref{Cor:L2RealizationConditional} upgrades the canonical isometries onto
$\Spc R=L_2$.

\medskip
\emph{Step 3 (Necessity)}:
Condition 1 is required for $\Op B=\Op L^\ast\Op L$ to be positive-definite, and Condition 2
is Item 2 of Corollary~\ref{Cor:L2RealizationConditional}. By
Theorem~\ref{Thm:NormalPair}, the representation property
($\Spc S_{\V p}\subseteq\Spc D$), the attainable annihilation, and the sharpness are all
necessary; Step 1 converts the attainable annihilation into
$\Spc D\subseteq\Spc S_{\V p}$, so that Condition 3 is necessary, and the sharpness is
Condition 4.

\medskip
\emph{Step 4 (Identifications)}:
These restate the identifications of Theorem~\ref{Thm:NormalPair} and
Corollary~\ref{Cor:L2RealizationConditional}, combined with
$Q(\varphi)=\|e_\varphi\|_{L_2}^2$ on $\Spc S_{\V p}$.
\end{proof}

Conditions 3 and 4 are the two halves of a single duality, expressed through the representers
$e_\varphi={\Op L^\ast}^{-1}\iota\varphi$ and the reproducing identity
\eqref{Eq:Reproducing}. Condition 3 asserts the dual-norm identity
\begin{align}
\sqrt{Q(\varphi)}
=\|e_\varphi\|_{L_2}
=\sup_{\psi\in\Spc S\backslash\{0\}}
\frac{|\langle\psi,\varphi\rangle|}{\|\Op L\psi\|_{L_2}},\label{Eq:DualNormEq}
\end{align}
together with its exact finiteness domain---the operator-generated form of
Corollary~\ref{Cor:GreenNormal}(iii)---and it thereby calibrates $\Spc N_{\V p}$ to the energy.
If $\V p$ is \emph{too small}, then $Q$ blows up somewhere inside $\Spc S_{\V p}$ and
$\Spc H_{\Op A}$ cannot form; if $\V p$ is \emph{too large}, then $Q$ remains finite on some
$\varphi\notin\Spc S_{\V p}$, which contradicts the attainability of $\Spc N_{\V p}$, as
happens for the Laplacian in dimension $3$ with the affine conditioning
(Example~\ref{Ex:FracLaplacian}). Both defects are detectable by inspection of $Q$, which makes
Condition 3 preferable to the attainable annihilation of Theorem~\ref{Thm:NormalPair}.

Condition 4, together with the automatic reverse bound of
Remark~\ref{Rem:SharpnessAsDuality}, asserts the \emph{bidual} identity
\begin{align}
\|\Op L\psi\|_{L_2}
=\sup_{\varphi\in\Spc S_{\V p}\backslash\{0\}}
\frac{|\langle\psi,\varphi\rangle|}{\sqrt{Q(\varphi)}},\label{Eq:Bidual}
\end{align}
which states that dualizing twice returns the energy. It fails when some nonzero $w\in L_2$ is
orthogonal to every representer---a \emph{ghost}, which is a limit of data $\Op L\psi_k$ whose
sources $\psi_k$ vanish in $\Spc S'$, and which would cost the realization of
$\dot{\Spc H}_{\Op B}$ its injectivity. Only the two conditions together make
$\|\Op L\cdot\|_{L_2}$ and $\sqrt{Q}$ a perfect dual pair of norms. In the Fourier-multiplier
case treated below, $Q$ is an explicit integral, so that all these checks are computations.

An important simplification occurs
for the class of regular operators on $\Spc S$, which is closed under
adjunction and composition.
\begin{definition}[Regular operator]
\label{Def:Regular}
A continuous operator $\Op T:\Spc S'\to\Spc S'$ is said to be \emph{regular} on $\Spc S$ if it
maps the test functions among themselves; that is, if $\Op T|_{\Spc S}:\Spc S\to\Spc S$.
\end{definition}

\begin{corollary}[Sharpness for regular operators]
\label{Cor:AutomaticSharpness}
Let
$\Op L:\Spc S\to L_2$ satisfy Conditions 1--3 of Theorem~\ref{Thm:PracticalChecklist}. Then,
for every $\psi\in\Spc S$ whose image $\Op B\psi=\Op L^\ast\Op L\psi$ is a test function, one
has $\Op B\psi\in\Spc S_{\V p}$ and the supremum of Condition 4 is attained at
$\varphi=\Op B\psi$, with $e_\varphi=\Op L\psi$ and
\begin{align*}
\frac{|\langle\psi,\varphi\rangle|}{\sqrt{Q(\varphi)}}=\|\Op L\psi\|_{L_2} .
\end{align*}
Consequently, if $\Op L$ (or, more generally, $\Op B$) is regular, then Condition 4 holds automatically,
so that Conditions 1--3 alone characterize the admissibility of such operators.
\end{corollary}

\begin{proof}
\medskip
\emph{Conditioned range}:
Let $\psi\in\Spc S$ be such that $\varphi\eqdef\Op B\psi$ belongs to $\Spc S$. Since
$\Op L\psi\in L_2$ and $\Op L^\ast(\Op L\psi)=\iota\varphi$, the uniqueness of the
representer in \eqref{Eq:RepresenterCriterion} gives $e_\varphi=\Op L\psi$, so that
$Q(\varphi)=\|\Op L\psi\|^2_{L_2}<\infty$. Condition 3 then places $\varphi$ in
$\Spc S_{\V p}$, which is what makes $\varphi$ admissible in the supremum of Condition 4. In
particular, if $\Op B$ is regular, then $\Op B(\Spc S)\subseteq\Spc S_{\V p}$.

\medskip
\emph{Attainment}:
The reproducing identity \eqref{Eq:Reproducing} gives
$\langle\varphi,\psi\rangle=\|\Op L\psi\|^2_{L_2}$, from which the announced ratio follows.
It attains the Cauchy--Schwarz bound of Remark~\ref{Rem:SharpnessAsDuality}: the supremum of Condition 4 is the norm of the orthogonal projection of
$\Op L\psi$ onto $\overline{\Spc E}$, so that sharpness is a statement of surjectivity rather
than an inequality.

\medskip
\emph{Regular normal operator}:
Let us now assume that $\Op B$ is regular, so that $\Op B\psi\in\Spc S$ for every
$\psi\in\Spc S$. The two preceding steps then produce, for every $\psi\in\Spc S$, a conditioned
test function $\varphi=\Op B\psi\in\Spc S_{\V p}$ with $e_\varphi=\Op L\psi$, which gives
$\Spc E\supseteq\Op L(\Spc S)$. The latter is dense in $L_2$ by Condition 2, and hence so is
$\Spc E$, which is Condition 4.

\medskip
\emph{Regular $\Op L$}:
The adjoint of $\Op L|_{\Spc S}:\Spc S\to\Spc S$ maps $\Spc S'$ into $\Spc S'$, while the
adjoint of $\Op L:\Spc S'\to\Spc S'$ maps $\Spc S=(\Spc S')'$ into $\Spc S$; the two coincide,
so that $\Op L^\ast$ is regular. The composition of two regular operators being regular, the
normal operator $\Op B=\Op L^\ast\Op L$ is regular as well.
\end{proof}

\begin{proposition}[Concrete construction by Fourier multipliers]
\label{Prop:MultiplierConstruction}
Let $(\Spc S,\Spc S')=(\Spc S(\R^d),\Spc S'(\R^d))$ and let $\Op L$ be the Fourier-multiplier
operator $\widehat{\Op L\varphi}=\hat L\,\hat\varphi$, where the symbol $\hat L:\R^d\to\mathbb{C}$ is
measurable, Hermitian-symmetric ($\hat L(-\V\omega)=\overline{\hat L(\V\omega)}$, so that
$\Op L$ maps real functions to real functions), and satisfies, for some order $\gamma>0$, exponents
$\beta,\beta'\ge0$, and positive constants $C,c_1,c_2,c_3$,
\begin{enumerate}
\item[(i)] $|\hat L(\V\omega)|\le C\,(1+\|\V\omega\|)^{\beta}$ for a.e.\ $\V\omega$;
\item[(ii)] $c_1\|\V\omega\|^{\gamma}\le|\hat L(\V\omega)|\le c_2\|\V\omega\|^{\gamma}$ for
$\|\V\omega\|\le1$;
\item[(iii)] $|\hat L(\V\omega)|\ge c_3\,(1+\|\V\omega\|)^{-\beta'}$ for $\|\V\omega\|\ge1$.
\end{enumerate}
Let $\V p$ be a basis of the polynomials of degree at most
$k_\gamma=\lfloor\gamma-\frac d2\rfloor$ (no conditioning if $k_\gamma<0$). Then, Conditions
1--4 of Theorem~\ref{Thm:PracticalChecklist} hold, with the dual energy and the representers
given by
\begin{align*}
Q(\varphi)
=\frac{1}{(2\pi)^d}\int_{\R^d}\frac{|\hat\varphi(\V\omega)|^2}{|\hat L(\V\omega)|^2}\,
\dint\V\omega,
\qquad
\hat e_\varphi=\frac{\hat\varphi}{\overline{\hat L}}.
\end{align*}
\end{proposition}

\begin{proof}
Condition (i) gives the continuity of $\Op L:\Spc S\to L_2$, and Conditions (ii)--(iii) imply
that $\hat L\neq0$ a.e., which yields Conditions 1 and 2 of the checklist; in particular,
$\Spc R=L_2$.

\medskip
\emph{Dual energy and representers}:
The adjoint $\Op L^\ast$ is the convolution with Fourier symbol $\overline{\hat L}$. For
$e\in L_2$, the product $\overline{\hat L}\,\hat e$ is locally integrable by (i), so that the
representer equation $\Op L^\ast e=\iota\varphi$ of \eqref{Eq:RepresenterCriterion} is an
identity between locally integrable functions; namely,
\begin{align*}
\overline{\hat L}(\V\omega)\,\hat e(\V\omega)=\hat\varphi(\V\omega)
\qquad\text{for a.e. }\V\omega\in\R^d .
\end{align*}
Since $\hat L\neq0$ a.e., this equation admits the single measurable solution
$\hat e=\hat\varphi/\overline{\hat L}$, which lies in $L_2$ if and only if the integral
$\int_{\R^d}|\hat\varphi|^2/|\hat L|^2\,\dint\V\omega$ converges. The criterion
\eqref{Eq:RepresenterCriterion} of Lemma~\ref{Lem:UnconditionedCriterion}, applied with
$\Op T=\Op L$, therefore delivers the dual energy at once:
\begin{align*}
\hat e_\varphi=\frac{\hat\varphi}{\overline{\hat L}},
\qquad
Q(\varphi)
=\|e_\varphi\|_{L_2}^2
=\frac{1}{(2\pi)^d}\int_{\R^d}
\frac{|\hat\varphi(\V\omega)|^2}{|\hat L(\V\omega)|^2}\,\dint\V\omega ,
\end{align*}
where both sides are infinite when the integral diverges, the representer equation having then
no solution in $L_2$. Remarkably, neither a supremum nor a limiting argument is required.

\medskip
\emph{Finiteness domain}:
The integral converges over
$\{\|\V\omega\|\ge1\}$ for every $\varphi\in\Spc S$, by (iii). Near the origin, the two-sided
bound (ii) shows that it converges if and only if
$\int_{\|\V\omega\|\le1}\|\V\omega\|^{-2\gamma}|\hat\varphi|^2\,\dint\V\omega<\infty$: if the
moments of $\varphi$ of order at most $k_\gamma$ vanish, then
$|\hat\varphi(\V\omega)|\lesssim\|\V\omega\|^{k_\gamma+1}$ and the radial exponent
$2(k_\gamma+1)-2\gamma+d-1$ exceeds $-1$; if a moment of order $j\le k_\gamma$ survives, then
the leading Taylor term of $\hat\varphi$ has order $j$ and forces divergence, because
$2j-2\gamma+d\le0$. Hence $\{Q<\infty\}=\Spc S_{\V p}$, which is Condition 3.

\medskip
\emph{Sharpness}:
A multiplier is regular in the sense of Definition~\ref{Def:Regular} exactly when its symbol
belongs to the space $\Spc O_M(\R^d)$ of smooth functions whose derivatives are all of at most
polynomial growth, which is the classical space of multipliers of $\Spc S(\R^d)$
\cite{Treves2006}. Whenever $\hat L$ or $|\hat L|^2$ is of that type,
Corollary~\ref{Cor:AutomaticSharpness} applies and Condition 4 requires no verification: such is
the case for the polynomial symbols $P(\jj\omega)$ of Example~\ref{Ex:ExpSplines}, and for the
fractional Laplacians of integer order $\gamma$ (Example~\ref{Ex:FracLaplacian}), whose symbol
$\|\V\omega\|^\gamma$ is smooth only for even $\gamma$, while
$|\hat L|^2=\|\V\omega\|^{2\gamma}$ is a polynomial for every integer $\gamma$.

For a general measurable symbol, we approximate the maximizer $\varphi=\Op B\psi$ of
Corollary~\ref{Cor:AutomaticSharpness}, whose transform is $|\hat L|^2\hat\psi$. Let
$\psi\in\Spc S$ and let $A_k=\{1/k<\|\V\omega\|<k\}$ be the open annulus, on whose closure the
bounds (i)--(iii) confine $|\hat L|$ between two positive constants $c_k$ and $C_k$. For each
$k$, we select $\hat\varphi_k\in C^\infty_c(A_k)$ with
$\big\|\hat\varphi_k-|\hat L|^2\hat\psi\,\Indic_{A_k}\big\|_{L_2}\le c_k/k$, which is
possible because $|\hat L|^2\hat\psi\,\Indic_{A_k}$ is bounded and supported in $A_k$. The
vanishing of $\hat\varphi_k$ near the origin annihilates all the moments of $\varphi_k$, so
that $\varphi_k\in\Spc S_{\V p}$, and
\begin{align*}
\big\|\hat e_{\varphi_k}-\hat L\hat\psi\big\|_{L_2}
&\le\frac{1}{c_k}
\big\|\hat\varphi_k-|\hat L|^2\hat\psi\,\Indic_{A_k}\big\|_{L_2}
+\big\|\hat L\hat\psi\,\Indic_{\R^d\backslash A_k}\big\|_{L_2}\\
&\le\frac1k
+\big\|\hat L\hat\psi\,\Indic_{\R^d\backslash A_k}\big\|_{L_2},
\end{align*}
where the choice of the tolerance $c_k/k$ compensates for the degradation of $1/c_k$ as the
annulus grows. The right-hand side tends to $0$, its second term because
$\hat L\hat\psi\in L_2$ and $\Indic_{A_k}\to1$ a.e. Therefore
$\Op L\psi\in\overline{\Spc E}$ for every $\psi\in\Spc S$: the orthogonal projection in the
supremum of Condition 4 is lossless, which is the announced sharpness.
\end{proof}
\subsection{Factorization of Conditional Riesz Maps}
In finite-dimensions with $(\Spc S,\Spc S')=(\R^N,\R^N)$, the space $\Spc H=(\R^N,|\cdot|_\M A)$ has the Hilbert--Gelfand property if and only its norm is induced by a symmetric matrix $\M A \in \R^{N \times N}$, which enables the formulation of other equivalent criteria from matrix algebra. For instance, the symmetry and non-singularity of $\M A$ are equivalent to the existence of an invertible matrix $\M L \in \R^{N \times N}$ (not necessarily unique) such that $\M A=(\M L^{-1})^\Tr\M L^{-1}$. This factorization induces an isometric equivalence between $\Spc H=\Op L^\ast(\ell_2)$ and the Euclidean space $\ell_2=(\R^N,\|\cdot\|_2)$ where $\Op L^\ast: \M x \mapsto \M L^\Tr \M x$.

Our next result states that a similar equivalence holds in infinite dimensions within the class of Hilbert--Gelfand spaces
induced by a $\V p$-condi\-tionally positive operator. The crucial difference with the finite-dimensional case is that not all symmetric positive (semi-)definite operators $\Op A: \Spc S \to \Spc S'$ are factorizable. For instance, the operators $\Op A$ and $\Op A_+=\Op A_{\V \phi} + \Op R_{\V p}$ of Theorem \ref{Thm:Augmentation}
 are not factorizable, while the $\V p$-corrected variant $\Op A_{\V \phi}$ of the former generally is, as asserted in Item 3 below.
\begin{proposition}[Factorization of the conditional inverse]
\label{Prop:Factorization}
Let Conditions 1--4 of Theorem~\ref{Thm:PracticalChecklist} hold and let us define the
\emph{conditioned inverse} of $\Op L$ as
\begin{align}
\label{Eq:CondInverse}
\Op L^{-1}_{\V \phi}
\eqdef
\Op A_{\V \phi}\,\Op L^\ast:\ L_2\to\Spc S'_{\V \phi}\subset\Spc S' .
\end{align}
Then, the following properties hold.
\begin{enumerate}
\item \emph{Inversion}: $\Op L^{-1}_{\V \phi}\,\Op L=\Op P_{\V \phi}^\ast$ on $\Spc S$ and
$\Op L\,\Op L^{-1}_{\V \phi}=\Identity$ on $L_2$. Moreover, $\Op L^{-1}_{\V \phi}$ is an
isometric isomorphism from $L_2$ onto $\Spc H'_{\Op A}$.
\item \emph{Adjoint}: the adjoint
$\Op L^{-1\ast}_{\V \phi}=\Op L\,\Op A_{\V \phi}:\Spc S\to L_2$ is the representer map
\begin{align*}
\Op L^{-1\ast}_{\V \phi}\varphi
=
e_{\Op P_{\V \phi}\varphi}
=
{\Op L^\ast}^{-1}\iota(\Op P_{\V \phi}\varphi)\in L_2 ,
\end{align*}
with $\|\Op L^{-1\ast}_{\V \phi}\varphi\|_{L_2}^2=Q(\Op P_{\V \phi}\varphi)$ and
$\ker\Op L^{-1\ast}_{\V \phi}=\Spc N_{\V \phi}$.
\item \emph{Factorization}: the conditional inverse admits the factorization
\begin{align*}
\Op A_{\V \phi}
=
\Op L^{-1}_{\V \phi}\circ\Op L^{-1\ast}_{\V \phi}:
\Spc S\to L_2 \to \Spc S'_{\V \phi}\subset\Spc S',
\end{align*}
in exact duality with the factorization $\Op B=\Op L^\ast\Op L$ of the normal operator.
\end{enumerate}
\end{proposition}

Let us stress that \eqref{Eq:CondInverse} involves the adjoint $\Op L^\ast:L_2\to\Spc S'$, which is
primary, together with the extension of $\Op A_{\V \phi}$ granted by
Definition~\ref{Def:Intrinsic}; no inverse of an isometric extension is invoked.

\begin{proof}
\medskip
\emph{Inversion}:
For $\psi\in\Spc S$, Theorem~\ref{Thm:Concretization} gives
\begin{align*}
\Op L^{-1}_{\V \phi}\Op L\psi
=\Op A_{\V \phi}\,\Op L^\ast\Op L\,\psi
=\Op A_{\V \phi}\Op B\psi
=\Op P_{\V \phi}^\ast\psi .
\end{align*}
The extension of $\Op L$ annihilates $\Spc N_{\V p}$, so that $\Op L\Op P_{\V \phi}^\ast=\Op L$;
the application of $\Op L$ to the identity above therefore yields
$\Op L\Op L^{-1}_{\V \phi}\Op L\psi=\Op L\psi$, which extends to $L_2$ by the density of
$\Op L(\Spc S)$ (Condition 2). By Corollary~\ref{Cor:L2RealizationConditional}, the adjoint
$\Op L^\ast$ maps $L_2$ isometrically onto $\Spc H_{\Op A}=\Op L^\ast(L_2)$, while
$\Op A_{\V \phi}$ extends to an isometry of $\Spc H_{\Op A}$ onto
$\Spc H'_{\Op A}=\Op P_{\V \phi}^\ast(\dot{\Spc H}_{\Op B})$ (Theorem~\ref{Thm:Concretization},
Item 1). The composition $\Op L^{-1}_{\V \phi}=\Op A_{\V \phi}\Op L^\ast$ is thus an isometric
isomorphism from $L_2$ onto $\Spc H'_{\Op A}$.

\medskip
\emph{Adjoint}:
The operator $\Op A_{\V \phi}$ is symmetric (Remark~\ref{Rem:AutomaticSymmetry}) and the adjoint
operation is involutive, from which we get
$\Op L^{-1\ast}_{\V \phi}=(\Op A_{\V \phi}\Op L^\ast)^\ast=\Op L\,\Op A_{\V \phi}$. Let
$\varphi\in\Spc S$. Theorem~\ref{Thm:Concretization} (Item 2) gives
$\pi\Op A_{\V \phi}\varphi=\pi\Op A(\Op P_{\V \phi}\varphi)$, while
Theorem~\ref{Thm:NormalPair} gives $\widetilde{\Op L}\circ\pi\Op A=\Op T$ on $\Spc H_{\Op A}$;
hence $\Op L^{-1\ast}_{\V \phi}\varphi=e_{\Op P_{\V \phi}\varphi}$, whose squared $L_2$-norm is
$Q(\Op P_{\V \phi}\varphi)$ by Theorem~\ref{Thm:PracticalChecklist}. Its kernel is
$\ker\Op P_{\V \phi}=\Spc N_{\V \phi}$, because $\varphi\mapsto e_\varphi$ is injective on
$\Spc S_{\V p}$.

\medskip
\emph{Factorization}:
The range of $\Op A_{\V \phi}$ lies in $\Spc S'_{\V \phi}$, so that
$\Op P_{\V \phi}^\ast\Op A_{\V \phi}=\Op A_{\V \phi}$. Together with the inversion identity
$\Op A_{\V \phi}\Op B=\Op P_{\V \phi}^\ast$ of Theorem~\ref{Thm:Concretization}, this yields
\begin{align*}
\Op L^{-1}_{\V \phi}\,\Op L^{-1\ast}_{\V \phi}
=\Op A_{\V \phi}\,\Op L^\ast\Op L\,\Op A_{\V \phi}
=\Op A_{\V \phi}\,\Op B\,\Op A_{\V \phi}
=\Op P_{\V \phi}^\ast\,\Op A_{\V \phi}
=\Op A_{\V \phi} ,
\end{align*}
which is the announced factorization.
\end{proof}

\begin{remark}
\label{Rem:FactorizationConsequences}
The factorization exhibits $\Op A_{\V \phi}$ as a Gram operator, hence manifestly
symmetric, positive-semi-definite on all of $\Spc S$, with null space $\Spc N_{\V \phi}=\ker\Op P_{\V \phi}$ and
with quadratic form
$\langle\Op A_{\V \phi}\varphi,\varphi\rangle=\|e_{\Op P_{\V \phi}\varphi}\|_{L_2}^2
=Q(\Op P_{\V \phi}\varphi)$.
\end{remark}

\begin{corollary}[Conditioned factorization through a Green inverse of $\Op L^\ast$]
\label{Cor:GreenFactor}
Let Conditions 1--4 of Theorem~\ref{Thm:PracticalChecklist} hold, and let us assume that
$\Op L^\ast$ admits a Green-function inverse; that is, a continuous operator
$\Op G:\Spc S\to\Spc S'$ such that $\Op G\Op L^\ast=\Identity$ on $L_2$ and
$\Op L^\ast\Op G=\Identity$ on $\Spc S_{\V p}$ through the
Hilbertian extensions, so that $\Op G=(\Op L^\ast)^{-1}$; its adjoint $\Op G^\ast$ is then the
Green's operator of $\Op L$. If
\begin{align}
\label{Eq:L2inclusion}
\Op G(\Spc S_{\V p})\subseteq L_2,
\end{align}
then $\Op G\varphi=e_\varphi={\Op L^\ast}^{-1}\iota\varphi$ for every
$\varphi\in\Spc S_{\V p}$, so that $\Op G\Op P_{\V \phi}=\Op L^{-1\ast}_{\V \phi}$ on $\Spc S$.
Consequently, the conditioned Gram operator
\begin{align*}
\Op A_{\V \phi}
\eqdef
\Op P_{\V \phi}^\ast\,\Op G^\ast\Op G\,\Op P_{\V \phi}
=(\Op G\Op P_{\V \phi})^\ast\,(\Op G\Op P_{\V \phi}):\Spc S\to\Spc S'
\end{align*}
is well defined and admissible, and it coincides with the canonical factorization
$\Op L^{-1}_{\V \phi}\Op L^{-1\ast}_{\V \phi}$ of Proposition~\ref{Prop:Factorization}, in
exact duality with the factorization $\Op B=\Op L^\ast\Op L$ of the normal operator. In the
Fourier-multiplier setting of Proposition~\ref{Prop:MultiplierConstruction}, $\Op G$ is the
multiplier operator with symbol $1/\overline{\hat L}$, and the requirement
$\Op G(\Spc S_{\V p})\subseteq L_2$ is met by the canonical normalization of that multiplier,
for which $\|\Op G\varphi\|_{L_2}^2=Q(\varphi)$, whose finiteness on $\Spc S_{\V p}$ is
Condition 3.
\end{corollary}

\begin{proof}
\medskip
\emph{Identification of the factor}:
Let $\varphi\in\Spc S_{\V p}$ and $w\eqdef\Op G\varphi\in L_2$. The inversion property, read
through the Hilbertian extensions, gives $\Op L^\ast w=\Op L^\ast\Op G\varphi=\iota\varphi
=\Op L^\ast e_\varphi$. Since $\ker(\Op L^\ast)=\{0\}$ by Condition 2, we
conclude that $w=e_\varphi$. The identity $\Op G\Op P_{\V \phi}=\Op L^{-1\ast}_{\V \phi}$ then follows
from Proposition~\ref{Prop:Factorization}, which gives
$\Op L^{-1\ast}_{\V \phi}\varphi=e_{\Op P_{\V \phi}\varphi}$.

\medskip
\emph{Conditioned Gram operator}:
The map $\Op G\Op P_{\V \phi}$ sends $\Spc S$ into $L_2$, which makes the composition with the adjoint
well defined. For $\varphi,\psi\in\Spc S$,
\begin{align*}
\langle\Op P_{\V \phi}^\ast\Op G^\ast\Op G\,\Op P_{\V \phi}\varphi,\psi\rangle
=\langle\Op G\Op P_{\V \phi}\varphi,\Op G\Op P_{\V \phi}\psi\rangle_{L_2}
=\langle e_{\Op P_{\V \phi}\varphi},e_{\Op P_{\V \phi}\psi}\rangle_{L_2},
\end{align*}
which is the quadratic form of the admissible operator that
Lemma~\ref{Lem:UnconditionedCriterion} attaches to $\Op T=\Psi$, as well as the form of the factorization
$\Op L^{-1}_{\V \phi}\Op L^{-1\ast}_{\V \phi}$ of Proposition~\ref{Prop:Factorization}; the
three operators therefore coincide.

\medskip
\emph{Multiplier case}:
The symbol of $\Op G$ is $1/\overline{\hat L}$, so that $\Op G\varphi=e_\varphi$ and
$\|\Op G\varphi\|_{L_2}^2=Q(\varphi)$ by
Proposition~\ref{Prop:MultiplierConstruction}, which is finite on $\Spc S_{\V p}$ by
Condition 3.
\end{proof}

\begin{remark}[Normalization of the Green's function and effectiveness]
\label{Rem:GreenNormalization}
The inclusion \eqref{Eq:L2inclusion} constrains the choice of $\Op G$ and not the
operator $\Op L$: it fixes the additive ambiguity of the fundamental solution, and the canonical
one meets it. What it excludes are exotic alternatives, such as
$\rho_{\Op L^\ast}(x)=-\tfrac{1}{2\pi}\ln\|\V x\|+(x_1^2-x_2^2)$ for $\Op L=-\Delta$ on $\Spc S(\R^2)$, whose
harmonic correction of degree $2$ survives the moment conditions of order at most $1$. The
conditioning itself is unavoidable: since $\Op G$ maps the complement of $\Spc S_{\V p}$ outside of
$L_2$, the unconditioned composition $\Op G^\ast\Op G$ is ill-defined. Proposition~\ref{Prop:Factorization}
and Corollary~\ref{Cor:GreenFactor} operate in opposite directions: the former decomposes a given
$\Op A_{\V \phi}$ into the analysis--synthesis pair $(\Op L^{-1\ast}_{\V \phi},\Op L^{-1}_{\V \phi})$,
whereas the latter produces $\Op A_{\V \phi}$ from a solution $e_\varphi=\rho_{\Op L^\ast}\ast\varphi$
of the representer equation $\Op L^\ast e=\iota\varphi$, which is available in closed form for the
Fourier multipliers of Proposition~\ref{Prop:MultiplierConstruction}.
\end{remark}

\subsection{The Native Space and its Uniqueness}
\label{Sec:NativeUniqueness}
Let $\Op L:\Spc S\to L_2$ be a candidate regularization operator that meets the admissibility
conditions of Theorem~\ref{Thm:PracticalChecklist}. The space that underlies the variational
theory of splines is then the native space
\begin{align}
\Spc H_{\Op L}
\eqdef
\overline{\big(\Spc S,|\cdot|_{\Op B_+}\big)},
\qquad
|f|_{\Op B_+}^2=\|\Op L f\|_{L_2}^2+|\V \phi(f)|_2^2,\label{Eq:NativeSpace}
\end{align}
which is induced by the augmented operator $\Op B_+=\Op L^\ast\Op L+\Op R_{\V \phi}$ of
Theorem~\ref{Thm:Augmentation}, where $\V \phi\subset\Spc S$ is biorthogonal to a basis
$\V p$ of the attainable null space $\Spc N_{\Op B}=\Spc N_{\V p}$. In the notation of the
previous sections, $\Spc H_{\Op L}=\Spc H'$: it is the concrete dual of the Hilbert--Gelfand
space $\Spc H$.
The description \eqref{Eq:NativeSpace} is remarkably compact, especially in comparison with the distributional characterization of the polyharmonic native spaces in \cite{Duchon1977}; Remark~\ref{Rem:Duchon} explains why the two spaces differ in high dimensions.

The definition \eqref{Eq:NativeSpace} involves two arbitrary choices, the basis $\V p$ of the attainable null space and the biorthogonal family $\V \phi$, and it writes $\|\Op L f\|_{L_2}$ for elements $f$ of a completion, although $\Op L$ is only defined on $\Spc S$. The following proposition settles both points. It shows that $\Op L$ extends to the native space, that $\Spc H_{\Op L}$ is independent of $\V p$ and $\V \phi$ up to an equivalence of norms and that, as a set, it is the only Hilbert--Gelfand space on which the energy is a bounded and surjective map onto $L_2$. It also identifies the norms \eqref{Eq:NativeSpace} among all equivalent Hilbertian norms.

\begin{proposition}[Extension of $\Op L$ and uniqueness of the native space]
\label{Prop:NativeUniqueness}
Let $\Op L:\Spc S\to L_2$ satisfy Conditions 1--4 of Theorem~\ref{Thm:PracticalChecklist}, let $\V p$ be a basis of $\Spc N_{\Op B}=\Spc N_{\V p}$, and let $\Spc H_{\Op L}$ be the native space \eqref{Eq:NativeSpace} for some family $\V \phi\subset\Spc S$ that is biorthogonal to $\V p$. Then, the following properties hold.
\begin{enumerate}
\item \emph{Extension of $\Op L$}: The map $\psi\mapsto\Op L\psi$ extends by continuity from $\Spc S$ to a bounded operator $\Op L:\Spc H_{\Op L}\to L_2$, which is given by
\begin{align}
\label{Eq:LExtension}
\Op L f=\widetilde{\Op L}\,\pi f,\qquad f\in\Spc H_{\Op L},
\end{align}
where $\widetilde{\Op L}:\dot{\Spc H}_{\Op B}\to L_2$ is the isometric isomorphism of Theorem~\ref{Thm:PracticalChecklist}. This extension is surjective, its null space is $\Spc N_{\V p}$, and $|f|_{\Op B_+}^2=\|\Op L f\|_{L_2}^2+|\V \phi(f)|_2^2$ for every $f\in\Spc H_{\Op L}$. Moreover, $\Spc H_{\Op L}=\{f\in\Spc S':\ \pi f\in\dot{\Spc H}_{\Op B}\}$ does not depend on $\V \phi$, and two admissible choices of $\V \phi$ yield equivalent norms (Corollary~\ref{Cor:PhiIndependence}).
\item \emph{Uniqueness as a set}: Let $\Spc X$ be a Hilbert--Gelfand space such that $\psi\mapsto\Op L\psi$ extends from $\Spc S$ to a bounded operator $\Op L_{\Spc X}:\Spc X\to L_2$ that is surjective. Then, $\Spc X=\Spc H_{\Op L}$ as sets, $\Op L_{\Spc X}$ coincides with the extension \eqref{Eq:LExtension}, and $\|\cdot\|_{\Spc X}$ is equivalent to $|\cdot|_{\Op B_+}$.
\item \emph{Characterization of the norms \eqref{Eq:NativeSpace}}: Let $\Spc X$ be as in Item 2, and let $\Spc X'$ be its dual, which is a Hilbert--Gelfand space whose pairing with $\Spc X$ extends the canonical one (Theorem~\ref{Theo:HilbertPositiveDef}). Then, $\|f\|_{\Spc X}=\|\Op L f\|_{L_2}$ for every $f$ in the orthogonal complement $\Spc N_{\V p}^{\perp}$ of $\Spc N_{\V p}$ in $\Spc X$ if and only if there exist a basis $\tilde{\V p}$ of $\Spc N_{\V p}$ and a family $\tilde{\V \phi}\subset\Spc X'$ with $\langle\tilde p_m,\tilde\phi_n\rangle=\delta_{mn}$ such that
\begin{align}
\label{Eq:AugmentedNormGeneral}
\|f\|_{\Spc X}^2=\|\Op L f\|_{L_2}^2+|\tilde{\V \phi}(f)|_2^2,\qquad f\in\Spc X.
\end{align}
In that case, the map $f\mapsto(\Op L f,\tilde{\V \phi}(f))$ is a unitary isomorphism from $\Spc X$ onto $L_2\times\R^{N_0}$. When $\tilde{\V \phi}\subset\Spc S$, the space $\Spc X$ is the native space \eqref{Eq:NativeSpace} that is associated with the biorthogonal pair $(\tilde{\V p},\tilde{\V \phi})$.
\end{enumerate}
\end{proposition}

\begin{proof}
\emph{Extension of $\Op L$}:
For $\psi\in\Spc S$, we have that $\|\Op L\psi\|_{L_2}=|\psi|_{\Op B}=\|\pi\psi\|_{\dot{\Spc H}_{\Op B}}\le|\psi|_{\Op B_+}$, so that $\psi\mapsto\Op L\psi$ extends by continuity to $\Spc H_{\Op L}$. The operator $\widetilde{\Op L}\pi$ is bounded on $\Spc H_{\Op L}$, because Item 3 of Theorem~\ref{Thm:Augmentation} gives $\|\pi f\|_{\dot{\Spc H}_{\Op B}}=\|\Op P_{\V \phi}^\ast f\|_{\Spc H'_{\Op A}}\le|f|_{\Op B_+}$, and it coincides with $\Op L$ on $\Spc S$. The two bounded operators therefore agree on the dense subspace $\Spc S$, which proves \eqref{Eq:LExtension}. The map $\pi$ is onto $\dot{\Spc H}_{\Op B}$ by \eqref{Eq:CanonicalHprime} and $\widetilde{\Op L}$ is onto $L_2$ by Condition 2 of Theorem~\ref{Thm:PracticalChecklist}, so that the extension is surjective. Since $\widetilde{\Op L}$ is injective, its null space is $\ker\pi\cap\Spc H_{\Op L}=\Spc N_{\V p}$. The norm identity is Item 3 of Theorem~\ref{Thm:Augmentation}, in which we substitute $\|\Op P_{\V \phi}^\ast f\|_{\Spc H'_{\Op A}}=\|\pi f\|_{\dot{\Spc H}_{\Op B}}=\|\widetilde{\Op L}\pi f\|_{L_2}$. The last assertion is Corollary~\ref{Cor:PhiIndependence}.

\medskip
\emph{Inclusion $\Spc X\subseteq\Spc H_{\Op L}$}:
Let $f\in\Spc X$ and let $(\psi_k)\subset\Spc S$ converge to $f$ in $\Spc X$. Then, $\Op L\psi_k\to\Op L_{\Spc X}f$ in $L_2$, so that $(\psi_k)$ is $|\cdot|_{\Op B}$-Cauchy and $\pi\psi_k$ converges to some $g\in\dot{\Spc H}_{\Op B}$ with $\widetilde{\Op L}g=\Op L_{\Spc X}f$. By Item 3 of Proposition~\ref{Prop:Realization}, this convergence also takes place in $\Spc S'/\Spc N_{\V p}$. On the other hand, $\psi_k\to f$ in $\Spc S'$ because $\Spc X$ is continuously embedded in $\Spc S'$, so that $\pi\psi_k\to\pi f$ in $\Spc S'/\Spc N_{\V p}$. Hence $\pi f=g\in\dot{\Spc H}_{\Op B}$, which places $f$ in $\Spc H_{\Op L}$ by \eqref{Eq:CanonicalHprime}, and $\Op L_{\Spc X}f=\widetilde{\Op L}\pi f=\Op L f$. The inclusion map $\Spc X\to\Spc H_{\Op L}$ is continuous: it is a linear map between two Hilbert spaces that are both continuously embedded in the Hausdorff space $\Spc S'$, so that its graph is closed, and the closed-graph theorem applies.

\medskip
\emph{Equality $\Spc X=\Spc H_{\Op L}$}:
Since $\Op L_{\Spc X}:\Spc X\to L_2$ is bounded and surjective, its restriction to the orthogonal complement of $\ker\Op L_{\Spc X}$ is a bounded bijection onto $L_2$. By the open-mapping theorem, the inverse of this restriction is a bounded operator $\Op R:L_2\to\Spc X$ with $\Op L_{\Spc X}\Op R=\Identity$. For $f\in\Spc X$, we have that $f-\Op R\Op L_{\Spc X}f\in\ker\Op L_{\Spc X}=\Spc N_{\V p}\cap\Spc X$, so that
\begin{align*}
\Spc X=\Op R(L_2)+(\Spc N_{\V p}\cap\Spc X).
\end{align*}
The range $\Op R(L_2)$ is closed in $\Spc H_{\Op L}$. Indeed, the previous step gives the bound $|\Op Rw|_{\Op B_+}\le C\|\Op Rw\|_{\Spc X}\le C\,\|\Op R\|\,\|w\|_{L_2}$, while Item 1 gives $\|w\|_{L_2}=\|\Op L\Op Rw\|_{L_2}\le|\Op Rw|_{\Op B_+}$. Hence, $\Op R$ is an isomorphism from $L_2$ onto its range in $\Spc H_{\Op L}$. The sum of a closed subspace and of a finite-dimensional subspace is closed. Consequently, $\Spc X$ is a closed subspace of $\Spc H_{\Op L}$ that contains $\Spc S$, which is dense in $\Spc H_{\Op L}$. We conclude that $\Spc X=\Spc H_{\Op L}$. The identity map from $(\Spc H_{\Op L},\|\cdot\|_{\Spc X})$ onto $(\Spc H_{\Op L},|\cdot|_{\Op B_+})$ is a continuous bijection between Hilbert spaces, so that its inverse is continuous as well, which is the equivalence of norms.

\medskip
\emph{Characterization of the norms}:
Let us first suppose that $\|f\|_{\Spc X}=\|\Op L f\|_{L_2}$ on $\Spc N_{\V p}^\perp$. By Item 2, $\Spc N_{\V p}$ is a subspace of $\Spc X$ of dimension $N_0$; let $\Op Q$ be the orthogonal projector onto it and let $\tilde{\V p}=(\tilde p_1,\dots,\tilde p_{N_0})$ be an orthonormal basis of $\Spc N_{\V p}$ in $\Spc X$. The Riesz representatives $\tilde\phi_n\eqdef\langle\cdot,\tilde p_n\rangle_{\Spc X}$ belong to $\Spc X'$ and satisfy $\langle\tilde p_m,\tilde\phi_n\rangle=\langle\tilde p_m,\tilde p_n\rangle_{\Spc X}=\delta_{mn}$, while $\Op Qf=\sum_n\langle f,\tilde\phi_n\rangle\tilde p_n$. Since $\Op L$ annihilates $\Spc N_{\V p}$, we obtain that
\begin{align*}
\|f\|_{\Spc X}^2=\|f-\Op Qf\|_{\Spc X}^2+\|\Op Qf\|_{\Spc X}^2=\|\Op L f\|_{L_2}^2+|\tilde{\V \phi}(f)|_2^2,
\end{align*}
which is \eqref{Eq:AugmentedNormGeneral}. Conversely, let \eqref{Eq:AugmentedNormGeneral} hold. By polarization, the inner product of $\Spc X$ is $\langle f,g\rangle_{\Spc X}=\langle\Op L f,\Op L g\rangle_{L_2}+\tilde{\V \phi}(f)^\Tr\tilde{\V \phi}(g)$. Since $\Op L\tilde p_m=0$ and by biorthogonality, $\langle f,\tilde p_m\rangle_{\Spc X}=\langle f,\tilde\phi_m\rangle$, so that $\Spc N_{\V p}^\perp=\bigcap_n\ker\tilde\phi_n$, on which \eqref{Eq:AugmentedNormGeneral} reduces to $\|f\|_{\Spc X}=\|\Op L f\|_{L_2}$. The unitarity of $f\mapsto(\Op L f,\tilde{\V \phi}(f))$ is \eqref{Eq:AugmentedNormGeneral} together with the surjectivity of $\Op L$ and the fact that $\tilde{\V \phi}$ takes every value in $\R^{N_0}$ on $\Spc N_{\V p}$. Finally, when $\tilde{\V \phi}\subset\Spc S$, the right-hand side of \eqref{Eq:AugmentedNormGeneral} is the norm $|\cdot|_{\Op B_+}$ of \eqref{Eq:NativeSpace} for the pair $(\tilde{\V p},\tilde{\V \phi})$, and $\Spc X$ is its completion because $\Spc S$ is dense in $\Spc X$.
\end{proof}

Item 2 concerns the set only: any Hilbertian norm on $\Spc H_{\Op L}$ that is equivalent to $|\cdot|_{\Op B_+}$---for instance, $\|\Op L f\|_{L_2}^2+\int|f|^2w$ for a weight $w$ that makes the second term bounded on $\Spc H_{\Op L}$ and positive-definite on $\Spc N_{\V p}$---defines another Hilbert--Gelfand structure on the same set, with the same dual set and another Riesz pair. Item 3 singles out the norms for which $\Op L$ is a partial isometry, that is, an isometry from $\Spc N_{\V p}^\perp$ onto $L_2$; these are the norms \eqref{Eq:NativeSpace}, with the proviso that the family $\tilde{\V \phi}$ may need to be taken in the dual $\Spc X'$ rather than in $\Spc S$. The classical norm $\|f'\|_{L_2}^2+|f(0)|^2$ of Example~\ref{Ex:LinearSpline} is of this kind, with $\tilde\phi_1=\delta$.

\section{Examples of Native Spaces}
\label{Sec:Examples}
We now instantiate the construction for the two primary families of operators of spline theory:
the fractional Laplacians, which cover the polyharmonic and thin-plate settings, and the ordinary
differential operators with constant coefficients, which cover the exponential splines. We first
summarize the workflow.

\subsection{Roadmap for the Construction of Native Spaces}
\label{Sec:Roadmap}
We summarize the workflow that leads from a candidate operator $\Op L$ to $\Spc H_{\Op L}$.

\medskip
\emph{Step 1 (Context)}:  One selects a nuclear Gelfand pair $(\Spc S,\Spc S')$ that is adapted
to the task---for instance, $(\Spc S(\R^d),\Spc S'(\R^d))$ for the interpolation of scattered
data---together with a candidate regularization operator $\Op L:\Spc S\to L_2$. This yields the
normal operator $\Op B=\Op L^\ast\Op L$ and the dual energy $Q$ of
Theorem~\ref{Thm:PracticalChecklist}.

\medskip
\emph{Step 2 (Admissibility)}:  One verifies Conditions 1--4 of
Theorem~\ref{Thm:PracticalChecklist}. The finiteness domain of $Q$ dictates the conditioning:
$\V p$ spans the attainable null space $\Spc N_{\Op B}=\Spc N_{\V p}$, which is the trend space
of the theory. For a Fourier multiplier whose symbol meets (i)--(iii) of
Proposition~\ref{Prop:MultiplierConstruction}, the four conditions hold automatically, with
$\V p$ a basis of $\Pi_{k_\gamma}$. When a Green inverse of $\Op B$ is available,
Corollary~\ref{Cor:GreenNormal} offers a shorter route.

\medskip
\emph{Step 3 (Conditioning)}:  One selects a family $\V \phi\in\Spc S^{N_0}$ that is
biorthogonal to $\V p$, whose existence is granted by Lemma~\ref{Lem:Biorthogonal}. It
specifies the rank-$N_0$ operators $\Op R_{\V p}$ and $\Op R_{\V \phi}$, the projectors
$\Op P_{\V \phi}$ and $\Op P_{\V \phi}^\ast$, and the splittings
$\Spc S=\Spc S_{\V p}\oplus\Spc N_{\V \phi}$ and $\Spc S'=\Spc S'_{\V \phi}\oplus\Spc N_{\V p}$.

\medskip
\emph{Step 4 (Green's operator)}:  One determines $\Op G=(\Op L^\ast)^{-1}$ of
Corollary~\ref{Cor:GreenFactor}; in the shift-invariant case, $\Op G$ is the convolution with
$\rho_{\Op L^\ast}=\Spc F^{-1}\{1/\overline{\hat L}\}$. This provides the representer map
$\Op L^{-1\ast}_{\V \phi}=\Op G\Op P_{\V \phi}$ and the conditioned Gram operator
$\Op A_{\V \phi}=(\Op G\Op P_{\V \phi})^\ast(\Op G\Op P_{\V \phi})$, in exact duality with the
factorization $\Op B=\Op L^\ast\Op L$.

\medskip
\emph{Step 5 (Native space)}:  The augmentation of Theorem~\ref{Thm:Augmentation} adjoins the
finite-dimensional blocks and produces the orthogonal direct sums
$\Spc H=\Spc H_{\Op A}\oplus\Spc N_{\V \phi}$ and
$\Spc H_{\Op L}=\Spc H'=\Spc H'_{\Op A}\oplus\Spc N_{\V p}$, which are mutually dual
Hilbert--Gelfand spaces. The native space
$\Spc H_{\Op L}$ that is specified by \eqref{Eq:NativeSpace} then has the three
properties that we were after: it contains the trend space $\Spc N_{\V p}$; it carries
a well-defined energy $\|\Op L f\|_{L_2}^2$; and it makes the minimum-energy
reconstruction problem well posed. Moreover, the outcome does not depend on the choices made in
Steps 2--3 (the basis $\V p$ of the attainable null space and the family $\V\phi$) other than
through equivalent norms (Proposition~\ref{Prop:NativeUniqueness}).

\subsection{Native Spaces Induced by Fractional Laplacians}

\begin{example}[$\Op L=(-\Delta)^{\gamma/2}$ on $\Spc S(\R^d)$]
\label{Ex:FracLaplacian}
Let $(\Spc S,\Spc S')=(\Spc S(\R^d),\Spc S'(\R^d))$ and let
$(-\Delta)^{\gamma/2}$ be the fractional Laplacian of order $\gamma>0$, which is the self-adjoint
Fourier-multiplier operator specified by
\begin{align*}
\widehat{(-\Delta)^{\gamma/2}\varphi}(\V\omega)
=\|\V\omega\|^{\gamma}\,\hat\varphi(\V\omega),
\end{align*}
where $\hat\varphi(\V\omega)=\int_{\R^d}\varphi(\V x)\,\ee^{-\jj\V\omega^\Tr\V x}\,\dint\V x$ and
$\|\cdot\|$ denotes the Euclidean norm. We set
\begin{align*}
\Op B
=\big((-\Delta)^{\gamma/2}\big)^{\ast}(-\Delta)^{\gamma/2}
=(-\Delta)^{\gamma},
\qquad
\widehat{\Op B\varphi}(\V\omega)=\|\V\omega\|^{2\gamma}\,\hat\varphi(\V\omega).
\end{align*}
Unless $\gamma$ is an even integer, the symbol $\|\V\omega\|^{\gamma}$ is not smooth at the
origin, so that $(-\Delta)^{\gamma/2}$ does not map $\Spc S(\R^d)$ into itself; the operator
$\Op B:\Spc S\to\Spc S'$ is nevertheless continuous, which is all that the framework requires.
It is symmetric and positive-definite, because
\begin{align*}
\langle\Op B\varphi,\varphi\rangle
=\|(-\Delta)^{\gamma/2}\varphi\|_{L_2}^2
=\frac{1}{(2\pi)^d}\int_{\R^d}\|\V\omega\|^{2\gamma}\,|\hat\varphi(\V\omega)|^2\,\dint\V\omega
>0
\end{align*}
for $\varphi\neq0$, with the induced norm
$|\varphi|_{\Op B}=\|(-\Delta)^{\gamma/2}\varphi\|_{L_2}$. We write $\Pi_k$ for the space of
polynomials of degree at most $k$ on $\R^d$, with the convention that $\Pi_k=\{0\}$ for $k<0$. The
linear-spline setting of Example~\ref{Ex:LinearSpline} is the case $(\gamma,d)=(1,1)$; the value
$\gamma=2$, for which $(-\Delta)^{\gamma/2}=-\Delta$ and $\Op B=\Delta^2$, leads to the thin-plate
setting that we work out at the end of this example.

\medskip
\emph{Distributional null spaces versus the attainable null space}:
Over $\Spc S'$, the operator $\Op B$ annihilates every polynomial of degree less than $2\gamma$ and,
for integer $\gamma$ and $d\ge2$, the infinite-dimensional space of $\gamma$-polyharmonic polynomials.
None of these spaces can serve as $\Spc N_{\V p}$: the object that the theory calls for is the
attainable null space $\Spc N_{\Op B}\subseteq\ker(\Op B:\Spc S'\to\Spc S')$ of
Definition~\ref{Def:AttainableNull}.

\medskip
\emph{Admissibility and the attainable null space}:
The symbol $\|\V\omega\|^{\gamma}$ satisfies conditions (i)--(iii) of
Proposition~\ref{Prop:MultiplierConstruction} with order $\gamma$. Consequently, Conditions
1--4 of Theorem~\ref{Thm:PracticalChecklist} hold with $\V p$ a basis of $\Pi_{k_\gamma}$
(e.g., the monomials of degree at most $k_\gamma$), and the theorem identifies the attainable
null space as
\begin{align*}
\Spc N_{\Op B}=\Spc N_{\V p}=\Pi_{k_\gamma},
\qquad
k_\gamma\eqdef\Big\lfloor\gamma-\frac d2\Big\rfloor,
\end{align*}
so that $\Spc N_{\Op B}$ is nontrivial only for $\gamma\ge d/2$, with
$N_0=\dim\Pi_{k_\gamma}=\binom{k_\gamma+d}{d}$ and
\begin{align*}
\Spc S_{\V p}
=\{\varphi\in\Spc S(\R^d):\
\text{the moments of $\varphi$ of order at most $k_\gamma$ vanish}\}.
\end{align*}
The dual energy is the explicit integral
\begin{align*}
Q(\varphi)
=\frac{1}{(2\pi)^d}\int_{\R^d}\frac{|\hat\varphi(\V\omega)|^2}{\|\V\omega\|^{2\gamma}}\,\dint\V\omega,
\end{align*}
whose finiteness domain is decided by the Taylor expansion of $\hat\varphi$ at the origin (see
the proof of Proposition~\ref{Prop:MultiplierConstruction}).

\medskip
\emph{Attainability made explicit}:
It is instructive to exhibit the annihilation of $\Pi_{k_\gamma}$ directly, without the
abstract cutoff construction of Theorem~\ref{Thm:PracticalChecklist}. Let
$\theta\in\Spc S(\R^d)$ be equal to $1$ in a neighborhood of the origin and consider, for a
monomial $\V x^{\V\alpha}$ with $|\V\alpha|=j\le k_\gamma$, the dilated products
$\varphi_t(\V x)=\V x^{\V\alpha}\,\theta(\V x/t)$, whose Fourier transforms are
$\hat\varphi_t(\V\omega)=\jj^{\,j}t^{d+j}h_\V\alpha(t\V\omega)$ with
$h_{\V \alpha}\eqdef\Op D^{\V\alpha}\hat\theta\in\Spc S(\R^d)$. A change of variables gives
\begin{align*}
|\varphi_t|_{\Op B}^2
=
t^{\,d+2j-2\gamma}\,\frac{1}{(2\pi)^d}\int_{\R^d}\|\V u\|^{2\gamma}|h_\V\alpha(\V u)|^2\,\dint\V u .
\end{align*}
If $j<\gamma-\frac d2$, then $|\varphi_t|_{\Op B}\to0$ as $t\to\infty$, while
$\varphi_t(\V x)\to\V x^{\V\alpha}$ in $\Spc S'$; hence $\V x^{\V\alpha}\in\Spc N_{\Op B}$. In the
borderline case $j=\gamma-\frac d2$, which occurs when $\gamma-\frac d2$ is a nonnegative
integer (so that $j=k_\gamma$), the dilations must be averaged logarithmically: the sequence
$\varphi_K(\V x)\eqdef(\ln K)^{-1}\int_1^K\V x^{\V\alpha}\,\theta(\V x/t)\,t^{-1}\,\dint t$
converges to $\V x^{\V\alpha}$ in $\Spc S'$, and a Cauchy--Schwarz estimate of the
cross-terms gives $|\varphi_K|_{\Op B}^2\lesssim(\ln K)^{-1}\to0$, because $2\gamma=d+2j$.

\medskip
\emph{The conditional inverse and the order of conditional positivity}:
There is no need to pick an explicit biorthogonal family: its existence is granted by
Lemma~\ref{Lem:Biorthogonal} (a concrete choice consists of Hermite functions; that is, of
suitable linear combinations of the derivatives of a normalized Gaussian). The canonical
fundamental solution of $\Op B$ is the radial kernel
\begin{align*}
g_\gamma(\V x)
=
\begin{cases}
c_{\gamma,d}\,\|\V x\|^{2\gamma-d}, & \gamma-\frac d2\notin\mathbb{Z}_{\ge0},\\[4pt]
c_{\gamma,d}\,\|\V x\|^{2\gamma-d}\ln\|\V x\|, & \gamma-\frac d2\in\mathbb{Z}_{\ge0},
\end{cases}
\end{align*}
with known normalization constants $c_{\gamma,d}$ \cite{Samko1993,Wendland2005}; its symbol is $\|\V\omega\|^{-2\gamma}$,
and we note that the logarithmic branch occurs exactly when the borderline degree
$k_\gamma=\gamma-\frac d2$ is present in $\Spc N_{\Op B}$. We set
$\Op A=g_\gamma\ast\cdot$, which is a two-sided Green inverse of $\Op B$, with
$\Op A\Op B=\Identity$ on $\Spc S(\R^d)$ and $\Op B\Op A=\Identity$ on $\Spc S_{\V p}$
through the Hilbertian extensions. By Corollary~\ref{Cor:GreenNormal}, $(\Op A,\Op B)$ is a
$\V p$-conditional Riesz pair with the \emph{minimal} conditioning
(Theorem~\ref{Thm:AttainableNull}); for $\varphi\in\Spc S_{\V p}$, the Parseval relation
identifies its quadratic form with the dual energy,
$\langle\Op A\varphi,\varphi\rangle=Q(\varphi)$, and
$\Op A_{\V \phi}=\Op P_{\V \phi}^\ast\Op A\Op P_{\V \phi}$ is admissible (Corollary~\ref{Cor:GreenFactor}). In
the classical terminology of radial basis functions, $k_\gamma+1$ is the order of
conditional positive-definiteness of the kernel $g_\gamma(\V x-\V y)$
\cite{Madych1990c,Wendland2005}.
The ambiguity of the fundamental solution is the trace of the distributional kernel of
$\Op B$: additive polynomial corrections of degree at most $2k_\gamma+1$ are harmless, because
they map $\Spc S_{\V p}$ into $\Spc N_{\Op B}$ and alter neither the class $\pi\Op A\varphi$ nor the
quadratic form, whereas the higher-degree corrections of the polyharmonic case are excluded by the
isometry requirement of Definition~\ref{Def:Intrinsic}. Finally, the affine conditioning $\V p=(1,x_1,\dots,x_d)$ with
$N_0=d+1$ corresponds to $k_\gamma=1$, that is, to $1\le\gamma-\frac d2<2$; for the Laplacian
energy $\gamma=2$, this selects $d\in\{1,2\}$, whereas $\Spc N_{\Op B}$ shrinks to $\Pi_0$ for
$d\in\{3,4\}$ and to $\{0\}$ for $d\ge5$. This differs from the classical Beppo--Levi setting \cite{Deny-Lions1954},
where the seminorm kernel is $\Pi_1$ in every dimension: the Beppo--Levi space is strictly larger
than the completion of $\Spc S$ in high dimensions, and the difference is precisely the polynomial
directions that are not attainable from $\Spc S$.

\medskip
\emph{Thin-plate setting ($\gamma=2$, $d=2$)}:
Here, $\Op B=\Delta^2$, $k_\gamma=1$, and we take $\V p=(1,x_1,x_2)$ with $N_0=3$, so that
$\Spc S_{\V p}$ consists of the test functions whose Fourier transform vanishes to second
order at the origin. The fundamental solution is the thin-plate function of Duchon
\cite{Duchon1977},
$g_2(\V x)=\tfrac{1}{8\pi}\|\V x\|^2\ln\|\V x\|$, which instantiates the logarithmic
branch ($\gamma-\frac d2=1\in\mathbb{Z}_{\ge0}$), and the quadratic form on
$\Spc S_{\V p}$ is the classical thin-plate kernel form
\begin{align*}
\langle\Op A\varphi,\varphi\rangle
=\frac{1}{8\pi}\iint_{\R^2\times\R^2}
\|\V x-\V y\|^2\ln\|\V x-\V y\|\,\varphi(\V x)\varphi(\V y)\,\dint\V x\,\dint\V y,
\end{align*}
with conditional order $k_\gamma+1=2$. The space
$\dot{\Spc H}_{\Op B}=\dot H^{2}(\R^2)$ is the Beppo--Levi space of the thin-plate energy,
whose elements are functions modulo affine polynomials, while $\Spc H_{\Op A}$ consists of
the tempered distributions $u$ with $\hat u/\|\V\omega\|^{2}\in L_2(\R^2)$, which
annihilate $\Pi_1$. The augmentation of Theorem~\ref{Thm:Augmentation}, with
$\Op B_+=\Delta^2+\sum_{n=1}^{3}\phi_n\langle\cdot,\phi_n\rangle$, then yields the
concrete dual pair of Hilbert--Gelfand spaces
\begin{align*}
\Spc H=\Spc H_{\Op A}\oplus\Spc N_{\V \phi},
\qquad
\Spc H'=\{f\in\Spc S'(\R^2):\Delta f\in L_2(\R^2)\}
=\Spc H'_{\Op A}\oplus\Pi_1.
\end{align*}
\end{example}

\begin{remark}[Duchon's polyharmonic splines and the Beppo--Levi space]
\label{Rem:Duchon}
The polyharmonic splines of Duchon \cite{Duchon1977} minimize $\|\Op D^m f\|^2_{L_2}$ over the
Beppo--Levi space $BL^m(\R^d)$ of Deny and Lions \cite{Deny-Lions1954}, whose seminorm has the
full algebraic kernel $\Pi_{m-1}$, in every dimension. By Example~\ref{Ex:FracLaplacian} with
$\gamma=m$, the attainable null space of the same energy is $\Spc N_{\Op B}=\Pi_{k_m}$ with
$k_m=\lfloor m-\tfrac d2\rfloor$, which coincides with $\Pi_{m-1}$ for $d\in\{1,2\}$ only. For
$d\ge3$, the polynomials of degree $k_m<j\le m-1$ have zero energy, but they are not limits of
test functions; in fact, $\Spc S$ is not dense in $BL^m(\R^d)$ for any Hilbertian norm of the
form \eqref{Eq:AugmentedNormGeneral} that fixes the polynomial part, so that $BL^m(\R^d)$ is not a
Hilbert--Gelfand space, and Item~2 of Proposition~\ref{Prop:NativeUniqueness} rules it out as
the native space. Instead, $BL^m(\R^d)=\Spc H_{\Op L}\oplus\Spc G$, where the complement $\Spc G$
is spanned by the polynomials of degree $k_m<j\le m-1$. Duchon's construction is correct within
$BL^m(\R^d)$, but the over-conditioning has practical consequences. First, it requires the data
sites for interpolation to be unisolvent for $\Pi_{m-1}$, whereas the energy only calls for $\Pi_{k_m}$: for $d=3$,
$m=2$ and the kernel $g_2(\V x)\propto\|\V x\|$, Duchon's interpolant needs four side conditions
and non-coplanar sites, while the spline of $\Spc H_{\Op L}$ exists and is unique for any set of
distinct points under the single condition $\sum_i a_i=0$. Second, the coefficients of the
polynomials in $\Spc G$ are not regularized by the energy, whatever the value of the
regularization parameter; the data alone determine them. Third, $\Spc G$ is annihilated by the
whole core $\Spc H_{\Op A}=\Op L^\ast(L_2)$, so that no functional of finite dual energy can
detect these components. The theory of radial basis functions \cite{Madych1990c,Wendland2005}
sidesteps the issue when it works with the minimal order $k_m+1$ of conditional positive
definiteness of $g_m$, which matches $\Spc N_{\Op B}$; the larger space $\Pi_{m-1}$ enters that
theory only through the Beppo--Levi spaces, where it is inherited from Duchon.
\end{remark}

\subsection{Native Spaces for Exponential Splines}

\begin{example}[$\Op L=P(\Op D)$ on $\Spc S(\R)$]
\label{Ex:ExpSplines}
Let $(\Spc S,\Spc S')=(\Spc S(\R),\Spc S'(\R))$ and let
\begin{align*}
\Op L=P(\Op D)=a_{N_{\mathrm{tot}}}\prod_{n=1}^{N_{\mathrm r}}(\Op D-\alpha_n\Identity)^{m_n},
\qquad
\Op D=\tfrac{\dint}{\dint x},
\end{align*}
where $P$ is an arbitrary polynomial with real coefficients, whose $N_{\mathrm r}$ distinct
roots $\alpha_n\in\mathbb{C}$ have multiplicities $m_n$. The total number of roots---that is,
the degree of $P$---is $N_{\mathrm{tot}}=\sum_{n=1}^{N_{\mathrm r}}m_n$. In contrast with the
fractional operators of Example~\ref{Ex:FracLaplacian}, $\Op L$ maps $\Spc S(\R)$ into itself; its
symbol is the polynomial $\hat L(\omega)=P(\jj\omega)$, and the normal operator is
\begin{align*}
\Op B=\Op L^\ast\Op L=P(-\Op D)\,P(\Op D),
\qquad
\widehat{\Op B\varphi}(\omega)=|P(\jj\omega)|^2\,\hat\varphi(\omega).
\end{align*}

\medskip
\emph{Null spaces}:
The kernel of $P(\Op D)$ over $C^\infty(\R)$ is the $N_{\mathrm{tot}}$-dimensional space of exponential
polynomials, spanned by the functions $x^k\ee^{\alpha_nx}$ with $0\le k\le m_n-1$; since the
operator is hypoelliptic, this is also its kernel over $\Spc D'(\R)$. The tempered part of
this kernel retains exactly the roots on the imaginary axis, because $x^k\ee^{\alpha_nx}$
grows exponentially when $\operatorname{Re}(\alpha_n)\neq0$. Writing $\alpha_n=\jj\omega_n$ for the
imaginary roots, we therefore set
\begin{align*}
\Spc N_{\V p}
=\ker\big(P(\Op D):\Spc S'\to\Spc S'\big)
=\mathrm{span}\{x^k\ee^{\jj\omega_nx}:\ \alpha_n=\jj\omega_n,\ 0\le k\le m_n-1\},
\end{align*}
with $N_0=\sum_{\alpha_n\in\jj\R}m_n$; a real basis is obtained from $x^k\cos(\omega_nx)$ and
$x^k\sin(\omega_nx)$, the imaginary roots of a real polynomial arising in conjugate pairs.
Correspondingly,
\begin{align*}
\Spc S_{\V p}
=\{\varphi\in\Spc S(\R):\
\hat\varphi^{(k)}(\omega_n)=0,\ 0\le k\le m_n-1,\ \alpha_n=\jj\omega_n\},
\end{align*}
because
$\hat\varphi^{(k)}(\omega_n)=(-\jj)^k\langle x^k\ee^{-\jj\omega_nx},\varphi\rangle$.

\medskip
\emph{Verification via the Green inverse}:
The symbol $P(\jj\omega)$ is a polynomial, so that it vanishes only on the finite set
$\{\omega_n\}$ of real zeros, each of order $m_n$; Conditions 1 and 2 of
Theorem~\ref{Thm:PracticalChecklist} follow. The dual energy is
\begin{align*}
Q(\varphi)
=\frac{1}{2\pi}\int_{\R}\frac{|\hat\varphi(\omega)|^2}{|P(\jj\omega)|^2}\,\dint\omega,
\end{align*}
which always converges at infinity, because
$|P(\jj\omega)|\sim|a_{N_{\mathrm{tot}}}|\,|\omega|^{N_{\mathrm{tot}}}$. Near a
real zero, $|P(\jj\omega)|^2\asymp|\omega-\omega_n|^{2m_n}$: if $\hat\varphi$ vanishes at
$\omega_n$ to order $m_n$, then the integrand is bounded there; if the leading Taylor term of
$\hat\varphi$ at $\omega_n$ has order $j\le m_n-1$, then the radial exponent is
$2j-2m_n\le-2$ and the integral diverges. Hence $\{Q<\infty\}=\Spc S_{\V p}$; we note that
the borderline exponent $-1$ is never attained, because $2j-2m_n$ is even.

This is precisely the scenario of Corollary~\ref{Cor:GreenNormal}. The Green's function
$\rho_\Op B$ of the normal operator, with symbol $|P(\jj\omega)|^{-2}$, is the classical
exponential-spline kernel (for $P(X)=X$, one recovers $\rho_\Op B(x)=-\tfrac12|x|$), and
$\Op A=\rho_\Op B\ast\cdot$ is a two-sided Green inverse of $\Op B$. Item (iii) of the corollary
holds by the computation above, so that $(\Op A,\Op B)$ is a $\V p$-conditional Riesz pair
with $\Spc N_{\Op B}=\Spc N_{\V p}$: the conditioning by the tempered kernel of $P(\Op D)$ is
attainable and minimal. Condition 4 of Theorem~\ref{Thm:PracticalChecklist} then comes for
free, by the necessity part of Theorem~\ref{Thm:NormalPair} applied to the pair just
obtained, and Corollary~\ref{Cor:L2RealizationConditional} provides the $L_2$-realizations.

\medskip
\emph{Conditional inverse and native space}:
We set $\Op A_{\V \phi}=\Op P_{\V \phi}^\ast\Op A\Op P_{\V \phi}$, where the biorthogonal family exists by
Lemma~\ref{Lem:Biorthogonal} and can be formed of Gabor--Hermite atoms
$x^k\ee^{\jj\omega_nx}G(x)$, with $G$ a normalized Gaussian window. All the conclusions of
Theorems~\ref{Thm:NormalPair} and~\ref{Thm:PracticalChecklist} and of
Proposition~\ref{Prop:Factorization} hold: the homogeneous space
$\dot{\Spc H}_{\Op B}=\{f\in\Spc S'(\R):P(\Op D)f\in L_2(\R)\}/\Spc N_{\V p}$ carries the
exponential-spline energy $\|P(\Op D)f\|_{L_2}^2$, with concrete realization
$\Spc H'_{\Op A}$, while the native space itself is the augmented space
$\Spc H_{\Op L}=\Spc H'_{\Op A}\oplus\Spc N_{\V p}$ of \eqref{Eq:NativeSpace}; and
$\Spc H_{\Op A}\cong\dot{\Spc H}_{\Op B}\cong L_2(\R)$ canonically,
with $\Op A_{\V \phi}=\Op L^{-1}_{\V \phi}\Op L^{-1\ast}_{\V \phi}$.

\medskip
\emph{Special cases}:
For $P(X)=X^m$, one recovers the polynomial-spline setting with
$\Spc N_{\V p}=\Pi_{m-1}$, in agreement with Example~\ref{Ex:FracLaplacian} for
$(\gamma,d)=(m,1)$. For $P(X)=X^2+\omega_0^2$, the conditioning is trigonometric:
$\Spc N_{\V p}=\mathrm{span}\{\cos(\omega_0x),$ $\sin(\omega_0x)\}$ with $N_0=2$. If $P$ has no
imaginary root (for instance $P(X)=X-a$ with $a\in\R\backslash\{0\}$), then $N_0=0$: the pair
is an unconditioned compatible Riesz pair, and
$\dot{\Spc H}_{\Op B}=\{f:P(\Op D)f\in L_2\}$ is a Sobolev-type space with a genuine norm.

\medskip
\emph{Two structural observations}:
First, only the imaginary-axis roots condition the pair: this is the one-dimensional
counterpart of the harmonic-polynomial discussion of Example~\ref{Ex:FracLaplacian}, with one
notable difference. Here, the attainable null space coincides with the \emph{entire} tempered
kernel, because each real zero of the symbol has order $m_n\ge1>d/2=\tfrac12$, so that all
the degrees $0\le k\le m_n-1$ pass the threshold; the gap between tempered and attainable
kernel elements is a genuinely multidimensional phenomenon. Second, the example extends
Proposition~\ref{Prop:MultiplierConstruction} to symbols whose zeros lie away from the
origin: each real zero $\omega_0$ of order $m$ contributes the modulated block
$\mathrm{span}\{x^k\ee^{\jj\omega_0x}:0\le k\le m-1\}$ to $\Spc N_{\V p}$, in place of the
polynomial block attached to a zero at the origin.
\end{example}

\appendix

\section*{Appendix: Proof of Theorem~\ref{Theo:HilbertPositiveDef}}

Throughout the appendix, all Hilbert spaces are real (the complex case is analogous), and
$\iota:\Spc S\to\Spc S'$ denotes the canonical embedding. The standing hypothesis of
Theorem~\ref{Theo:HilbertPositiveDef} is that $\Spc S$ is a dense subspace of $\Spc H$. Whenever
$\Spc H$ is continuously embedded in $\Spc S'$, we require that embedding to extend $\iota$: for
$\Spc H$ as in Item 1, we denote the two embeddings by $\jmath:\Spc S\to\Spc H$ and
$\kappa:\Spc H\to\Spc S'$, with $\kappa\circ\jmath=\iota$, and the Riesz map by
$\Op J_{\Spc H}:\Spc H\to\Spc H'$.
The pairing $\langle\varphi,\psi\rangle=\langle\iota\varphi,\psi\rangle$ is symmetric and
non-degenerate on $\Spc S\times\Spc S$. Since $\Spc S$ is a Montel space, it is barrelled and
reflexive, while the injection $\iota(\Spc S)$ is weak-$\ast$ dense in $\Spc S'$ by hypothesis.

\begin{proof}
We prove $1\Rightarrow4$, $4\Rightarrow1\wedge2$ (which contains the converse part),
$2\Rightarrow1$, and $1\Leftrightarrow3$; we then establish the compatibility of the pairings
and the uniqueness of the pair.

\medskip
\emph{Step 1 ($1\Rightarrow4$)}:
For $\varphi\in\Spc S$, the functional
$\kappa^\ast\varphi:u\mapsto\langle\kappa u,\varphi\rangle$ is continuous on $\Spc H$, so that
$\kappa^\ast\varphi\in\Spc H'$. As $\kappa^\ast\varphi\in\Spc H'$ for every $\varphi$, the family
$\{\langle\kappa u,\cdot\rangle:\|u\|_{\Spc H}\le1\}$ is pointwise bounded on $\Spc S$; since
$\Spc S$ is barrelled, the Banach--Steinhaus theorem \cite[Ch.~III]{Schaefer-Wolff1999} makes it
equicontinuous, so that
$\varphi\mapsto\|\kappa^\ast\varphi\|_{\Spc H'}$ is a continuous seminorm on $\Spc S$.
Moreover, $\kappa^\ast(\Spc S)$ is dense in $\Spc H'$, because its annihilator in $\Spc H$ is
$\ker(\kappa)=\{0\}$, and $\kappa^\ast$ is the realization of $\Spc S$ in
$\Spc H'$, so that $\|\kappa^\ast\varphi\|_{\Spc H'}=\|\varphi\|_{\Spc H'}$.

We now set $\Op A\eqdef\jmath^\ast\Op J_{\Spc H}\jmath$ and
$\Op A^{-1}\eqdef\kappa\,\Op J_{\Spc H}^{-1}\kappa^\ast$, where
$\jmath^\ast:\Spc H'\to\Spc S'$ is the adjoint of $\jmath$. Both are continuous
$\Spc S\to\Spc S'$, and, for $\varphi,\psi\in\Spc S$,
\begin{align*}
\langle\Op A\varphi,\psi\rangle
=\langle\jmath\varphi,\jmath\psi\rangle_{\Spc H},
\qquad
\langle\Op A^{-1}\varphi,\psi\rangle
=\langle\kappa^\ast\varphi,\kappa^\ast\psi\rangle_{\Spc H'}.
\end{align*}
Hence both operators are (symmetric and) positive-definite, with
$|\varphi|_{\Op A}=\|\varphi\|_{\Spc H}$ and
$|\varphi|_{\Op A^{-1}}=\|\varphi\|_{\Spc H'}$, which is the ``moreover'' identity of the
theorem. Under the identifications $\Spc S\subset\Spc H$, $\Spc S\subset\Spc H'$,
$\Spc H\subset\Spc S'$, and $\Spc H'\subset\Spc S'$, they read
$\Op A=\Op J_{\Spc H}|_{\Spc S}$ and $\Op A^{-1}=\Op J_{\Spc H}^{-1}|_{\Spc S}$. Since
$\Spc S$ is dense in $\Spc H$ and in $\Spc H'$ with matching norms, the isometric extensions
of $\Op A$ and $\Op A^{-1}$ are $\Op J_{\Spc H}$ and $\Op J_{\Spc H}^{-1}$, which are the
inverse of each other; the pair $(\Op A,\Op A^{-1})$ is therefore a compatible Riesz pair,
which is Item 4.

Finally, the identity
$\langle\varphi,\psi\rangle
=\langle\kappa\jmath\psi,\varphi\rangle
=\langle\kappa^\ast\varphi,\jmath\psi\rangle_{\Spc H'\times\Spc H}$
yields, by the Cauchy--Schwarz inequality and by density, the relations
\begin{align}
\label{Eq:CompatBound}
|\langle\varphi,\psi\rangle|
\le|\varphi|_{\Op A^{-1}}\,|\psi|_{\Op A},
\qquad\varphi,\psi\in\Spc S,
\end{align}
\begin{align}
\label{Eq:CompatSharp}
|\psi|_{\Op A}
=\sup_{\varphi\neq0}
\frac{|\langle\varphi,\psi\rangle|}{|\varphi|_{\Op A^{-1}}},
\qquad
|\varphi|_{\Op A^{-1}}
=\sup_{\psi\neq0}
\frac{|\langle\varphi,\psi\rangle|}{|\psi|_{\Op A}}.
\end{align}
The relations \eqref{Eq:CompatBound}--\eqref{Eq:CompatSharp} express, at the level of
$\Spc S$, that the extensions of $\Op A$ and $\Op A^{-1}$ are mutually inverse Riesz maps;
they constitute the operational content of Definition~\ref{Def:CompatibleRieszPair}.

\medskip
\emph{Step 2 ($4\Rightarrow1\wedge2$, and converse part)}:
Let $(\Op A,\Op A^{-1})$ be a compatible Riesz pair, so that
\eqref{Eq:CompatBound}--\eqref{Eq:CompatSharp} hold. We form the completions
$\Spc H=\overline{(\Spc S,|\cdot|_{\Op A})}$ and
$\Spc H'=\overline{(\Spc S,|\cdot|_{\Op A^{-1}})}$ with $\Spc S\embedD \Spc H$ and $\Spc S\embedD \Spc H'$ by construction.

We now write $\Spc H^\ast$ for the topological dual of $\Spc H$. Each $\psi\in\Spc S$ defines the
functional $\ell_\psi=\langle\cdot,\psi\rangle$, which is $|\cdot|_{\Op A}$-continuous by
\eqref{Eq:CompatBound}, with $\|\ell_\psi\|_{\Spc H^\ast}=|\psi|_{\Op A^{-1}}$ by
\eqref{Eq:CompatSharp}. The map $\Lambda:\psi\mapsto\ell_\psi$ thus extends to a linear isometry
$\Lambda:\Spc H'\to\Spc H^\ast$, whose image is closed because $\Spc H'$ is complete. To see that
$\Lambda$ is onto, we use the property that $\{\ell_\psi\}$ is norming for $\Spc H$ (first identity of
\eqref{Eq:CompatSharp}): if $\nu=\langle\cdot,u\rangle_{\Spc H}\in\Spc H^\ast$ is orthogonal to every
$\ell_\psi$, then $\ell_\psi(u)=\langle u,\psi\rangle=0$ for all $\psi\in\Spc S$, so that $u=0$ and
hence $\nu=0$. The image $\Lambda(\Spc H')$ is therefore dense as well as closed, hence all of
$\Spc H^\ast$; consequently $\Spc H'=\Spc H^\ast$, and the pairing
$\langle f,g\rangle_{\Spc H'\times\Spc H}$ extends the $\Spc S$--$\Spc S'$ pairing.

It remains to realize the spaces in $\Spc S'$. The inclusion $\jmath:\Spc S\to\Spc H$ is continuous
with dense range, so its transpose $\jmath^\ast:\Spc H^\ast\to\Spc S'$ is injective and
weak-$\ast$ continuous, and it extends $\iota$. Through $\Spc H'=\Spc H^\ast$, this realizes
$\Spc H'$ as a subspace of $\Spc S'$; that is, $\Spc S\embedD\Spc H'\embedC\Spc S'$. By the symmetry
of \eqref{Eq:CompatBound}--\eqref{Eq:CompatSharp} in $(\Op A,\Op A^{-1})$, the same argument gives
$\Spc S\embedD\Spc H\embedC\Spc S'$. Thus $\Spc H$ and $\Spc H'$ are Hilbert--Gelfand spaces (Items
1 and 2), mutually dual under the extended pairing.

Finally, for $\varphi,\psi\in\Spc S$,
\begin{align*}
\langle\Op J_{\Spc H}\varphi,\psi\rangle_{\Spc H'\times\Spc H}
=\langle\varphi,\psi\rangle_{\Spc H}
=\langle\Op A\varphi,\psi\rangle,
\end{align*}
so that the realization of $\Op J_{\Spc H}\varphi$ in $\Spc S'$ is $\Op A\varphi$; by density,
$\Op J_{\Spc H}$ is the unique continuous extension of $\Op A$, and $\Op J_{\Spc H}^{-1}$ that of
$\Op A^{-1}$. This proves the converse part. If, moreover, $\Spc H$ is prescribed as in Item 4,
then its norm agrees with $|\cdot|_{\Op A}$ on the dense subspace $\Spc S$, so that the identity on
$\Spc S$ extends to an isometric isomorphism onto the realization above, through which $\Spc H$
inherits the embeddings; this proves $4\Rightarrow1$.

\medskip
\emph{Step 3 ($2\Rightarrow1$)}:
Let us assume that $\Spc H'$ is a Hilbert--Gelfand space. Step 1, applied to $\Spc H'$,
provides a compatible Riesz pair $(\Op A_1,\Op A_1^{-1})$ with
$|\varphi|_{\Op A_1}=\|\varphi\|_{\Spc H'}$ and
$|\varphi|_{\Op A_1^{-1}}=\|\varphi\|_{(\Spc H')'}=\|\varphi\|_{\Spc H}$, where the last
equality uses the reflexivity of $\Spc H$. The swapped pair $(\Op A_1^{-1},\Op A_1)$ is again
a compatible Riesz pair, so that Step 2 realizes
$\overline{(\Spc S,\|\cdot\|_{\Spc H})}=\Spc H$ as a Hilbert--Gelfand space, which is Item 1.

\medskip
\emph{Step 4 ($1\Leftrightarrow3$)}:
The implication $1\Rightarrow3$ is immediate. Conversely, the inclusion $\jmath$ has a closed
graph: if $\varphi_n\to\varphi$ in $\Spc S$ and $\jmath\varphi_n\to u$ in $\Spc H$, then
$\iota\varphi_n=\kappa\jmath\varphi_n$ converges to both $\iota\varphi$ and $\kappa u$, which
implies that $u=\jmath\varphi$ by the injectivity of $\kappa$. Since $\Spc S$ is barrelled and
$\Spc H$ is a Banach space, the closed-graph theorem (in Pt\'ak's form, for barrelled domains
and $B$-complete ranges) \cite{Ptak1958} (see also \cite[Ch.~IV]{Schaefer-Wolff1999}) yields the
continuity of $\jmath$; its range is dense by the
standing hypothesis of Theorem~\ref{Theo:HilbertPositiveDef}. The range of $\kappa$ contains
$\iota(\Spc S)$ and is dense in $\Spc S'$. Hence Item 1 holds.

\medskip
\emph{Step 5 (Compatibility of the pairings)}:
For $f\in\Spc H'$, realized in $\Spc S'$ as $\jmath^\ast f$, and for $\varphi\in\Spc S$, the
definition of the adjoint gives
\begin{align*}
\langle\jmath^\ast f,\varphi\rangle
=\langle f,\jmath\varphi\rangle_{\Spc H'\times\Spc H},
\end{align*}
which is \eqref{Eq:DualityProdH}. The extension to $(f,g)\in\Spc H'\times\Spc H$ and the bound
\eqref{Eq:HDualityBound} follow from the Cauchy--Schwarz inequality for the dual pairing of
$\Spc H$.

\medskip
\emph{Step 6 (Uniqueness of the pair)}:
Let $(\Op A_2,\Op A_2^{-1})$ be another compatible Riesz pair with
$|\cdot|_{\Op A_2}=\|\cdot\|_{\Spc H}$ on $\Spc S$. By Step 2, it generates a
Hilbert--Gelfand space whose norm agrees with that of $\Spc H$ on the dense subspace
$\Spc S$; the identity map on $\Spc S$ extends to an isometric isomorphism between the two
spaces and, since both embeddings into $\Spc S'$ are continuous and the weak-$\ast$ topology
is Hausdorff, the two spaces coincide as subsets of $\Spc S'$ with equal norms. In particular,
they share the same Riesz map, and Step 2 identifies the realization of
$\Op J_{\Spc H}\varphi$ in $\Spc S'$ both as $\Op A\varphi$ and as
$\Op A_2\varphi$; hence $\Op A_2=\Op A$, the symmetry of both operators being automatic by
Remark~\ref{Rem:AutomaticSymmetry}. The same argument, applied to
$\Spc H'$, gives $(\Op A_2)^{-1}=\Op A^{-1}$. This completes the proof.
\end{proof}

\subsection*{Acknowledgment}
The author used Claude (Anthropic) as an interactive assistant during the preparation of
this paper, for proofreading, for checking and simplifying proofs. All results were verified by the author,
who takes full responsibility for the content. This research was supported by the Swiss National Science Foundation (SNSF) under Grant 200020\_219356.

\bibliographystyle{unsrtinitials}
\bibliography{SplineOps}

\end{document}